\documentclass{amsart}
\usepackage{amssymb, graphics, color}
\usepackage[all]{xy}

 \usepackage[applemac]{inputenc}
\usepackage[cyr]{aeguill}

\usepackage[margin = 1in]{geometry}

\newtheorem{lemma}{Lemma}[section]
\newtheorem{proposition}[lemma]{Proposition}
\newtheorem{corollary}[lemma]{Corollary}
\newtheorem{theorem}[lemma]{Theorem}

\newtheorem{definition}[lemma]{Definition}
\newtheorem{remark}[lemma]{Remark}

\newtheorem*{Acknowledgement}{Acknowledgements}

\newtheorem{Theorem}{Theorem}

\newcommand\cf{cf\@. }
\newcommand\ie{i\@.e\@. }

\newcommand\pa{ \partial}

\newcommand\bbB{\mathbb B}
\newcommand\bbC{\mathbb C}

\newcommand\bbN{\mathbb N}
\newcommand\bbP{\mathbb P}

\newcommand\bbR{\mathbb R}

\newcommand\bbZ{\mathbb Z}

\renewcommand\Re{\operatorname{Re}}
\renewcommand\Im{\operatorname{Im}}

\usepackage{color}

\newcommand\bU{\overline{\mathcal{U}}}

\newcommand\hF{\widehat{F}}

\newcommand\bD{\overline{D}}

\newcommand\CI{\mathcal{C}^{\infty}}

\newcommand\Diff{\operatorname{Diff}}
\newcommand\Ric{\operatorname{Ric}}

\newcommand\ord{\operatorname{ord}}
\newcommand\cC{\mathcal{C}}
\newcommand\cA{\mathcal{A}}
\newcommand\cE{\mathcal{E}}

\newcommand\cF{\mathcal{F}}
\newcommand\db{\overline{\pa}}

\newcommand\cV{\mathcal{V}}
\newcommand\cU{\mathcal{U}}
\newcommand\cI{\mathcal{I}}

\newcommand\End{\operatorname{End}}

\newcommand\pr{\operatorname{pr}}

\newcommand\phg{\operatorname{phg}}
\newcommand\Spec{\operatorname{Spec}}
\newcommand\Id{\operatorname{Id}}

\newcommand\bM{\overline{M}}
\newcommand\bN{\overline{N}}

\renewcommand\sc{\operatorname{sc}}
\newcommand\CY{\operatorname{CY}}
\newcommand\WH{\operatorname{WH}}
\newcommand\lb{\operatorname{lb}}

\newcommand\cH{\mathcal{H}}
\newcommand\cZ{\mathcal{Z}}
\newcommand\tM{\widetilde{M}}
\newcommand\tN{\widetilde{N}}

\newcommand\Cl{\operatorname{Cl}}
\newcommand\bbA{\mathbb{A}}
\newcommand\Ch{\operatorname{Ch}}
\newcommand\ind{\operatorname{ind}}
\newcommand\hA{\widehat{A}}
\newcommand\STr{\operatorname{STr}}
\newcommand\cN{\mathcal{N}}
\newcommand\cM{\mathcal{M}}
\newcommand\ran{\operatorname{ran}}
\newcommand\Str{\operatorname{Str}}
\newcommand\Tr{\operatorname{Tr}}
\newcommand\Q{\mathcal{Q}}
\newcommand\Td{\operatorname{Td}}
\newcommand\bZ{\overline{Z}}
\newcommand\bell{\overline{\ell}}
\newcommand\rel{\operatorname{rel}}
\newcommand\WP{\operatorname{WP}}
\newcommand\Vol{\operatorname{Vol}}
\newcommand\FP{\operatorname{FP}}
\newcommand\AC{\operatorname{ACyl}}

\newcommand\SU{\operatorname{SU}}
\newcommand\tr{\operatorname{tr}}
\newcommand\diver{\operatorname{div}}
\newcommand{\fib}{\mathrm{fib}}

\begin{document}
\title[asymptotically cylindrical Calabi-Yau moduli]
{The moduli space of asymptotically cylindrical \\ Calabi-Yau manifolds}

\author{Ronan J. Conlon}
\address{Département de Mathématiques, Universit\'e du Qu\'ebec \`a Montr\'eal}
\email{rconlon@cirget.ca}
\author{Rafe Mazzeo}
\address{Department of Mathematics, Stanford University}
\email{mazzeo@math.stanford.edu}

\author{Fr\'ed\'eric Rochon}
\address{Département de Mathématiques, Universit\'e du Qu\'ebec \`a Montr\'eal}
\email{rochon.frederic@uqam.ca}

\begin{abstract}
We prove that the deformation theory of compactifiable asymptotically cylindrical Calabi-Yau manifolds is unobstructed.
This relies on a detailed study of the Dolbeault-Hodge theory and its description in terms of the cohomology of
the compactification. We also show that these Calabi-Yau metrics admit a polyhomogeneous expansion at infinity,
a result that we extend to asymptotically conical Calabi-Yau metrics as well.  We then study the moduli space of
Calabi-Yau deformations that fix the complex structure at infinity. There is a Weil-Petersson metric on this space
which we show is Kähler.  By proving a local families $L^2$-index theorem, we exhibit its Kähler form as a multiple
of the curvature of a certain determinant line bundle.
\end{abstract}
\maketitle

\tableofcontents

\section*{Introduction}

A complete Riemannian manifold $(M,g)$ of dimension $2n$ is said to be Calabi-Yau if its holonomy group is contained in $\SU(n)$, in
which case $(M,g)$ is Ricci-flat and Kähler.  Conversely, if $(M,g)$ is Ricci-flat Kähler, then its reduced holonomy group is contained in
$\SU(n)$, hence $(M,g)$ is Calabi-Yau if $M$ is simply connected.  The principal source of examples of Calabi-Yau manifolds is the
famous Calabi conjecture proved by Yau \cite{Yau1978}:  a compact Kähler manifold with trivial canonical line bundle admits a
unique Calabi-Yau metric in each Kähler class. Subsequent work by Tian \cite{Tian1987} and Todorov \cite{Todorov1989} shows
that the moduli space of (polarized simply connected) compact Calabi-Yau manifolds has at most quotient singularities, and moreover,
its natural Weil-Petersson metric is Kähler. This moduli space is central in the study of mirror symmetry, and is thus of importance in
mathematical physics, algebraic geometry, differential geometry and number theory.

Fundamental results of Tian-Yau \cite{Tian-Yau1990, Tian-Yau1991} and Joyce \cite{Joyce} imply the existence of many non-compact,
complete, quasi-projective Calabi-Yau manifolds.  In the present paper, we study the moduli space of compactifiable asymptotically
cylindrical Calabi-Yau manifolds. The only previous generalization of the Tian-Todorov theorem (to any complete quasi-projective
setting) is for the same class of asymptotically cylindrical metrics, but only in complex dimension $2$, by
Hein \cite[Corollary~4.3]{Hein2012}. We mention also the formal deformation theory in the same setting, but in general
dimensions, in \cite[\S4.3.3]{KKP2008}. Recall that a complete Riemannian manifold $(M,g)$ is \textbf{asymptotically cylindrical} if there exist a compact set $K\subset M$, a closed Riemannian manifold $(N,h)$ and a diffeomorphism $\Phi: M\setminus K\to N\times (0,\infty)$ such that for some $\delta>0$, $| \nabla^k(\Phi_*g-g_{\infty})|= \mathcal{O}(e^{-\delta t})$ for all $k\in\bbN_0$, where $g_{\infty}= dt^2 +h$ is a product metric.  By the Cheeger-Gromoll splitting theorem, see also \cite{Salur2006}, a connected, complete manifold with nonnegative
Ricci curvature can have at most one end unless it splits as a global Riemannian product $\mathbb R \times N$, so we may as well
assume that $(M,g)$ has a single cylindrical end.  The recent improvements by Haskins-Hein-Nordström \cite{HHN2012} of the Tian-Yau
construction \cite{Tian-Yau1990} give many new examples of asymptotically cylindrical Calabi-Yau spaces. Indeed, let $\bM$ be a
compact Kähler orbifold of complex dimension $n\ge 2$.  Let $\bD\in |-K_{\bM}|$ be an effective orbifold divisor satisfying the
following two conditions:
\begin{itemize}
\item[(i)] The complement $M:= \bM\setminus \bD$ is a smooth manifold;
\item[(ii)] The orbifold normal bundle of $\bD$ is biholomorphic to $(\bbC\times D)/ \langle \iota\rangle$ as an orbifold line bundle, where $D$ is a connected complex manifold and $\iota$ is a complex automorphism of $D$ of order $m<\infty$ acting on the product via $\iota(w,x)= (e^{\frac{2\pi i}{m}}w,\iota(x))$.
\end{itemize}
Then if $\Omega$ is a meromorphic $n$-form on $\bM$ with a simple pole along $\bD$, the construction of \cite{Tian-Yau1990} and \cite{HHN2012} ensures that for every Kähler class $\mathfrak{t}$ on $\bM$, there exists an asymptotically cylindrical Calabi-Yau metric $g_{\CY}$ on $M$ with Kähler form $\omega_{\CY}$ such that $\omega_{\CY}\in \left. \mathfrak{t}\right|_{M}$ and $\omega_{\CY}^n= i^{n^2}\Omega\wedge \overline{\Omega}$.   We say that a Calabi-Yau manifold $(M,g_{\CY})$ obtained in this way is a \textbf{compactifiable asymptotically cylindrical Calabi-Yau manifold} with compactification $\bM$.

The existence result of Haskins-Hein-Nordström was used in \cite{CHNP2013} to obtain many new examples of asymptotically cylindrical Calabi-Yau $3$-folds.  Those play a distinguished role because they can be used as building blocks in Kovalev's twisted connected sum construction of compact manifold with holonomy $G_2$, see \cite{Kovalev2003,Kovalev-Lee,CHNP}.    
Haskins-Hein-Nordström also prove a uniqueness result, see Theorem~\ref{un.3} below for the formulation that will be used here.  More surprisingly, they establish a converse by recovering the compactification $\bM$ in many important cases; namely if $(M,g)$ is a simply-connected, irreducible asymptotically cylindrical Calabi-Yau manifold of complex dimension $n>2$, then $(M,g)$ arises from their construction.  

In the present paper, we shall study compactifiable asymptotically cylindrical Calabi-Yau manifolds and their moduli spaces.  After some
preliminaries on $b$-metrics
and the $b$-calculus of Melrose \cite{MelroseAPS}, we begin our investigation by determining the space of  $L^2$-harmonic forms of type $(p,q)$ on such a manifold.  As shown in \cite{MelroseAPS}, see also \cite{HHM2004}, the space of (de Rham) $L^2$-harmonic forms of an asymptotically cylindrical manifold is identified in terms of the (de Rham) cohomology of an associated manifold with boundary.  However,
to respect the $(p,q)$ decomposition, it turns out to be more natural here to relate this $(p,q)$ Hodge cohomology with the Dolbeault
cohomology of the compactification $\bM$. More precisely, if $E\to \bM$ is a holomorphic vector bundle over $\bM$,  then
Theorem~\ref{ht.5} below is the following assertion:
\begin{Theorem}
$
\hspace{0.5cm} L^2\cH^{p,q}(M;E) \cong  \Im \{ H^{q}(\bM;\Omega^p(\log \bD)\otimes E(-\bD))\to H^{q}(\bM; \Omega^{p}(\log\bD) \otimes E) \}.
$
\label{int.1}\end{Theorem}
See \S~\ref{ht.0} for notation.

The proof uses a sheaf theoretic argument, along with some key facts about elliptic $b$-operators which lead to the characterization
of weighted Dolbeault $L^2$ cohomology
\begin{equation}
 \WH^{p,q}(g_b,\epsilon,M;E)\cong H^{q}(\bM ,\Omega^p(\log \bD) \otimes E(-\bD)), \quad  \WH^{p,q}(g_b,-\epsilon,M;E)\cong H^{q}(\bM;\Omega^{p}(\log\bD)\otimes E)
\label{int.2}\end{equation}
when $\epsilon>0$ is small enough, see Theorem~\ref{ht.3} for details.   Further analysis yields, in Theorem~\ref{bi.6}, a
$\pa\db$-lemma adapted to this setting, and the existence of canonical harmonic representatives for classes in
$\WH^{p,q}(g_b, -\epsilon, M; E)$, which is necessary for the deformation theory.  These Hodge theoretic results do not require the
full regularity of the metric assumed here,
and also do not require that $g_b$ be Calabi-Yau.   We continue to  assume, however, that $g_b$ is a polyhomogeneous exact $b$-metric,
i.e., $g_b$ admits a complete asymptotic expansion in the cylindrical end in powers of $\rho=e^{-t}$, see \S~\ref{b.0} for details, and
in fact, our next result shows that our Calabi-Yau spaces possess this sharp regularity:
\begin{Theorem}
Compactifiable asymptotically cylindrical Calabi-Yau metrics are polyhomogeneous exact $b$-metrics.
\label{int.3}\end{Theorem}
This is the content of Theorem~\ref{accy.1} and Corollary~\ref{accy.5} below.   The paper \cite{Santoro} already contains some results
in this direction, but what we prove here is more precise.

Similar regularity results for K\"ahler-Einstein metrics trace back to the work of Lee-Melrose \cite{Lee-Melrose}, where the
polyhomogeneity of the Cheng-Yau metric on a strictly pseudoconvex domain is established: we refer also to
\cite{Jeffres-Mazzeo-Rubinstein, Rochon-Zhang}, which prove polyhomogeneity of other types of Kähler-Einstein metrics.
All of these results are proved by using a linear regularity theorem (Corollary~\ref{breg.9} below in our case)
in an inductive bootstrapping argument for the complex Monge-Ampère equation.

Using that asymptotically conical metrics are conformal to asymptotically cylindrical metrics, we can deduce from the proof here
a similar polyhomogeneity result for asymptotically conical Calabi-Yau metrics, as constructed in \cite{Tian-Yau1991, CH2014};
this is carried out in Corollary~\ref{acocy.11}.

This sharp regularity of asymptotically cylindrical Calabi-Yau metrics becomes extremely useful when studying the deformation
theory of these metrics and for understanding the Weil-Petersson geometry of the corresponding moduli space.
Our approach to the deformation theory follows Kawamata \cite{Kawamata1978}, who studied deformations of
compactifiable complex manifolds.  Infinitesimal deformations of the complex structure are described by the cohomology group
$H^1(\bM; T_{\bM}(\log\bD))$, where $T_{\bM}(\log\bD)$ is the sheaf of holomorphic vector fields on $\bM$ tangent to $\bD$.
By our study of the Dolbeault-Hodge theory, these infinitesimal deformations admit canonical harmonic representatives. Using the Tian-Todorov theorem
as well as the $\pa\db$-lemma in Lemma~\ref{dd.1}, we recover the result of \cite{KKP2008} that the
deformation theory is formally unobstructed in this setting.   There are important simplifications in complex dimension $2$
which follow from the vanishing of the Frölicher-Nijenhuis bracket of constant differential forms on a flat cylinder, see \cite{Hein2012}.
To obtain actual deformations, we must choose the terms in the formal series of the deformation systematically. This is done
using a parametrix for the Laplacian in the sense of \cite{MelroseAPS} and \cite{MazzeoEdge}.  Invoking some estimates for
this parametrix, we can then safely apply the standard argument of Kodaira-Spencer \cite[\S~5.3]{Kodaira} to obtain the following result.
\begin{Theorem}
The deformation theory of compactifiable asymptotically cylindrical Calabi-Yau manifolds is unobstructed.
\label{int.4}\end{Theorem}

Combining this with the work of Kovalev \cite{Kovalev2006}, which in turn generalizes results of Koiso \cite{Koiso1983}, we
see that any Ricci-flat asymptotically cylindrical metric sufficiently close to a compactifiable asymptotically cylindrical
Calabi-Yau metric $g$ is in fact Kähler for some nearby deformation of the complex structure of $g$.

Similar results about the deformation theory in some different (though closely related) settings which involve asymptotically
cylindrical geometries may be found in \cite{Joyce-Salur} and \cite{Salur-Todd}.

We next consider relative deformations, i.e., those which fix the complex structure at infinity.  The infinitesimal analogue of
this type of deformation is $\Im \{ H^{1}(\bM; T_{\bM}(\log\bD)(-\bD))\to H^{1}(\bM; T_{\bM}(\log\bD)) \}$, the space of
$L^2$-harmonic forms $L^2\cH^{0,1}(M; T_{\bM}(\log\bD))$ by Theorem~\ref{int.1}.  Fixing a polarization $\bM$
and assuming that $H^{1}(\bM;\bbR)=0$, we show how to systematically choose a Calabi-Yau metric $g_m$ for each point $m$
in the relative moduli space $\cM_{\rel}$.

Now define a Weil-Peterson metric on the moduli space by
\[
g_{\WP}(u,v)= \int_{M_m} \langle u, v\rangle_{g_m} d\mu(g_m),  \quad u,v\in T_m\cM_{\rel}\cong L^2\cH^{0,1}(M_m,g_m, T^{1,0}M_m),
\]
where $M_m$ is the deformation corresponding to the point $m$.  Using a suitable notion of renormalized volume, we show in
Proposition~\ref{wp.8} that this metric is Kähler with Kähler form $\omega_{\WP}$, a multiple of the first Chern class of
the vertical tangent bundle.

Just as in the compact setting, we show that $\dim L^2 \mathcal H^{p,q}$ is constant in $\cM_{\rel}$, which means that
 it is possible to define a determinant line bundle associated  to the family of $\db$ operators on $\cM_{\rel}$ with a
corresponding Quillen metric and Quillen connection.  Twisting by a suitable choice of holomorphic vector bundle $E$,
see \eqref{wp.10}, our final result generalizes \cite[5.30]{BCOV1994}, see also \cite{Fang-Lu2005}.

\begin{Theorem}
The curvature of the determinant line bundle of the family of Dolbeault operators $\sqrt(\db+\db^*)$ associated to the holomorphic vector bundle $E$ is
\[
          \frac{i}{2\pi} (\nabla^Q)^2= \frac{\chi(M)}{12\pi} \omega_{\WP}.
\]
\label{int.5}\end{Theorem}

The key step in proving this is to obtain a local families $L^2$-index theorem.  This is not quite the setting of the families index
theorem of Melrose-Piazza \cite{MP1997} since ours is not a Fredholm family of operators. In particular, the heat kernel does
not decay exponentially fast for large time.   There are special features which help our calculations. One is that the collection
of $L^2$ kernels and cokernels form bundles over the moduli space.  The other is that the indicial family, that is, the model operator at infinity, is the same for all members of the family.

Using these and the scattering theory of \cite{MelroseAPS}, we show in Proposition~\ref{lfi.3} that the heat kernel decays rapidly in positive degree. We can then apply the argument of \cite{MP1997} to
obtain the local families $L^2$-index formula, see Theorem~\ref{lfi.16}.  Because of the constancy of the indicial family,
our formula contains no eta form in positive degree, and only the standard `Atiyah-Singer' integrand appears.
Note that our formula also applies to certain families of signature operators. For the Dolbeault operator,
proceeding as in \cite{BGSIII} and regularizing as in \cite{MelroseAPS} to define the analytic torsion, we
obtain the formula in Theorem~\ref{qc.8} for the curvature of the Quillen determinant line bundle.

The paper is organized as follows. The initial sections \S~\ref{b.0} and \S~\ref{greg.0} review the notion of $b$-metrics
and recall some important properties of elliptic $b$-operators.  In \S~\ref{ht.0}, we then study the Hodge theory of polyhomogeneous Kähler $b$-metrics admitting some suitable compactification by a compact Kähler manifold.  We then prove in \S~\ref{accy.0} that the asymptotically cylindrical Calabi-Yau metrics of \cite{HHN2012} are polyhomogeneous exact $b$-metrics.  We also show that the asymptotically
conical Calabi-Yau metrics of \cite{CH2013,CH2014} are polyhomogeneous as well. These results are used in \S~\ref{tt.0} to show
that the deformation theory of asymptotically cylindrical Calabi-Yau manifolds is unobstructed.  In \S~\ref{lfi.0} and \S~\ref{qc.0},
we obtain a local families $L^2$-index theorem and a curvature formula of the associated Quillen determinant line bundle for families of Dolbeault operators parametrized by the relative moduli space of asymptotically cylindrical Calabi-Yau manifolds.  Finally, in \S~\ref{wp.0},
we define the Weil-Peterson metric on the relative moduli space and explore some of its properties.

\begin{Acknowledgement}
R.M. was partially supported by the NSF under DMS-1105050; F.R. was partially supported by NSERC, FRQNT and a
Canada research chair.
\end{Acknowledgement}

\numberwithin{equation}{section}

\section{Asymptotically cylindrical metrics} \label{b.0}
In this section we define various classes of asymptotically cylindrical metrics, defined through different decay and regularity
assumptions presented in the language of $b$-geometry. This point of view is the one most coherent with the
analytic methods used later in this paper.  For those unacquainted with this language, we refer principally to
the book of Melrose \cite{MelroseAPS}.  We first introduce the $b$-vector fields, which give structure to later
definitions. This leads to the introduction of function spaces which are used later, in particular, the spaces of
polyhomogeneous functions as well as the various classes of asymptotically cylindrical metrics, also called $b$-metrics.
Differences between these spaces are due to the precise asymptotic regularity at infinity we impose on them.

Suppose that $\tM$ is a compact manifold with boundary with $\dim \tM = n$, and let $\rho\in\CI(\tM)$ be a boundary defining
function, i.e., $\rho>0$ in the interior $M= \tM\setminus \pa \tM$,  $\rho=0$ on $\pa \tM$ and $d\rho$ is nowhere
zero on $\pa\tM$.  We define the \textbf{Lie algebra of $b$-vector fields} on $\tM$ by
$$
     \cV_b(\tM)= \{ \xi \in \CI(\tM;T\tM) \; | \; \xi \; \mbox{is tangent to} \; \pa\tM \}.
$$
In local coordinates $(\rho,y)$ near $\pa \tM$, a $b$-vector field $\xi$ takes the form
$$
      \xi= a\rho \frac{\pa}{\pa \rho}+ \sum_{i=1}^{n-1} a_i \frac{\pa}{\pa y^i} \quad \mbox{with} \quad a, a_1,\ldots, a_{n-1}\in \CI(\tM).
$$
As an alternate characterization, $\xi\in \CI(\tM; T\tM)$ is in $\cV_b(\tM)$ if and only if $\xi\rho\in \rho\CI(\tM)$
for any boundary defining function $\rho$.

Associated to $\cV_b(\tM)$ is the \textbf{$b$-tangent bundle}, ${}^{b}T\tM\to \tM$. This is a natural smooth vector bundle
with fibre over $p\in \tM$ given by
$$
     {}^{b} T_p\tM= \cV_b(\tM)/ (I_p\cV_b(\tM)),  \quad  I_p= \{f\in \CI(\tM) \; | \; f(p)=0\}.
$$
There is a canonical morphism $\iota_b: {}^{b}T\tM\to T\tM$ of vector bundles such that
$$
       (\iota_b)_*\CI(\tM;{}^{b}T\tM)= \cV_b(\tM)\subset \CI(\tM;T\tM).
$$
Note that $\iota_b$ is only an isomorphism when restricted to $\tM\setminus \pa \tM$.  The vector bundle ${}^{b}T\tM$ is
a Lie algebroid with anchor map given by $(\iota_b)_*$.

\begin{definition}
A \textbf{$b$-metric} is a complete Riemannian metric $g$ on $M=\tM\setminus \pa\tM$ which can be written as
$$
       g=(\iota^{-1}_b)^* (\left. g_b\right|_{\tM\setminus \pa \tM})
$$
for some positive definite section $g_b\in \CI(\tM; \mathrm{Sym}^2({}^{b}T^*\tM))$. 
\label{b.1}\end{definition}
\begin{remark}
It is convenient and innocuous to regard $g_b$ as the $b$-metric.
\end{remark}

In local coordinates $(\rho,y)$ near $\pa\tM$, a $b$-metric is of the form
\begin{equation}
  g= a\frac{d\rho^2}{\rho^2} + \sum a_i dy^i\odot \frac{d\rho}{\rho} + \sum a_{ij} dy^i\odot dy^j, \quad a,a_j,a_{ij}\in \CI(\tM).
\label{b.2}\end{equation}
The $b$-metrics with all $a_i \equiv 0$ are particularly interesting.
\begin{definition}
A $b$-metric $g$ is a \textbf{product $b$-metric} if there exists a collar neighborhood
$$
    c: \pa\tM\times [0,\epsilon)_{\rho}\to \tM
$$
of the boundary such that $c^*g= \lambda^2\frac{d\rho^2}{\rho^2} + g_{\pa \tM}$, where $\lambda$ is a positive constant and $g_{\pa \tM}$ is a Riemannian metric on $\pa\tM$.  A $b$-metric $g$ is \textbf{exact} if
$g-g_p\in \rho\CI(\tM; {}^bT\tM\otimes {}^bT\tM)$ for some product $b$-metric $g_p$.
\label{b.3}\end{definition}
In terms of the variable $t=-\lambda\log \rho$, a product $b$-metric has the form
$$
                dt^2+ g_{\pa\tM}, \quad t\in (-\lambda\log\epsilon,\infty)
$$
in a collar neighborhood of $\pa\tM$, i.e., is isometric to a half-cylinder outside a compact set, while exact $b$-metrics are
those which converge exponentially to product metrics. An alternate characterization is that $g$ is an exact $b$-metric
if each of the coefficients $a_i$ of the cross-terms in \eqref{b.2} vanish at $\pa\tM$. One useful feature of exact $b$-metrics,
see \cite[Proposition~2.37]{MelroseAPS}, is that their Levi-Civita connection $\nabla$ extends naturally to the boundary to give a
connection for the $b$-tangent bundle ${}^bT\tM$.

We now describe some useful function spaces in this setting. Fix a volume density $\nu_g$ associated to any exact $b$-metric
$g$; this gives the Hilbert spaces $L^2_b(M)$ and $L^2_b(M;E)$ of square integrable functions and of square integrable
sections of a vector bundle $E\to \tM$ with Hermitian metric. Using a connection $\nabla^E$ for $E$ and the Levi-Civita
connection of $g$, we define the $b$-Sobolev spaces
$$
  H^k_b(M;E)= \{ f\in L^2(M;E) \; | \; \nabla^{\ell}f\in L^2_b(M; ({}^bT^*\tM)^{\ell}\otimes E) \; \forall \ell=0,\ldots, k\},
$$
where ${}^bT^*\tM$ denotes the dual of ${}^bT\tM$.  Since the elements of $\cV_b(\tM)$ are simply the vector fields
on $M$ which extend smoothly to the boundary and which have uniformly bounded length with
respect to any fixed $b$-metric, these $b$-Sobolev spaces can also be defined by requiring that $u \in H^k_b(M)$
if $u$ and $V_1 \ldots V_\ell u$ lie in $L^2_b$ for any collection of $b$-vector fields $V_i$ and for every $\ell \leq k$.
From this it is clear that the space $H^k_b(M;E)$ is independent of choices, even though the inner product is not.
We shall also use weighted versions of these Sobolev spaces, namely
$$
    \rho^{\ell}H^k_b(M;E)= \{ \rho^{\ell} \sigma \; | \; \sigma\in H^k_b(M;E)\}.
$$

We next define the space of $k$-times differentiable sections of $E$ with derivatives uniformly bounded on $M$ (with
respect to $g$ and the metric on $E$):
$$
      \cC^k_b(M;E)= \{ \sigma\in \cC^k(M;E) \; | \; \sup_{p\in M} |\nabla^{\ell} \sigma(p)|_{g,g_E}<\infty \; \; \forall \ell=0,1,\dots, k\}.
$$
As before, $\cC^k_b(M;E)$ (but not its norm) is independent of choices.  Set
$$
    \CI_b(M;E) =\bigcap_{k=0}^{\infty} \cC^k_b(M;E) \quad \mbox{and}  \quad H^{\infty}_b(M;E)= \bigcap_{k=0}^{\infty} H^k_b(M;E).
$$
Note that
$$
       \CI(\tM;E)\subsetneq \CI_b(M;E), \quad \mbox{and}\qquad    H^{\infty}_b(M;E) \subsetneq \CI_b(M;E).
$$
The first inclusion is proper since $u \in \CI_b$ only requires the boundedness of all $b$-derivatives of $u$, but not
that they extend continuously to the boundary. Thus, for example, $\cos(\log\rho)$ lies in $\CI_b(M)$, but not in $\CI(\tM)$.
The latter inclusion follows from the Sobolev embedding theorem, and is proper since elements of $\CI_b(M;E)$ which
are bounded but do not decay are not square integrable.  The space $\CI_b(M;E)$ is often called the space of conormal
sections of order $0$ and denoted $\mathcal A^0(M; E)$.

It is certainly too restrictive to require that the metric coefficients $a,a_i, a_{ij}$ in \eqref{b.2} are smooth up to the boundary.
One way of generalizing this, which appears in \cite{HHN2012}, is as follows.
\begin{definition}
An \textbf{asymptotically cylindrical metric} on $M=\tM\setminus \pa \tM$ ($\AC$-metric for short) is a complete Riemannian metric
$g$ on $M$ such that there exists a $\delta > 0$ and a product $b$-metric $g_p$ on $M$ for which
$$
         g-g_p = \rho^{\delta}\CI_b(M;{}^bT^*\tM\otimes {}^bT^*\tM).
$$
\label{b.4}\end{definition}
Unfortunately, this class of metrics is now too general for our purposes, so for reasons which will become clear later, we consider a
class of metrics intermediate between $\AC$ and exact $b$-metrics, which are characterized as having an asymptotic
expansion at infinity.  To make this precise, we first recall the definition of polyhomogeneous expansions of functions
and sections of bundles over $\tM$; this, in turn, relies on the notion of index sets, so this is our starting point.

\begin{definition}
An \textbf{index set $F$} is a discrete subset of $\bbC\times \bbN_0$ such that
\begin{align*}
& i) \ \, (z_j,k_j)\in F, \; |(z_j,k_j)|\to \infty \; \Longrightarrow \; \Re z_j\to \infty, \\
& ii) \ (z,k)\in F \; \Longrightarrow \; (z+p,k)\in F \; \ \forall p\in \bbN, \\
& iii) \ (z,k)\in F \; \Longrightarrow \; (z,p)\in F \; \ \forall p=0,\ldots,k.
\end{align*}
The index set $F$ is called \textbf{positive} if
$$
      (z,k)\in F \; \Longrightarrow \; \Im z=0, \Re z> 0,
$$
and is \textbf{nonnegative} if 
\begin{gather*}
 (z,k)\in F \; \Longrightarrow \; \Im z=0, \Re z\ge 0, \\
 (0,k)\in F \; \Longrightarrow \; k=0.
\end{gather*}

Finally, if $F$ and $G$ are two index sets, then their extended union $F \, \overline{\cup} \, G$ consists of the union of these
two sets along with the pairs $(z, k + \ell + 1)$ where $(z,k) \in F$ and $(z,\ell) \in G$.
\label{b.5}\end{definition}

If $F \subset \bbR\times \bbN_0$, we define $\inf F$ to be the smallest element of $F$ with respect to the lexicographic order relation on
$\bbR\times \bbN_0$, i.e.,
$$
(z_1,k_1)< (z_2, k_2) \quad \Longleftrightarrow \quad  z_1<z_2 \quad \mbox{or} \quad  z_1=z_2 \; \mbox{and} \; k_1>k_2.
$$

\begin{definition}
Given an index set $F$, define the space $\cA^F_{\phg}(\tM)$ of {\bf polyhomogeneous} functions with index set $F$
to consist of all functions $f$ which have an asymptotic expansion at $\pa\tM$ of the form
\begin{equation}
    f\sim \sum_{(z,k)\in F} a_{(z,k)} \rho^z(\log\rho)^k, \quad a_{z,k}\in \CI(\tM).
\label{b.6}\end{equation}
The symbol $\sim$ means here that for all $N\in \bbN$,
$$
       f- \underset{\Re z\le N}{\sum_{(z,k)\in F}} a_{(z,k)}\rho^z(\log \rho)^{k} \in \rho^N\CI_b(M).
$$

If $F$ is a nonnegative index set (or one such that every $(z,k) \in F \setminus \{(0,0)\}$ has $\Re z > 0$),
then $\cA^F_{\phg} \subset \CI_b$.  We call polyhomogeneous functions with these types of index sets
{\bf bounded polyhomogeneous}.   More generally, if $(s,k) = \inf F$, then $\cA^F_{\phg} \subset \rho^{-s-\epsilon} \CI_b$
for every $\epsilon > 0$.
\end{definition}

The coefficients $a_{(z,k)}$ in the expansion \eqref{b.6} depend on the choice of boundary defining function $\rho$, but because
of condition ii) in the definition of index sets, the space $\cA^F_{\phg}(\tM)$ itself is independent of this choice.
There are two familiar examples of these spaces: first, $\cA^\emptyset_{\phg}(\tM)$ is the same as $\dot{\cC}^{\infty}(\tM)$, the
space of smooth functions on $\tM$ vanishing with all derivatives on $\pa \tM$; next, $\cA^F_{\phg}(\tM)$ with
$F=\bbN_0\times \{0\}$ is the same as $\CI(\tM)$.  The reason for introducing these spaces with more general index
sets is that solutions of natural elliptic operators associated to even just product $b$-metrics are polyhomogeneous
with index sets determined by spectral data on $\pa\tM$, hence are only rarely smooth up to the boundary.

The space $\cA^F_{\phg}(\tM)$ is a $\CI(\tM)$-module, and thus, for any vector bundle $E\to \tM$, we can
define the \textbf{space of polyhomogeneous sections of $E$ with index set $F$} by
$$
      \cA^{F}_{\phg}(\tM;E)= \cA^F_{\phg}(\tM)\otimes_{\CI(\tM)} \CI(\tM;E).
$$
\begin{definition}
A \textbf{polyhomogeneous} $\AC$-metric on $M$ is an asymptotically cylindrical metric $g$ on $M$ of the form
$$
        g= (\iota^{-1}_b)^* (\left.g_b\right|_{\pa\tM}),
$$
where $g_b\in \cA^{F}_{\phg}(\tM; \mathrm{Sym}^2({}^bT^*\tM))$ with $F$ a nonnegative index set.
\label{b.7}\end{definition}

\section{Elliptic  $b$-operators} \label{greg.0}
We next review some aspects of the theory of elliptic $b$-operators with particular emphasis on their mapping properties
on spaces of polyhomogeneous and conormal sections.

The space of \textbf{$b$-differential operators} on $\tM$, $\Diff^*_b(\tM)$, is the universal enveloping algebra of $\cV_b(\tM)$
over $\CI(\tM)$.  In other words, an element of $P \in \Diff^*_b(\tM)$ is generated by $\CI(\tM)$ and locally finite sums of products of
$b$-vector fields.  In local coordinates $(\rho,y)$ near $\pa \tM$,
\begin{equation}
 P= \sum_{\alpha+|\beta|\le k} a_{\alpha\beta} \left(\rho\frac{\pa}{\pa \rho}\right)^{\alpha} \left( \frac{\pa}{\pa y} \right)^{\beta},
\quad a_{\alpha \beta}\in \CI(\tM),
\label{b.8}\end{equation}
where $k$ is the order of $P$. Since $\Diff^k_b(\tM)$ is a $\CI(\tM)$-module, we can immediately define the space of
$b$-differential operators acting on sections of a vector bundle $E\to \tM$ by
$$
      \Diff^{k}_b(\tM;E)= \Diff^k_{b}(\tM) \otimes_{\CI(\tM)} \CI(\tM; \mathrm{End}(E)).
$$
A connection $\nabla$ on $({}^bT^*\tM)^{\ell} \otimes E$ is obtained from a connection $\nabla$ on $E$ and
the Levi-Civita connection of an exact $b$-metric $g$. Any $P\in \Diff^k_b(\tM;E)$ then takes the form
\begin{equation}
P= \sum_{\ell=0}^k a_{\ell} \cdot \nabla^{\ell}, \quad a_{\ell}\in \CI(\tM; ({}^bT\tM)^{\ell}\otimes \End(E)),
\label{b.9}\end{equation}
where ``$\cdot$" denotes contraction between the copies of ${}^bT\tM$ and ${}^bT^*\tM$.   Important examples of $b$-differential
operators include the geometric operators associated to $b$-metrics, e.g. the Laplacian or Dirac-type operators. If $g$
is a polyhomogeneous $\AC$-metric, these geometric operators are elements of
$$
     \Diff^*_{b,F}(\tM;E)= \cA^{F}_{\phg}(\tM)\otimes_{\CI(\tM)} \Diff^*_b(\tM;E),
$$
the polyhomogeneous $b$-differential operators with index set $F$.

\begin{definition}
The \textbf{principal symbol} of $P\in \Diff^k_b(\tM;E)$ is the map $\sigma_k(P): {}^bT^*\tM\to \End(E)$ which is homogeneous
of degree $k$ on the fibres and is given by
$$
    \sigma(P)(\xi)= i^ka_k( \underset{k \; \mbox{times}}{\underbrace{\xi,\ldots,\xi}}) \in \End(E),  \quad \xi\in {}^bT^*\tM;
$$
here $a_k$ is the leading coefficient in \eqref{b.9}.  It is not hard to check that this definition is independent of the choice of
connection.  We say that $P$ is \textbf{elliptic} if $\sigma_k(P)(\xi)$ is an invertible element of $\End(E_p)$ for all
$p\in \tM$ and $\xi\in {}^bT^*_p\tM\setminus \{0\}$.
\label{b.10}\end{definition}

\begin{remark}
The principal symbol and ellipticity also make sense for polyhomogeneous $b$-differential operators with nonnegative index set.
\label{b.10b}\end{remark}

In contrast to the situation on closed manifolds, ellipticity alone does not ensure that a $b$-differential operator is Fredholm.
The extra information needed to produce a Fredholm theory is encoded in the indicial family. This is a family of operators on
sections of $E$ over $\pa \tM$ defined by
\begin{equation}
\bbC \ni \tau \mapsto  I(P,\tau)\sigma= \left.  \rho^{-i\tau} P \rho^{i\tau} \widetilde{\sigma} \right|_{\pa \tM},   \quad \sigma\in\CI(\pa \tM;E),
\label{b.11}\end{equation}
where $\widetilde{\sigma}\in \CI(\tM;E)$ is any smooth extension of $\sigma$ to $\tM$.  From the local coordinate description
\eqref{b.8}, we can write
\begin{equation}
  I(P,\tau)= \sum_{\alpha+|\beta|\le k} \left(\left.a_{\alpha,\beta}\right|_{\pa\tM}\right) (i\tau)^{\alpha} \left( \frac{\pa}{\pa y} \right)^{\beta}.
\label{b.11b}\end{equation}

For any elliptic operator $P \in \Diff^*_b(\tM;E)$, we define
$$
   \Spec_b(P)= \{ \tau\in \bbC \; | \; I(P,\tau) \; \mbox{is not invertible}\}.
$$
This set is of fundamental importance in the description of the mapping properties of $P$.  We recall two standard results,
see \cite[Theorem~5.60 and Proposition~5.61]{MelroseAPS}.

\begin{theorem}  If $P\in \Diff^k_{b}(\tM;E)$ is elliptic, then the map
$$
          P: \rho^{\alpha} H^{m+k}_{b}(M;E)\to \rho^{\alpha}H^m_b(M;E)
$$
between weighted $b$-Sobolev spaces is Fredholm if and only if $\alpha\notin -\Im\Spec_b(P)$.
\label{fred.1}\end{theorem}

\begin{proposition}
Let $P\in \Diff^k_{b}(\tM;E)$ be elliptic.  If $u \in \rho^{\alpha} L^2_b (M;E)$ and $Pu\in \cA^{G}_{\phg}(\tM;E)$, then
$$
        u\in \cA^{G\overline{\cup} \widehat{F}^+(\alpha)}_{\phg}(\tM;E),
$$
where 
$$
 \widehat{F}^+(\alpha)=\{ (z,k)\in \bbC\times \bbN_0 \; | \; \exists r\in\bbN_0, \Re z>\alpha+r, \quad -i(z-r)\in \Spec_b(\hat{P}), \; k+1\le \sum_{j=0}^{r} \ord (-i(z-j) ) \}.
$$
\label{breg.1}\end{proposition}

\begin{remark}
The appearance of this somewhat complicated looking index set $\widehat{F}^+(\alpha)$ and the need for taking its extended union
with $G$ to obtain the correct index set for $u$ is explained by the fact that once we know that $u$ is polyhomogeneous, then a
purely formal calculation, matching terms on either side of $Pu = f$ with equal exponents, regulates which terms can appear
in the expansion for $u$. Hence one part of this result is simply the assertion that the solution $u$ must be polyhomogeneous if
$Pu$ is, while the second part asserts precisely which terms appear in its expansion.
\end{remark}

There is an important generalization of Proposition~\ref{breg.1} which follows easily from Theorem~\ref{fred.1}.

\begin{theorem}
Let $P\in \Diff^k_{b}(\tM;E)$ be elliptic. Suppose that $u\in\rho^{\alpha}\CI_b(M;E)$ and
\begin{equation}
       P u= f_1 +f_2, \quad \mbox{where} \quad f_1\in \rho^{\alpha+\beta}\CI_b(M;E) \quad \mbox{and}\quad f_2\in \cA^{G}_{\phg}(\tM;E)
\label{breg.5a}\end{equation}
for some $\beta > 0$ and some index set $G$. Then $u=u_1+u_2$, where
$$
       u_1\in \bigcap_{\delta > 0} \rho^{\alpha+\beta-\delta}\CI_b(M;E), \quad u_2\in \cA^{\hF^+(\alpha)\overline{\cup}G}(\tM;E).
$$
If no  $z \in \Spec_b(P)$ has $-\Im z = \alpha + \beta$, then $u_1 \in \rho^{\alpha + \beta}\CI_b(M;E)$.
\label{breg.5}\end{theorem}
\begin{proof}
Choose  $\delta>0$ small enough so that $-\Im\Spec_b(P)\cap [\alpha+\beta-\delta,\alpha+\beta) =\emptyset$.
Then  $u\in\rho^{\alpha-\delta}H^{m+2}_b(M)$ and $f_1\in \rho^{\alpha+\beta-\delta}H^m_b(M;E)$ for all $m\in\bbN_0$ and the map
$$
    P: \rho^{\alpha+\beta-\delta} H^{m+k}_{b}(M;E)\to \rho^{\alpha+\beta-\delta}H^m_b(M;E)
$$
is Fredholm.  By the density of $\dot{\cC}^{\infty}(\tM;E)$ in $\rho^{\alpha+\beta-\delta}H^m_b(M;E)$, we can find a finite dimensional
subspace $V\subset \dot{\cC}^{\infty}(\tM;E)$ such that
$$
\rho^{\alpha+\beta-\delta} H^m_b(M;E)= P\left( \rho^{\alpha+\beta-\delta} H^{m+k}_b(M;E)\right) + V.
$$
We can therefore find $f_3\in V$ and $u_1\in \rho^{\alpha+\beta-\delta}H^{m+k}_b(M;E)$ such that
$$
    Pu_1= f_1-f_3.
$$
This is true for every $m$, so $u_1\in \rho^{\alpha+\beta-\delta}H^{\infty}_b(M;E)$.  On the other hand, if we set $u_2= u-u_1$,
then
$$
      Pu_2= f_2+f_3\in \cA^{G}_{\phg}(\tM;E),
$$
so by Proposition~\ref{breg.1}, $u_2\in \cA^{\hF^+(\alpha)\overline{\cup}G}(\tM;E)$ as claimed.  Finally, if $(\hF^+(\alpha)\overline{\cup}G)
\cap[\alpha+\beta-\delta,\alpha+\beta) =\emptyset$, then $u_2$ is independent of the choice of $\delta>0$ up to an error term
in $\rho^{\alpha+\beta}(\log\rho)^k\CI_b(M;E)$ for some fixed $k\in\bbN_0$. Hence $u_1\in \rho^{\alpha+\beta-\delta}H^{\infty}_b(M;E) \subset \rho^{\alpha+\beta - \delta} \CI_b(M;E)$
for all $\delta>0$.
\end{proof}

These results extend easily to allow $P$ to have polyhomogeneous coefficients. For Theorem~\ref{fred.1}, we refer to \cite{MazzeoEdge}.  
For Proposition~\ref{breg.1}, one can systematically and with little effort modify the parametrix construction of \cite{MelroseAPS},
see \cite{MazzeoEdge}.  Alternatively, one can extract this generalization directly from Theorem~\ref{breg.5} as follows.
\begin{corollary}
Let $F$ be a nonnegative index set and suppose that $P\in \Diff^k_{b,F}(\tM;E)$ is elliptic. If $u\in\rho^{\alpha}H^m_b(M;E)$ and
$Pu = f \in \cA^G_{\phg}(\tM;E)$, then $u$ is polyhomogeneous as well.
\label{breg.6}\end{corollary}
\begin{proof}
Take $\beta, \delta >0$ sufficiently small so that no element $(z,k) \in F$ has $ z \in (0,\beta + \delta)$, and then decompose
$$
      P= P_0+ \rho^{\beta}P_1,
$$
where $P_0\in \Diff^k_b(\tM;E)$ and $P_1\in \Diff^k_{b,F'}(\tM;E)$ with $F'=(F\setminus\{0\})-\beta > \delta$.  Then
\begin{equation}
      P_0 u= - \rho^{\beta}P_1 u + f,
\label{breg.7}\end{equation}
and since $\rho^{\beta}P_1u\in \rho^{\alpha+\beta+\delta'}\CI_b(M;E)$ for $0<\delta'<\delta$, Theorem~\ref{breg.5} implies that $u=u_1+u_2$
with $u_2 \in \cA^{\hF^+(\alpha)\overline{\cup}G}(\tM;E)$ and $u_1\in \rho^{\alpha+\beta}\CI_b(M;E)$.  Reinserting this into
\eqref{breg.7} gives
\begin{equation}
    P_0 u_1= - \rho^{\beta}P_1 u_1 + f_1 \quad \mbox{with} \quad \rho^{\beta}P_1 u_1\in \rho^{\alpha+2\beta+\delta}\CI_b(M;E) \;
\mbox{and } \; f_1 \; \mbox{polyhomogeneous}.
 \label{breg.8}\end{equation}
Applying Theorem~\ref{breg.5} again, we thus see that $u_1= v_1+v_2$
with $v_1\in \rho^{\alpha+2\beta}\CI_b(M;E)$ and $v_2$ polyhomogeneous, and hence $u = v_1 + v_1'$ with $v_1'=v_2+u_2$ polyhomogneous.  This argument can be iterated,
so for each $k\in \bbN$, we can write $u=v_k+v_k'$ with $v_k\in \rho^{\alpha+k\beta}\CI_b(M;E)$ and $v_k'$
polyhomogeneous. Since $k$ is arbitrary, we see that $u$ is polyhomogeneous as well.
\end{proof}

Replacing Proposition~\ref{breg.1}  by Corollary~\ref{breg.6} in the proof of Theorem~\ref{breg.5}, we obtain the following 
\begin{corollary}
Let $P\in \Diff^k_{b,F}(\tM;E)$ be elliptic with nonnegative index family $F$ and let $G$ be another index set.
Suppose that $u\in\rho^{\alpha}\CI_b(M;E)$ satisfies
$$
       P u= f_1 +f_2 \quad \mbox{with} \quad f_1\in \rho^{\alpha+\beta}\CI_b(M;E) \; \mbox{and}\; f_2 \; \mbox{polyhomogeneous}.
$$
Then $u=u_1+u_2$ with
$$
       u_1\in \bigcap_{\delta > 0} \rho^{\alpha+\beta-\delta}\CI_b(M;E)   \qquad \mbox{and}
\qquad u_2 \ \  \mbox{polyhomogeneous}.
$$
\label{breg.9}\end{corollary}

\section{Hodge theory for asymptotically cylindrical Kähler manifolds} \label{ht.0}
Let $\bM$ be a compact Kähler orbifold of complex dimension $n\ge 2$.  Let $\bD$ be an effective orbifold divisor satisfying the following two conditions:
\begin{itemize}
\item[(i)] The complement $M:= \bM\setminus \bD$ is a smooth manifold;
\item[(ii)] The orbifold normal bundle of $\bD$ is biholomorphic to $(\bbC\times D)/ \langle \iota\rangle$ as an orbifold line bundle, where $D$ is a connected smooth complex manifold and $\iota$ is a complex automorphism of $D$ of order $m<\infty$ acting on the product via $\iota(w,x)= (e^{\frac{2\pi i}{m}}w,\iota(x))$.
\end{itemize}

Let $\tM:=[\bM;\bD]$ be the manifold with boundary obtained by taking the real blow-up of $\bM$ around $\bD$, cf.\ \cite{MelroseAPS}.
Although $\bM$ may have an orbifold singularity along $\bD$, the blow-up $\tM$ is a smooth manifold with boundary $\pa\tM$
which is naturally identified with the total space of the orbifold unit normal bundle of $\bD$.  Thus $\pa \tM$ is foliated by circles
and the space of leaves is identified with the orbifold $\bD$.  In particular, if $\bD$ is smooth, this circle foliation is a circle bundle.
The manifold $\tM$ is an example of a $b$-complex manifold, as defined by Mendoza \cite{Mendoza}.

Suppose now that $g_b$ is a polyhomogeneous Kähler $\AC$-metric on $M= \bM\setminus \bD=
\tM\setminus \pa\tM$, and denote by $\omega_b$ its Kähler form. Fix a defining function $\rho\in \CI(\tM)$ and let $E\to \bM$
be a holomorphic vector bundle over $\bM$ equipped with a Hermitian metric $h$. We then consider on the
quasiprojective manifold $M= \bM\setminus \bD$, for any $a\in \bbR$, the weighted $L^2$-Dolbeault complex
\begin{equation}
\xymatrix{
  \cdots \ar[r]^-{\db} & \rho^{a}L^2_{\db}\Omega^{p,q}(M;E,g_b) \ar[r]^-{\db} &   \rho^{a}L^2_{\db}\Omega^{p,q+1}(M;E,g_b) \ar[r]^-{\db} & \cdots},
\label{ht.1}\end{equation}
where $L^2\Omega^{p,q}(M;E,g_b)$ is the space of forms of type $(p,q)$ on $M$ which are $L^2$ with respect to the metric $g_b$, and
\begin{equation}
  \rho^aL^2_{\db}\Omega^{p,q}(M;E,g_b)= \{ \mu\in \rho^aL^2\Omega^{p,q}(M;E,g_b) \quad | \quad \db\mu \in   \rho^aL^2\Omega^{p,q+1}(M;E,g_b)  \}.
\label{ht.2}\end{equation}
Denote the cohomology groups of the complex \eqref{ht.1} by
\begin{equation}
  \WH^{p,q}(g_b,a,M;E) =  \frac{  \{ \mu\in \rho^aL^2\Omega^{p,q}(M;E,g_b) \quad | \quad \db \mu=0  \}}{ \{ \db \zeta \in \rho^aL^2\Omega^{p,q}(M;E,g_b) \quad | \quad
        \zeta \in \rho^aL^2_{\db}\Omega^{p,q-1}(M;E,g_b) \}}.
\label{ht.2}
\end{equation}

These weighted $L^2$-cohomology groups are related to the sheaf cohomology of certain holomorphic vector bundles
on $\bM$.
\begin{definition}
The logarithmic tangent sheaf $T_{\bM}(\log \bD)$ is the subsheaf of the tangent sheaf $T_{\bM}$ of $\bM$ consisting of derivations of $\mathcal{O}_{\bM}$ sending the ideal sheaf $\mathcal{I}_{\bD}$ of $\bD$ in $\mathcal{O}_{\bM}$ to itself.  In other words, $T_{\bM}(\log \bD)$ is the sheaf of holomorphic vector fields tangent to $\bD$.  We also denote by $\Omega^1(\log \bD)$ the corresponding dual sheaf of logarithmic $1$-forms and by $\Omega^p (\log \bD)$ the sheaf of the $p^{\mathrm{th}}$ exterior power of $\Omega^1(\log \bD)$ with itself.
\end{definition}
\begin{theorem}
For $\epsilon>0$ sufficiently small, there are canonical identifications
$$
\WH^{p,q}(g_b,\epsilon,M;E)\cong H^{q}(\bM ,\Omega^p(\log \bD) \otimes E(-\bD)), \quad  \WH^{p,q}(g_b,-\epsilon,M;E)
\cong H^{q}(\bM;\Omega^{p}(\log\bD)\otimes E),
$$
where $E(-\bD)$ is the holomorphic vector bundle on $\bM$ associated to the sheaf of holomorphic sections of $E$ vanishing along $\bD$.
\label{ht.3}
\end{theorem}
\begin{proof}
The idea is to adapt the sheaf theoretic proof of the Dolbeault theorem, see \cite[p.45]{Griffiths-Harris} for example, to our context.
Fix $a\ne 0$.  Denote by $\rho^{a}L^2\cH^{p,0}(E)$ the sheaf induced by the presheaf of  local holomorphic $p$-forms on $\bM$ with values in
$E$ which are $\rho^{a}L^2$ with respect to $g_b$ and the Hermitian metric $h$ of $E$.  Also, let $\rho^{a}L^2_{\db}\Omega^{p,q}(E)$
denote the sheaf defined by the presheaf which associates to $\bU\subset \bM$ the abelian group
 $$
    \{ \mu \in \rho^{a}L^2\Omega^{p,q}(\bU\cap M;E, g_b) \quad | \quad \db \mu \in \rho^{a}L^2\Omega^{p,q+1}(\bU\cap M;E,g_b) \}.
 $$
Finally, let $\rho^{a}L^2 \cZ^{p,q}(E)$ be the subsheaf of $\rho^{a}L^2_{\db}\Omega^{p,q}(E)$ which associates to $\bU$ the abelian group
$$
           \{  \mu\in \rho^{a}L^2\Omega^{p,q}(E)_{\bU} \quad | \quad \db \mu=0  \}.
$$

By Lemma~\ref{ht.3b} below (which is a version of the $\db$-Poincaré lemma in $\rho^{a}L^2$), we know that if $a$ is sufficiently
close to $0$, there are short exact sequences of sheaves,
 \begin{gather}
 \xymatrix{ 0 \ar[r] & \rho^{a}L^2\cH^{p,0}(E) \ar[r] & \rho^{a}L^2_{\db}\Omega^{p,0}(E) \ar[r]^-{\db} & \rho^{a}L^2\cZ^{p,1}(E)\ar[r] & 0,
 }  \quad q=0,  \\
  \xymatrix{ 0 \ar[r] & \rho^{a}L^2\cZ^{p,q}(E) \ar[r] & \rho^{a}L^2_{\db}\Omega^{p,q}(E) \ar[r]^-{\db} & \rho^{a}L^2\cZ^{p,q+1}(E)\ar[r] & 0,
 }  \quad q>0.
 \end{gather}
We know from \cite[Proposition~2, p.500]{HHM2004} (see also the beginning of the proof of \cite[Corollary~17]{GR2013}) that
the sheaf $\rho^{a}L^2_{\db}\Omega^{p,q}(E)$ is fine, so that $H^k(\bM; \rho^{a}L^2_{\db}\Omega^{p,q}(E))=\{0\}$ when $k>0$.
The corresponding long exact sequences in cohomology then give that
 \begin{equation}
 \begin{aligned}
    H^q(\bM; \rho^{a}L^2\cH^{p,0}(E)) &\cong H^{q-1}(\bM; \rho^{a}L^2\cZ^{p,1}(E)) \\
          & \cong H^{q-2}(\bM; \rho^{a}L^2\cZ^{p,2}(E)) \\
          &\cong \cdots \cong H^1(\bM; \rho^{a}L^2\cZ^{p,q-1}(E)) \\
          &\cong H^0(\bM; \rho^{a}L^2\cZ^{p,q}(E)) /  \db H^0(\bM; \rho^{a}L^2_{\db}\Omega^{p,q-1}(E)) \\
          & \cong \WH^{p,q}(g_b,a,M;E).
\end{aligned}\label{ht.3c}\end{equation}
If $a>0$ is sufficiently close to zero, then $\rho^aL^2\cH^{p,0}(E)$ is identified with the sheaf $\Omega^p(\log \bD)\otimes (E(-\bD))$
of holomorphic $p$-forms on $\bM$ with values in $E(-\bD)$, while $\rho^{-a}L^2\cH^{p,0}(E)$ is identified with the sheaf \linebreak
 $\Omega^p(\log \bD)\otimes E$ of holomorphic $p$-forms with values in $E$.  Thus, by \eqref{ht.3c}, when $\epsilon>0$ is
sufficiently small,
\begin{equation}
 \begin{gathered}
 \WH^{p,q}(g_b,\epsilon,M;E)\cong H^q(\bM; \Omega^{p}(\log\bD)\otimes (E(-\bD))), \\
 \WH^{p,q}(g_b,-\epsilon,M;E)\cong H^q(\bM; \Omega^{p}(\log\bD)\otimes E).  \end{gathered}
\label{ht.3d}\end{equation}
 \end{proof}

 \begin{lemma}
 The morphism of sheaves $\db: \rho^{a}L^2_{\db}\Omega^{p,q}(E) \to \rho^{a}L^2\cZ^{p,q+1}(E)$ is surjective for $a\ne 0$ sufficiently small.
 \label{ht.3b}\end{lemma}
\begin{proof}
Suppose first that $\bD$ is smooth.  It then suffices to show that the map
\begin{equation}
   \db: \rho^{a}L^2_{\db}\Omega^{p,q}(E)_{\bU} \to \rho^{a}L^2\cZ^{p,q+1}(E)_{\bU}
\label{pdl.1}\end{equation}
is surjective for any open set $\bU\subset \bM$ biholomorphic to a polycylinder $\Delta\subset \bbC^n $ over which $E$ is trivial
when lifted to $\Delta$.  We can restrict to sets of this type which are either disjoint from $\bD$, or else for which there is a
biholomorphism $\varphi: \bU\to \Delta \subset \bbC^n$ mapping $\bD\cap \bU$ onto $\Delta\cap (\bbC^{n-1}\times \{0\})$.

The assertion is not hard in complex dimension $n=1$.  Indeed, in that case regard $\Delta$ as a disk centered at $0$ in
$\bbC\cup \{\infty\}= \bbC\bbP^1$.  If $\bU\cap \bD=\emptyset$, we use the surjectivity of
 $$
     \db: H^1(\bbC\bbP^1)\to L^2\Omega^{0,1}(\bbC\bbP^1)
 $$
to see that \eqref{pdl.1} is surjective.  If $\bU\cap \bD\ne \emptyset$, then assume that the biholomorphism
$\varphi: \bU\to \Delta\subset \bbC$ maps $\bU\cap \bD$ to $0\in \Delta$.  Now put a complete asymptotically cylindrical
Kähler metric $k_b$ on $\bbC\bbP^1\setminus \{0\}$ and let $x\in C^{\infty}(\bbC\bbP^1)$ be the boundary defining function
which equals the Euclidean distance to the origin near $0\in \Delta\subset \bbC$.
Since there are no nontrivial meromorphic $1$-forms on $\bbC\bbP^1$ with at most one simple pole, then by
\cite[\S 6.3]{MelroseAPS}, we know that
$$
        \db: x^a H^1(\bbC\bbP^1\setminus \{0\}, k_b) \to x^a L^2\Omega^{0,1}(\bbC\bbP^1\setminus \{0\} ,k_b)
$$
is Fredholm and surjective whenever $a$ is nonzero but sufficiently small. Its kernel is trivial when $a>0$, while if $a < 0$, it is just the constants. Restricting to $\Delta$, we see that the map \eqref{pdl.1} is again surjective.  This proves
the result in complex dimension $1$.

From this discussion, we see that when $n=1$, there is a right inverse
\begin{equation}
       (\db)^{-1}: x^{a}L^2\cZ^{p,q+1}(E)_{\bU}\to x^{a}L^2_{\db}\Omega^{p,q}(E)_{\bU}
\label{pdl.3}\end{equation}
to \eqref{pdl.1}.  To prove the result in complex dimensions greater than $1$, we then proceed just as in
\cite[p.25-26]{Griffiths-Harris},  although we use \eqref{pdl.3} instead of the one-variable $\db$-Poincaré lemma,
cf.\ \cite[p.5]{Griffiths-Harris}.

Finally, if $\bM$ has orbifold singularities, then we can proceed as before away from $\bD$.  Near $\bD$ however, we
must also consider open sets of the form $\bU= V/\langle \mu \rangle$ with $\mu$ a finite order automorphism of $V$
such that the restriction of $E$ to $\bU$ lifts on $V$ to a trivial holomorphic vector bundle.  The preceding discussion
proves the surjectivity of \eqref{pdl.1} on sufficiently small open sets $V$, and we can then average with respect to the action
of $\mu$ to obtain the desired surjectivity on $\bU$.
\end{proof}

Our main interest here is in the case where the weight $a = 0$, but the cohomology then is often infinite dimensional
since $\db+ \db^*$ does not have closed range when acting between appropriate Sobolev spaces. Here $\db^*$ is the
formal adjoint of $\db$ with respect $g_b$ and a choice of Hermitian metric on $h$.  However, we can still consider the
space of $L^2$ harmonic forms of type $(p,q)$ with values in $E$, namely
\begin{equation}
  L^2\cH^{p,q}(M;E) = \{ \mu\in L^2\Omega^{p,q}(M;E, g_b,h) \quad | \quad   \db\mu=0, \quad \db^{*}\mu=0\}.
\label{ht.4}\end{equation}
Since $\db+ \db^*$ is an elliptic $b$-operator (acting on sections of $\Lambda^{p,*} \otimes E = \oplus_q \Lambda^{p,q} \otimes E$),
this space is finite dimensional and every element $\eta \in L^2\cH^{p,q}(M;E)$ is polyhomogeneous, namely
$$
             \eta \sim  \sum_{(z,k)\in \cI}  \rho^z (\log \rho)^k \eta_{z,k} \quad \mbox{near $\rho=0$, where }
\quad \eta_{z,k}\in \CI(\widetilde{M}, \Lambda^{p,q}({}^bT^* \widetilde{M})\otimes E).
$$
The index set $\cI$ here is determined solely by the indicial family of $\db+\db^*$, i.e., it is independent of $\eta$.

For convenience below, let us denote by $\cI_a$ the index set corresponding to elements $\eta \in \rho^a L^2 \Omega^{p,*}$
which lie in the common nullspace of $\db$ and $\db^*$. Note that the condition $\rho^z \in L^2$ implies
$z > 0$, so $\cI_0$ is a positive index set.  Henceforth we shall always choose $\epsilon$ so that
\begin{equation}
0 < \epsilon < \inf \cI_0.
\label{defe}
\end{equation}

By virtue of this remark, any $\eta \in L^2 \cH^{p,q}(M;E)$ is in $\rho^{\epsilon}L^2\Omega^{p,q}(M;E,g_b)$ and satisfies $\db \eta = 0$, thus represents an element in $\WH^{p,q}(g_b, \epsilon, M;E)$. Hence, composing with the natural map
between weighted cohomologies, we obtain a map
\begin{equation}
\Phi: L^2\cH^{p,q}(M;E) \to  \Im\{ \WH^{p,q}(g_b,\epsilon,M;E)\to  \WH^{p,q}(g_b,-\epsilon,M;E)\}.
\end{equation}
We wish to show that $\Phi$ is an isomorphism, and to this end we collect some results.

First note that
\begin{equation}
      \db+ \db^* : \rho^{-\epsilon} H^1_b\Omega^{p,*}(M;E) \to \rho^{-\epsilon}L^2\Omega^{p,*}(M;E,g_b)
\label{bi.1}\end{equation}
is Fredholm when $\epsilon$ satisfies \eqref{defe}, where $H^k_b\Omega^{p,q}(M;E)$ is the $b$-Sobolev space of order $k$ of forms of type $(p,q)$ taking values in $E$. The symmetry of $\db+\db^*$ with respect to the $L^2_b$ pairing
and the fact that the spaces
$\rho^{\pm\epsilon}L^2\Omega^*(M;E,g_b)$ are dual to each other means that the cokernel of \eqref{bi.1} can be identified
with the kernel of $\db+\db^*$ on $\rho^{\epsilon} H^1_b\Omega^*(M;E)$, which as we have just shown is the same as
$L^2\cH^*(M;E)$.  Hence there is a direct sum decomposition
\begin{equation}
\rho^{-\epsilon}L^2\Omega^*(M;E,g_b)= \Im\{  \db+ \db^* : \rho^{-\epsilon} H^1_b\Omega^*(M;E,g_b) \to
\rho^{-\epsilon}L^2\Omega^*(M;E,g_b)  \} \oplus L^2\cH^*(M;E).
\label{bi.2}
\end{equation}
For similar reasons, the $\db$-Laplacian $\Delta_{\db}= (\db+\db^*)^2$ induces the decomposition
\begin{equation}
\rho^{-\epsilon}L^2\Omega^*(M;E,g_b)=\Im\{  \Delta_{\db} : \rho^{-\epsilon} H^2_b\Omega^*(M;E,g_b) \to \rho^{-\epsilon}L^2\Omega^*(M;E,g_b)  \}
\oplus L^2\cH^*(M;E).
\label{bi.2a}
\end{equation}

\begin{proposition}
There is a finite dimensional subspace $A\subset \dot{\Omega}^*(M;E)$ orthogonal to $L^2\cH^*(M;E)$
such that
\begin{equation}
\rho^{\epsilon}L^2\Omega^*(M;E,g_b)= \Im\{  \Delta_{\db} : \rho^{\epsilon} H^2_b\Omega^*(M;E) \to \rho^{\epsilon}L^2\Omega^*(M;E,g_b)  \}
\oplus A \oplus L^2\cH^*(M;E).
\label{bi.2b}\end{equation}
\label{bi.4}\end{proposition}
\begin{proof}
Using the density of $\dot{\Omega}^*(M;E)$ in $L^2\Omega^*(M;E,g_b)$, we first find a finite dimensional subspace
$A'\subset \dot{\Omega}^*(M;E)$ such that
$$
\rho^{\epsilon}L^2\Omega^*(M;E,g_b)= \Im\{  \Delta_{\db} : \rho^{\epsilon} H^2_b\Omega^*(M;E) \to \rho^{\epsilon}L^2\Omega^*(M;E,g_b)  \}
\oplus A' \oplus L^2\cH^*(M;E).
$$
This $A'$ need not be orthogonal to $L^2\cH^*(M;E)$, but by subtracting the $L^2$-harmonic component of each element of $A'$,
we obtain a finite dimensional space $A''\subset \cA^{\cI_0}_{\phg}\Omega^*(M;E)$ orthogonal to $L^2\cH^*(M;E)$ such that
$$
\rho^{\epsilon}L^2\Omega^*(M;E,g_b)= \Im\{  \Delta_{\db} : \rho^{\epsilon} H^2_b\Omega^*(M;E) \to \rho^{\epsilon}L^2\Omega^*(M;E,g_b)  \} \oplus
A''\oplus L^2\cH^*(M;E).
$$
Choose a basis $a_1, \ldots, a_p$ for $A''$.  Then by \cite[Lemma~5.44]{MelroseAPS}, we can find $b_1, \ldots,  b_p\in
\cA^G_{\phg}\Omega_b^*(\tM;E)$, where $G$ is an index set containing $\cI_0$ and with $\inf G=\inf \cI_0$, such that
$$
          a_i- \Delta_{\db}b_i \in \dot{\Omega}^*(M;E) \quad \mbox{for} \; i=\{1,\ldots,p\}.
$$
Clearly, the $\Delta_{\db}b_i$ are all orthogonal to $L^2 \cH^*(M,E)$. We thus let $A$ be the span of
the forms $a_i-\Delta_{\db}b_i$, $i=1,\ldots,p$.
\end{proof}
This has the following useful consequence.
\begin{corollary}
Let $\zeta\in \rho^{-\epsilon}H^1_b\Omega^{p,q-1}(M;E) \oplus \rho^{-\epsilon}H^1_b\Omega^{p,q+1}(M;E)$
and suppose that
\begin{equation}
 (\db+\db^*)\zeta\in \rho^{\epsilon}L^2_b\Omega^{p,q}(M;E).
\label{bi.4b}\end{equation}
Then  $\zeta= \zeta_1+\zeta_2$, where $\zeta_1$ is polyhomogeneous and $\zeta_2\in \rho^{\epsilon}H^1_b\Omega^*(M;E)$, and
the only nonzero components of $\zeta_1$ and $\zeta_2$ are in degrees $(p,q-1)$ and $(p,q+1)$.
Furthermore, $\zeta_1 = \mu_0+ \frac{d\rho}{\rho}\wedge \nu_0 + \mathcal O(\rho^\epsilon)$, where $\mu_0$ and $\nu_0$ are
harmonic, so $\db\zeta, \db^*\zeta \in \rho^{\epsilon}L^2_b\Omega^{p,q}(M;E)$.
\label{bi.4a}
\end{corollary}
\begin{proof}
By \eqref{bi.2a}, \eqref{bi.2b} and \eqref{bi.4b}, we can write $(\db+\db^*)\zeta=\Delta_{\db}(\eta_1+\eta_2) + \eta_3$
where $\eta_1\in \rho^{-\epsilon}H^2_b\Omega^{p,q}(M;E)$ satisfies $\Delta_{\db}\eta_1 \in A\subset \dot{\Omega}^*(M;E)$,
$\eta_2\in \rho^{\epsilon}H^2_b\Omega^{p,q}(M;E)$ and $\eta_3 \in L^2 \cH^{p,q}(M;E)$. Set
$\zeta_1=(\db+\db^*)\eta_1\in \rho^{-\epsilon}H^1_b\Omega^*(M;E)$ and $\zeta_2= (\db+\db^*)\eta_2\in \rho^{\epsilon}
H^1_b\Omega^{*}(M;E)$; these only have nonzero components in degrees $(p,q-1)$ and $(p,q+1)$, and we have that
$$
(\db+\db^*)\zeta= (\db+\db^*) (\zeta_1 + \zeta_2) + \eta_3.
$$
Integrating by parts in $|| \eta_3||^2 = \langle \eta_3, (\db + \db^*)(\zeta - \zeta_1 - \zeta_2)\rangle$, where the
pairing is in $L^2_b$, shows that $\eta_3 = 0$.  We then see from this that $\zeta - \zeta_1 - \zeta_2 = \gamma$
is an element of the nullspace of $\db+\db^*$ on $\rho^{-\epsilon}H^1_b\Omega^*(M;E)$, so by replacing $\zeta_1$ by
$\zeta_1 - \gamma$, we may as well assume that $\zeta=\zeta_1+ \zeta_2$.

By Corollary~\ref{breg.6}, $\zeta_1$ lies in $\cA^{\cI_{-\epsilon}}_{\phg}$, so the leading term in its expansion is $\rho^0$
with coefficient lying in the kernel of the indicial family of $\db+ \db^*$, and hence also in the kernel of
the indicial family of $\Delta_{\db}= \frac12 \Delta= (\db+\db^*)^2$ at $\tau = 0$, i.e., it is harmonic and has the form
\begin{equation}
            \mu_0 + \frac{d\rho}{\rho}\wedge \nu_0
\label{bi.4c}\end{equation}
with $\mu_0, \nu_0$ harmonic on $\pa \tM$.  Thus $d\zeta_1 \in \rho^{\epsilon}L^2_b\Omega^{*}(M;E)$, and since $\zeta$ only
has nonzero components of types $(p,q-1)$ and  $(p,q+1)$, we see that  $\pa \zeta_1, \db \zeta_1\in
\rho^{\epsilon}L^2_b\Omega^{*}(M;E)$ individually. Since $(\db+\db^*)\zeta_1 \in \rho^{\epsilon}L^2\Omega^{*}(M;E)$,
we also obtain that $\db^*\zeta_1 \in \rho^{\epsilon}L^2\Omega^{*}(M;E)$.  Altogether, we have shown the final claim, that $\db\zeta,
\db^*\zeta\in \rho^{\epsilon}L^2_b\Omega^{*}(M;E)$.
 \end{proof}

The following is a simple adaptation of a result of \cite{MelroseAPS}, see also \cite{HHM2004}.
\begin{theorem}
For $\epsilon>0$ sufficiently small, there is a natural identification
\begin{equation}
\begin{aligned}
L^2\cH^{p,q}(M;E) &\cong  \Im\{ \WH^{p,q}(g_b,\epsilon,M,E)\to  \WH^{p,q}(g_b,-\epsilon,M,E)    \} \\
   &\cong \Im \{ H^{q}(\bM;\Omega^p(\log \bD)\otimes E(-\bD))\to H^{q}(\bM; \Omega^{p}(\log\bD) \otimes E) \}.
\end{aligned}
\end{equation}
\label{ht.5}\end{theorem}
\begin{proof}
We follow closely the proof of \cite[Theorem~2B]{HHM2004}.

Let us first show that the map $\Phi$ in \eqref{ht.4} is injective. Thus, assume that there is an $\eta\in L^2\cH^{p,q}(M;E)$
such that $\Phi(\eta)=0$.  This means that $\eta= \db \zeta$ for some $\zeta\in \rho^{-\epsilon}L^2\Omega_{\db}^{p,q-1}(M;E,g_b)$.
Since $\epsilon<\inf \cI$, we can integrate by parts to deduce that 
$$
         \|\eta \|^2_{L^2} = \int_M \langle \eta, \db \zeta\rangle d g_b = \int_M \langle \db^* \eta, \zeta\rangle d g_b=0,
$$
i.e., $\eta=0$.

To prove the surjectivity of $\Phi$, fix $[\eta]\in  \Im\{ \WH^{p,q}(g_b,\epsilon,M;E)\to  \WH^{p,q}(g_b,-\epsilon,M;E)\}$.
We must show that $[\eta]$ is in the image of $\Phi$. If $\eta\in \rho^{\epsilon}L^2_{\db}\Omega^{p,q}(M;E,g_b)$ is a representative
of this class, then by \eqref{bi.2a}, there are $\nu\in \rho^{-\epsilon}H^2_b\Omega^{p,q}(M;E)$ and $\gamma\in L^2\cH^{p,q}(M;E)$ such that
\begin{equation}
     \eta = (\db +\db^*)^2\nu + \gamma= (\db+\db^*)\zeta + \gamma  \quad\mbox{with} \quad  \zeta:=
(\db+\db^*)\nu\in \rho^{-\epsilon}H^1_b\Omega^{*}(M;E).
\label{bi.3}\end{equation}
The assertion is then equivalent to showing that $\db^* \zeta=0$.  By Corollary~\ref{bi.4a}, $\db^*\zeta\in \rho^{\epsilon}L^2_b\Omega^{p,q}(M;E)$, so the integration by parts
$$
\langle \db \zeta, \db^*\zeta  \rangle_{L^2_b}= \int_M (\db\zeta) \wedge * \db^* \zeta = \int_M \db( \zeta\wedge *\db^* \zeta)=  \int_M d(\zeta\wedge *\db^*\zeta)=0
$$
is justified and shows that $\db\zeta$ is orthogonal to $\db^*\zeta$. We have used here that $\zeta\wedge*\db^*\zeta$ is a
form of type $(n,n-1)$, so that $\pa(\zeta\wedge *\db^*\zeta)=0$.  Similarly, $\langle \db^* \zeta, \eta \rangle_{L^2_b}=
\langle \db^* \zeta, \gamma\rangle_{L^2_b}=0$, so we conclude from \eqref{bi.3} that
$$
  \| \db^* \zeta \|^2_{L^2}=0  \quad \Longrightarrow \quad \db^*\zeta \equiv 0.
$$
We have thus proved that $\eta= \db \zeta + \gamma$, hence $[\eta]$ indeed lies in the image of the map $\Phi$.  The second identification
of the theorem follows by applying Theorem~\ref{ht.3}.
\end{proof}

\begin{corollary}
If the metric $g_b$ is Ricci-flat, then $L^2\cH^{p,0}(M)\cong L^2\cH^{0,p}(M)\cong \{0\}$, and so
$$
     \Im \{ H^{0}(\bM;\Omega^p(\log \bD)\otimes \mathcal{O}(-\bD))\to H^{0}(\bM; \Omega^{p}(\log\bD)) \}\cong   \Im \{ H^{p}(\bM;\mathcal{O}(-\bD))\to H^{p}(\bM; \mathcal{O}) \} =\{0\}.
$$

\label{ht.5a}\end{corollary}
\begin{proof}
This is a standard argument, \cf \cite[Proposition~6.2.4]{Joyce}. The Weitzenbock formula on $(p,0)$-forms specializes,
since $g_b$ is Ricci-flat, to
$$
      \Delta_{d} \xi = \nabla^*\nabla \xi.
$$
Thus if $\xi\in L^2\cH^{p,0}(M)$, then the left hand side vanishes, and integrating by parts yields that $\xi$ is parallel.
But $\xi \in L^2$, so $\xi\equiv 0$. This proves that $L^2\cH^{p,0}(M)= \{0\}$, and by Hodge duality, that $L^2\cH^{0,p}(M)= \{0\}$.
\end{proof}

Parallel to \cite[Proposition~6.18]{MelroseAPS}, we now give an analytical characterization of the weighted cohomology
$\WH^{p,q}(g_b,-\epsilon,M,E)$ using the $\db$-Laplacian $\Delta_{\db}= \db\db^* + \db^*\db$ instead of $\db+\db^*$
and in terms of the space of extended $L^2$ harmonic forms
$$
     \ker^{p,q}_{-} \Delta_{\db}= \bigcap_{\epsilon>0} \{ \eta\in \rho^{-\epsilon}H^2_b\Omega^{p,q}(M;E,g_b)\quad | \quad \Delta_{\db} \eta=0 \}.
$$
As in \cite{MelroseAPS}, consider the space 
\begin{equation}
\begin{aligned}
    F(\db+\db^*,0)&=  \{ \mu_0 + \frac{d\rho}{\rho}\wedge \nu_0 \quad | \quad \mu_0+ \frac{d\rho}{\rho} \wedge\nu_0 \in \ker I(\Delta_{\db},0)\} \\
        &=\{ \mu_0 + \frac{d\rho}{\rho}\wedge \nu_0 \quad | \quad \mu_0, \nu_0 \; \mbox{are harmonic on } \;\pa \tM\}.
\end{aligned}
\label{fn.1}\end{equation}
Here, $I(\Delta_{\db},\lambda)=  \frac{1}2 \Delta_{\pa\tM}+c\lambda^2$ is the indicial family of $\Delta_{\db}$, with $\Delta_{\pa\tM}$
the Laplacian induced by the `restriction' of the metric $g_b$ on $\pa\tM$, and $c =g_b(\rho\pa_\rho, \rho \pa_\rho)$ at $\rho=0$.
Now define a pairing on $F(\db+\db^*,0)$ by
\begin{equation}
      B(u,v)=   \langle(\db+\db^*)\widetilde{u}, \widetilde{v}\rangle_{L^2_b} - \langle \widetilde{u}, (\db+\db^*)\widetilde{b}\rangle_{L^2_b},
\label{fn.2}\end{equation}
where $\widetilde{u}, \widetilde{v}\in \CI(\tM; \Lambda^*({}^bT^*M))$ are smooth extensions of $u$ and $v$ to $\tM$.
This is independent of the choice of extensions.  As shown in \cite[Proposition~6.2]{MelroseAPS}, this pairing is non degenerate, 
and clearly, if $\widetilde{v}$ is a smooth extension of an element $v\in F(\db+\db^*,0)$ of type $(p,q)$, then
\begin{equation}
      \db \widetilde{v}\in \rho^{\epsilon}H^{\infty}_b\Omega^{p,q+1}(M;E), \quad \db^*\widetilde{v}\in  \rho^{\epsilon}H^{\infty}_b\Omega^{p,q-1}(M;E)
\label{fn.3}\end{equation}
for $\epsilon\in (0, \inf\cI)$.
\begin{lemma}
Suppose that $u \in \rho^{-\epsilon}H^{2}_{b}\Omega^{p,q}(M;E)$ and $\Delta_{\db }u$ is polyhomogeneous, lying
in some $\cA^{F}_{\phg} \Omega^{p,*}(M;E)$ where $\inf F > \epsilon$. Then $u$ is polyhomogeneous and  $\db u,\db^*\!u$ are bounded polyhomogeneous. If $u$ is bounded polyhomogeneous, then $\db u, \db^*\!u \in \cA^{G}_{\phg}\Omega^{p,*}(M;E)$ for
some index set $G$ with $\inf G>\epsilon$.
\label{bi.5c}\end{lemma}
\begin{proof}
Corollary~\ref{breg.6} already shows that $u$ is polyhomogeneous.  Since the indicial family
$I(\Delta_{\db}, \lambda)=\frac12 \Delta_{\pa\tM}+
c\lambda^2$ is quadratic in $\lambda$, its inverse has a pole of order at most $2$ at $\lambda=0$. This means that
the leading terms in the expansion of $u$ are
$$
         u_{0,1} \log \rho + u_0,  \quad u_{0,1}, u_0\in \ker I(\Delta_{\db},0).
$$
But $I(\Delta_{\db},0)=I(\db+\db^*,0)^2$ and $I(\db+\db^*,0)$ is self-adjoint, so that $u_{0,1}, u_0\in \ker I(\db+\db^*,0)$.
Furthermore, since $\db u$ is of type $(p,q+1)$ and $\db^*u$ is of type $(p,q-1)$, we get $I(\db,0)u_{0,1}= I(\db^*,0)u_{0,1}=0$ and
$I(\db,0)u_{0}= I(\db^*,0)u_{0}=0$.  In particular,  $\db u$ and $\db^* u$ are bounded, and if $u$ is also bounded, i.e.,
$u_{0,1}=0$, then in fact $\db u, \db^* u\in \rho^{\epsilon}H^{\infty}_b\Omega^{*}(M;E)$.
\end{proof}

\begin{lemma}
If $w\in \db^* \ker^{p,q}_{-} \Delta_{\db}$, then $w$ is bounded polyhomogeneous with $\db w=0$.  More generally, any
bounded $u\in \ker^{p,q}_{-} \Delta_{\db}$ satisfies $\db u=\db^* u=0$.
\label{bi.5}\end{lemma}
\begin{proof}
By Lemma~\ref{bi.5c}, $w$ is bounded polyhomogeneous. Thus we must show that if $u\in \ker^{p,q}_{-}\Delta_{\db}$ is bounded,
then $\db u=\db^* u =0$.  But in this case, $\db u\in \rho^{\epsilon}H^{\infty}_b\Omega^{p,q+1}(M;E,g_b)$, so
$\db u$, which a priori only lies in $\ker^{p,q}_- \Delta_{\db}$, is actually an $L^2$-harmonic form. Thus, $\db^* \db u=0$,
and integrating by parts yields
$$
  \| \db u \|^2_{L^2} = \langle \db u, \db u \rangle_{L^2_b}= \langle u, \db^*\db u\rangle_{L^2_b}= 0 \quad \Longrightarrow \quad \db u=0.
$$
Similarly, $\db^* u$ is an $L^2$-harmonic form, and
$$
  \| \db^* u \|^2_{L^2} = \langle \db^* u, \db^* u \rangle_{L^2_b}= \langle u, \db \, \db^* u\rangle_{L^2_b}= 0 \quad \Longrightarrow \quad \db^* u=0.
$$
\end{proof}

Using these lemmas with $E$ a trivial holomorphic line bundle, we can prove a $\pa\db$-lemma which will be important later on.
\begin{lemma}[$\pa\db$-lemma]
Let $\alpha$ be a bounded polyhomogeneous $(p,q)$-form which is $\db$-closed, with $\alpha= \pa \beta$ for some
bounded polyhomogeneous $(p-1,q)$-form. Then there exists a polyhomogeneous $(p-1,q-1)$-form  \nolinebreak $\mu$ such that
$\alpha= \pa\db \mu$ with  $\pa \mu$ bounded polyhomogeneous.  Furthermore, if
$\beta \in \cA^F_{\phg}$ for some positive index set $F$, then $\mu$ is bounded polyhomogeneous and
$\pa\mu\in \rho^{\epsilon}H^{\infty}_b\Omega^{p,q-1}(M)$.
\label{dd.1}\end{lemma}
\begin{proof}
From \eqref{bi.2a}, there exist $\psi\in \rho^{-\epsilon}H^{2}_b\Omega^{p-1,q}(M)$ and $\gamma\in L^2\cH^{p-1,q}(M)$ such that
$$
     \beta= \Delta_{\db} \psi+ \gamma.
$$
But $\beta$ and $\gamma$ are both polyhomogeneous, so $\psi$ is polyhomogeneous as well, with top order terms
$$
    \psi_{0,2} (\log \rho)^2 + \psi_{0,1}\log \rho + \psi_{0,0}, \qquad  \psi_{0,2}, \psi_{0,1} \in
\mathrm{ker}\, I(\Delta_{\db}, 0).
$$
As in the proof of Lemma~\ref{bi.5c}, this means that $\psi_{0,2}$ and $\psi_{0,1}$ are also in the kernel of $I(\db,0)$ and $I(\db^*,0)$.
The same considerations for $\overline{\psi}$ imply that $\psi_{0,2}$ and $\psi_{0,1}$ are in the kernel of $I(\pa,0)$ and $I(\pa^*,0)$,
and thus $\db\pa \psi$ and $\pa\db^*\psi$ are bounded polyhomogeneous.  Furthermore,
$$
\Delta_{\db}\, \db\pa \psi = \db\pa \Delta_{\db}\psi= \db (\pa (\beta-\gamma))= \db \alpha=0.
$$
Thus, by Lemma~\ref{bi.5}, we have that $\db^*\db\pa\psi=0$, so that
$$
       \alpha= \pa \beta= \pa \db\, \db^*\psi.
$$
We can now set $\mu= \db^*\psi$, which gives the result.

Finally, if $\beta$ is polyhomogeneous and vanishes to positive order, then by Lemma~\ref{bi.5c}, $\db^*\psi$ is bounded
polyhomogeneous.  But $\Delta_{\db}(\db^* \psi)=\db^*\beta$, so applying Lemma~\ref{bi.5c} to the complex conjugate
of $\db^*\psi$ gives that $\pa \mu = \pa \db^* \psi$ vanishes to positive order.
\end{proof}

\begin{theorem}
There is a natural isomorphism
$$
        \WH^{p,q}(g_b,-\epsilon,M,E) \cong
L^2_b\cH^{p,q}_{-}(M;E) := L^2\cH^{p,q}(M;E) \oplus \db^*  \ker^{p,q+1}_{-} \Delta_{\db}.
$$
\label{bi.6}\end{theorem}
\begin{proof}
By Lemma~\ref{bi.5}, there is a well-defined map
$$
   \Psi: L^2_b\cH^{p,q}_{-}(M;E) \to   \WH^{p,q}(g_b,-\epsilon,M;E).
$$
If $u\in\rho^{-\epsilon}L^2_{\db}\Omega^{p,q}(M;E)$ represents a class in $\WH^{p,q}(g_b,-\epsilon,M;E)$, then by \eqref{bi.2a},
$$
       u = \db\,\db^* \zeta + \db^*\db \zeta + \gamma
$$
for some $\zeta\in \rho^{-\epsilon} H^2_b\Omega^{p,q}(M;E,g_b)$ and $\gamma\in L^2\cH^{p,q}(M;E)$.  Set $v=\db \zeta$.
Then $\db u=0$ implies $\db\, \db^*v=0$.  Since $\db v=0$ as well, we have that 
$$
       v\in \ker^{p,q+1}_- \Delta_{\db},
$$
and thus
$$
      u- \db\, \db^*\zeta= \gamma+ \db^* v \in L^2\cH^{p,q}(M;E) + \db^*  \ker^{p,q+1}_{-} \Delta_{\db}.
$$
This shows that the map $\Psi$ is surjective.

To show that $\Psi$ is injective, fix  $\gamma\in L^2\cH^{p,q}(M;E)$ and  $w\in \ker^{p,q+1}_{-} \Delta_{\db}$ and suppose that there exists
$\zeta\in \rho^{-\epsilon}L^2_b\Omega^{p,q}(M;E)$ such that
$$
    \gamma+ \db^* w= \db \zeta.
$$
We must show that $\gamma+ \db^*w=0$.  First compute
$$
      \| \gamma \|^2_{L^2_b}= \langle \gamma, \gamma\rangle_{L^2_b}= \langle \gamma, \db\zeta- \db^*w\rangle_{L^2_b}= \langle \db^*\gamma, \zeta\rangle_{L^2_b} - \langle \db\gamma, w\rangle_{L^2_b}=0,
$$
so $\gamma=0$ and $\db^*w = \db\zeta$.  By Lemma~\ref{bi.5}, $\db^*w \in \ker (\db+\db^*)$ is bounded polyhomogeneous.
Its restriction $u$ to $\pa\tM$ is an element of $F(\db+\db^*,0)$.  To show that $u=0$, we use the nondegeneracy of the
pairing \eqref{fn.2}. Namely, it suffices to show that for any $v\in F(\db+\db^*,0)$, we have $B(u,v)=0$.  But if $\widetilde{v}$
is a smooth extension of $v$, then using \eqref{fn.3}, we find that 
\begin{equation*}
\begin{aligned}
B(u,v) & = \langle (\db+\db^*)\db^* w, \widetilde{v}\rangle_{L^2_b} - \langle \db^* w, (\db+\db^*)\widetilde{v}\rangle_{L^2_b},   \\
 & = - \langle \db^* w, (\db+\db^*)\widetilde{v}\rangle_{L^2_b}, \\
 & =- \langle  w, \db\,\db\widetilde{v}\rangle_{L^2_b}- \langle  \zeta, \db^*\db^*\widetilde{v}\rangle_{L^2_b}=0.
\end{aligned}
\end{equation*}
This shows that $u=0$, and so $\db^*w$ vanishes to positive order.
This justifies the final integration by parts
$$
    \| \db^*w \|^2_{L^2_b}= \langle \db^*w, \db^*w\rangle_{L^2_b}= \langle \db^*w, \db\zeta\rangle_{L^2_b}= \langle w,\db\,\db \zeta\rangle_{L^2_b}=0,
$$
so that $\db^*w\equiv 0$.

\end{proof}

\section{Polyhomogeneity at infinity for asymptotically cylindrical Calabi-Yau metrics}  \label{accy.0}
Let $\bM$, $\bD$, $g_b$ and $\omega_b$ be as in \S~\ref{ht.0}.  We now prove the main regularity result for the
complex Monge-Ampère equation.

\begin{theorem}
Let $F$ be a positive index set.  If there exists $f\in \cA^F_{\phg}(\tM)$ and $u\in \rho^\epsilon\CI_b(M)$ for some
$\epsilon > 0$ such that
\begin{equation}
           \frac{ (\omega_b +i\pa \db u)^n}{\omega_b^n} = e^f,
\label{accy.1a}\end{equation}
then $u\in\cA^G_{\phg}(\tM)$ for some positive index set $G$.
\label{accy.1}\end{theorem}
\begin{proof}
It suffices to show that for each $k\in \bbN$, there is a real index set $G^k$ and $u_k\in \cA^{G^k}_{\phg}(\tM)$ such that
\begin{equation}
       u-u_k\in \rho^{\frac{k\epsilon}2+\delta}\CI_b(M),
\label{accy.2}\end{equation}
where $\delta=\frac{\epsilon}{4}$.

For $k=1$, it suffices to take $u_1=0$ and $G^1=\emptyset$.  Suppose then that we already have a real index set $G^k$ and
$u_k\in \cA^{G^k}_{\phg}(\tM)$ such that \eqref{accy.2} holds.  We must find $G^{k+1}$ and $u_{k+1}$ so that \eqref{accy.2}
holds with $k$ replaced by $k+1$.

Since $u_k\in \rho^{\epsilon}\CI_b(M)$, we can replace $u_k$ by $\chi(\frac{\rho}{r})u_k$, where $\chi\in \CI_c([0,\infty))$ is
a cut-off function with $\chi(t)=1$ for $t<1$ and $r \ll 1$, so as to make the $\cC^2_{b}(M)$-norm of  $u_k$ small enough to ensure that
$$
                  \frac{\omega_b}{2} < \omega_b+ i\pa\db u_k < 2 \omega_b.
$$
Thus $\omega_{b,k}:= \omega_b +i\pa \db u_k$ is also the Kähler form of an exact polyhomogeneous $b$-metric.
By our inductive hypothesis, $v_k:= u-u_k\in\rho^{\frac{k\epsilon}{2}+\delta} \CI_b(M)$ satisfies the equation
\begin{equation}
                \frac{(\omega_{b,k}+ i\pa\db v_k)^n}{\omega_{b,k}^n}= e^{f_k}, \quad \mbox{with} \; f_k= f+ \log\left(\frac{\omega_b^n}{\omega^n_{b,k}} \right).
\label{accy.3}\end{equation}

Since $f\in \cA^G_{\phg}(\tM)$ and $u_k\in \cA^{G^k}_{\phg}(\tM)$, we see that $f_k\in \cA^{\tilde{G}^k}_{\phg}(\tM)$ for some real index set
$\tilde{G}^k$.  Moreover, by \eqref{accy.2},
$$
     f_k= f+ \log\left(\frac{\omega_b^n}{\omega^n_{b,k}} \right)=  \log\left( \frac{ (\omega_b +i\pa \db u)^n}{\omega_b^n}  \right)    -\log\left(\frac{\omega_{b,k}^n}{\omega^n_{b}} \right) \in \rho^{\frac{k\epsilon}{2}+\delta}\CI_b(M).
$$
Now write \eqref{accy.3} as
\begin{equation}
   1 + \Delta_{\omega_{b,k}} v_k + \sum_{j=2}^n N^{\omega_{b,k}}_j(v_k)= e^{f_k},
\label{accy.4}\end{equation}
where $\Delta_{\omega_{b,k}}$ is the $\db$-Laplacian associated to the Kähler form $\omega_{b,k}$, that is, half the Laplacian associated to the corresponding Riemannian metric, and where
$$
      N^{\omega_{b,k}}_j(h)= \frac{n!}{(n-j)! j!} \left( \frac{ \omega^{n-j}_{b,k}\wedge (i\pa\db h)^j}{\omega^n_{b,k}} \right),  \quad h\in \CI_b(M).
$$
From \eqref{accy.2}, we deduce that $N^{\omega_{b,k}}_j(v_k)\in \rho^{\frac{jk\epsilon}{2}+ j\delta}\CI_b(M)$, so that 
$$
        \Delta_{\omega_{b,k}} v_k  = w_k + (e^{f_k}-1),
$$
where $w_k\in \rho^{k\epsilon+ 2\delta}\CI_b(M)$ and $(e^{f_k}-1)\in \cA^{H^k}_{\phg}(\tM)\cap \rho^{\frac{k\epsilon}{2}+\delta}\CI_b(M)$
for some real index set $H^k$.  By Corollary~\ref{breg.9}, we can therefore find $h_{k+1}$ polyhomogeneous and in  $\rho^{\frac{k\epsilon}{2}+\delta}\CI_b(M)$ such that
$$
    v_k- h_{k+1} \in \rho^{k\epsilon+ \delta}\CI_b(M).
$$
Thus, we can take $u_{k+1}= u_k+ h_{k+1} \in \cA^{G^{k+1}}_{\phg}(\tM)$ where $G^{k+1}$ is  some real index set.  This completes the inductive step and the proof.
\end{proof}

This theorem shows that the Calabi-Yau $\AC$-metrics constructed in \cite{HHN2012} are polyhomogeneous
at infinity.  Let us first recall the construction of \cite{HHN2012}.

\begin{definition}
Let $\bM$ be a compact Kähler orbifold of complex dimension $n\ge 2$.  Let $\bD\in |-K_{\bM}|$ be an effective orbifold divisor satisfying the following two conditions:
\begin{itemize}
\item[(i)] The complement $M:= \bM\setminus \bD$ is a smooth manifold;
\item[(ii)] The orbifold normal bundle of $\bD$ is biholomorphic to $(\bbC\times D)/ \langle \iota\rangle$ as an orbifold line bundle, where $D$ is a connected complex manifold and $\iota$ is a complex automorphism of $D$ of order $m<\infty$ acting on the product via $\iota(w,x)= (e^{\frac{2\pi i}{m}}w,\iota(x))$.
\end{itemize}
Then if we pick a meromorphic $n$-form $\Omega$ on $\bM$ with a simple pole along $\bD$, the construction of \cite{Tian-Yau1990} and \cite{HHN2012} ensures that for every Kähler class $\mathfrak{t}$ on $\bM$, there exists a Calabi-Yau $\AC$-metric $g_{\CY}$ on $M$ with Kähler class $\omega_{\CY}$ such that $\omega_{\CY}\in \left. \mathfrak{t}\right|_{M}$ and $\omega_{\CY}^n= i^{n^2}\Omega\wedge \overline{\Omega}$.
We say that the Calabi-Yau manifold $(M,g_{\CY})$ of the above construction is a \textbf{compactifiable asymptotically cylindrical Calabi-Yau manifold} with compactification $\bM$.
\label{tt.1}\end{definition}

To obtain the uniqueness of such a Calabi-Yau metric, we need to better understand the role of the $n$-form $\Omega$ in the
construction of \cite{HHN2012}.  First notice that $\Omega$ corresponds to a holomorphic section of $K_{\bM}(\bD)$.
Since this bundle is holomorphically trivial by hypothesis, restriction to $\bD$ and the adjunction formula give a canonical identification
\begin{equation}
  H^0(\bM; K_{\bM}(\bD))= H^0(\bD; \left. K_{\bM}(\bD)\right|_{\bD})=H^0(\bD; K_{\bD}).
\label{un.1}\end{equation}
In other words, $\Omega$ corresponds to a choice of a holomorphic section of $K_{\bD}$.  Its restriction to
$\pa\tM$ yields a section $\Omega_{\pa \tM}\in \CI(\pa\tM; \Lambda^{0,n}(\left. {}^{b}T^*\tM\right|_{\pa \tM}))$.
Clearly then,
\begin{equation}
\omega_{\CY}^n= i^{n^2}\Omega\wedge \overline{\Omega}  \quad \Longleftrightarrow    \quad  \left.  \omega_{\CY}^{n} \right|_{\pa \tM} = i^{n^2}\Omega_{\pa\tM}\wedge \Omega_{\pa\tM}.
\label{un.2}\end{equation}
Thus the role of $\Omega$ is to impose a condition at infinity for the metric $g_{\CY}$.  Indeed, in the construction of \cite{HHN2012},
part of the behavior of $\omega_{\CY}$ at infinity is specified by requiring that the `pull-back' of $\omega_{\CY}$ to $\bD$ corresponds
to the Kähler form of the Calabi-Yau metric on $\bD$ associated to the Kähler class $\left. \mathfrak{t}\right|_{\bD}$.  From this,
\eqref{un.2} then completely determines $g_b$ in the direction conormal to $\bD$.  This suggests that one can describe this
condition at infinity directly without choosing a meromorphic $n$-form $\Omega$.
Let $c:\bD\times \Delta/\langle \iota\rangle \to \bM$ be a choice of smooth orbifold tubular neighborhood for $\bD$, where
$\Delta\subset \bbC$ is the unit disk. Let $q: \bD\times \Delta\to \bD\times \Delta/\langle \iota\rangle$ be the quotient map.
The existence and uniqueness results of  \cite[Theorem~D and Theorem~E]{HHN2012} can then be
combined into the following.
\begin{theorem}[Haskins-Hein-Nordström]
Let $\bM$ and $\bD$ be as in Definition~\ref{tt.1}.  For any choice of Kähler class $\mathfrak{t}$ on $\bM$ and any $\lambda>0$,
there exists a unique asymptotically cylindrical Calabi-Yau metric $g_{\CY}$ on $M$ with Kähler form $\omega_{\CY}$ such that
$[\omega_{\CY}]=\left.\mathfrak{t}\right|_{M}$ and
\begin{equation}
          g_{\CY} - c_*q_*\left( g_{\bD} + \lambda \frac{dw\odot d\overline{w}}{|w|^2}\right) \in \rho^{\delta}\CI_b(M;  \mathrm{Sym}^2(T^*M))
\label{cycy.1}\end{equation}
for some $\delta>0$, where $g_{\bD}$ is the Calabi-Yau metric on $\bD$ associated to the Kähler class $\left. \mathfrak{t}\right|_{\bD}$.
\label{un.3}\end{theorem}

\begin{corollary}
The asymptotically cylindrical Calabi-Yau metric of the previous theorem is in fact a polyhomogeneous exact b-metric.
\label{accy.5}\end{corollary}
\begin{proof}
The existence of $\omega_{\CY}$ as an element of $\CI_b(M; \Lambda^2(T^*(M\setminus \pa M)))$ is obtained in \cite{HHN2012} by finding for $\epsilon>0$ small enough a solution
$u\in \rho^{\epsilon}\CI(M)$  of the complex Monge-Ampère equation \eqref{accy.1a} with $\omega_b$ the Kähler form of a carefully chosen exact $b$-metric and with
$$
      f= \log\left(  \frac{i^{n^2}\Omega\wedge \overline{\Omega}}{\omega_b^n}\right) \in \rho\CI(\tM).
$$
The polyhomogeneity of $\omega_{\CY}= \omega_b +i\pa \db u$ then follows from Theorem~\ref{accy.1}.
\end{proof}

\section{Polyhomogeneity at infinity for asymptotically conical Calabi-Yau metrics}
Another important class of complete noncompact quasi-projective Calabi-Yau spaces are those which are asymptotically
conical at infinity.  These are conformal to asymptotically cylindrical metrics, so essentially the same techniques as
above can be used to prove their polyhomogeneity.  Since this is only a slight detour, we carry this out here.

We begin with a more careful definition of asymptotically conical metrics. Once again, let $\tM$ be a compact manifold
with connected boundary $\pa \tM$.  Fix a collar neighborhood of the boundary described by some
diffeomorphism $c:\pa \tM \times [0,\nu) \hookrightarrow \tM$. The projection $\pr_{R}: \pa \tM\times [0,\nu)\to [0,\nu)$
determines a boundary defining function $\rho$ in this neighborhood, which we then extend smoothly to all of $\tM$.
Having fixed $\rho$, consider the Lie algebra of {\bf scattering} vector fields,
\begin{equation}
  \cV_{\sc}(\tM)= \{  \xi \in \CI(\tM;T\tM) \; | \;  \xi \rho\in \rho^2\CI(\tM;T\tM) \}.
\label{sc.1}\end{equation}
This is a Lie subalgebra of $\cV_b(\tM)$ and its definition depends on the choice of $\rho$.  As for $b$-vector fields,
there is an associated \textbf{scattering tangent bundle} ${}^{\sc}T\tM$ with 
$$
            {}^{\sc}T_p\tM = \cV_{\sc}(\tM)/ I_p\cV_{\sc}(\tM),  \quad I_p=\{f\in \CI(\tM) \; | \; f(p)=0\}.
$$
There is a canonical morphism $\iota_{\sc}: {}^{\sc}T\tM\to T\tM$ such that $(\iota_{\sc})_*\CI(\tM;{}^{\sc}T\tM)=
\cV_{\sc}(\tM) \subset \CI(\tM;T\tM)$, inducing on ${}^{\sc}T\tM$ the structure of a Lie algebroid with anchor map
$(\iota_{\sc})_*$.  Just as for $\iota_b$, $\iota_{\sc}$ is only an isomorphism when restricted to the interior of $\tM$.

There is a space of scattering differential operators, $\Diff_{\sc}^*(\tM)$, where an element of order $k$ is generated by $\CI(\tM)$ and
products of up to $k$ $\sc$-vector fields.  We can also consider the space of polyhomogeneous  scattering differential operators
$$
    \Diff^k_{\sc,F}(\tM)= \cA^{F}_{\phg}(\tM)\otimes_{\CI(\tM)} \Diff^k_{\sc}(\tM).
$$

\begin{definition}
A \textbf{scattering metric} $g$ on $M=\tM\setminus \pa \tM$ is a complete Riemannian metric on $M$ of the form
$$
            g= (\iota_{\sc}^{-1})^* g_{\sc}
$$
for some positive definite section $g_{\sc}\in \CI(\tM; \mathrm{Sym}^2({}^{\sc}T^*\tM))$.  It is a \textbf{warped product}
scattering metric if in the collar neighborhood,
\begin{equation}
c^*g= \frac{d\rho^2}{\rho^4}+ \frac{g_{\pa \tM}}{\rho^2}
\label{acocy.2}\end{equation}
where $g_{\pa \tM}$ is a metric on $\pa \tM$, and it is \textbf{exact} if
$g-g_p\in \rho\CI(\tM; \mathrm{Sym}^2({}^{\sc}T^*\tM))$ for some warped product scattering metric $g_p$.
\label{sc.1}\end{definition}
If $g$ is a scattering metric, then $g_b=\rho^2g$ is a $b$-metric; with this correspondence, warped product and
exact scattering metrics correspond to product and exact $b$-metrics.  Under the change of variable $t=1/x$,
we recognize a warped product scattering metric as an exact conical metric
$$
         dt^2+ t^2g_{\pa\tM},  \quad  t>\frac{1}{\nu}.
$$
More generally, \cf \cite{CH2013}, a complete metric $g$ on $M$ is an \textbf{asymptotically conical metric} on $M$  if there is a choice of collar neighborhood, compatible boundary defining function $\rho$,
and warped product scattering metric $g_{p}$ such that
$$
             g-g_p\in \rho^{\delta}\CI_{b}(\tM; \mathrm{Sym}^2({}^{\sc}T^*\tM)) \quad \mbox{for some} \quad \delta > 0.
$$
An asymptotically conical metric $g$ is called a \textbf{polyhomogeneous scattering metric} if \linebreak $g\in \cA^{F}_{\phg}(\tM;
\mathrm{Sym}^2({}^{\sc}T^*\tM))$ for some $F \geq 0$. Notice that the exactness condition is assumed.

Let $\Delta_{\sc}$ be the Laplacian (with negative spectrum) associated to
$g_{\sc}\in\cA^F_{\phg}(\tM;\mathrm{Sym}^2({}^{\sc}T^*\tM))$.  In the collar neighborhood, and for some $\delta>0$,
$$
\Delta_{\sc}-\rho^{2}\left( \Delta_{\pa \tM} +\left(\rho\frac{\pa}{\pa \rho}\right)^2 - (n-2)\rho\frac{\pa}{\pa \rho}\right)
\in \rho^{\delta} \Diff^2_{\sc,F'}(\tM) \subset \rho^{2+\delta} \Diff^2_{b,F'}(\tM),
$$
where $F'\geq 0$. In particular, $A := \rho^{-2}\Delta_{\sc} \in \Diff^2_{b,F'}(\tM)$ is an elliptic $b$-operator
with indicial family
$$
   \hat{A}(\tau)= \Delta_{\pa \tM} -\tau^2 -i(n-2)\tau.
$$
We may thus apply Corollary~\ref{breg.9} directly to obtain the following.
\begin{corollary}
Let $g_{\sc}$ be a polyhomogeneous scattering metric.  If $u\in \rho^{\alpha}\CI_{b}(M)$ satisfies
$$
    \Delta_{\sc} u = \rho^2 (f_1+ f_2)\quad \mbox{with} \; f_1\in \rho^{\alpha+\beta}\CI_b(M), \quad  f_2\in \cA^G_{\phg}(\tM_q),
$$
for some $\beta > 0$ and some index set $G$ with $\inf G > \alpha$, then $u=u_1+u_2$ with $u_1\in \bigcap_{\delta>0}
\rho^{\alpha+\beta -\delta}\CI_b(M)$ and $u_2$ polyhomogeneous.
\label{screg.3}\end{corollary}

We now turn to the complex Monge-Amp\`ere equation.  Suppose that $\tM\setminus \pa \tM$ is a complex manifold and
that the complex structure $J$ extends to an element $J\in\cA^Q_{\phg}(\tM;\End({}^{\sc}T\tM))$ for some $Q\geq 0$.
Suppose $g_{\sc}$ is a polyhomogeneous scattering metric which is Kähler with respect to $J$, and has Kähler form $\omega_{\sc}$.

\begin{theorem}
Let $F$ be a positive index set.  If $f\in \cA^F_{\phg}(\tM)$ and for some $\epsilon>0$, $u\in \rho^{\epsilon-2}\CI_b(M)$ satisfies
\begin{equation}
           \frac{ (\omega_{\sc} +i\pa \db u)^n}{\omega_{\sc}^n} = e^f,
\label{acocy.6a}\end{equation}
then $u\in\cA^{G-2}_{\phg}(\tM_q)$ for some $G > 0$.
\label{acocy.6}\end{theorem}
\begin{proof}
The strategy is the same as in the proof of Theorem~\ref{accy.1}, with small variations.

It suffices to show that for each $k\in \bbN$, there is a positive index set $G^k$ and $u_k\in \cA^{G^k-2}_{\phg}(\tM)$ such that
\begin{equation}
       u-u_k\in \rho^{\frac{k\epsilon}2-2+\delta}\CI_b(M),
\label{acocy.7}\end{equation}
where $\delta=\frac{\epsilon}{4}$.  For $k=1$, we take $u_1=0$ and $G^1=\emptyset$.

Suppose that \eqref{acocy.7} holds for some $u_k\in \cA^{G^k-2}_{\phg}(\tM_q)$ with $G^k> 0$; we must show that
\eqref{acocy.7} holds at the next level.  Just as before, replace $u_k$ by $\chi(\frac{\rho}{r})u_k$ with $r \ll 1$
to make the $\rho^{-2}\cC^2_{b}(M)$-norm of $u_k$ small enough so that
$$
                  \frac{\omega_{\sc}}{2} < \omega_{\sc}+ i\pa\db u_k < 2 \omega_{\sc}.
$$
Thus, $\omega_{\sc,k}:= \omega_{\sc} +i\pa \db u_k$ is the Kähler form of a polyhomogeneous $\sc$-metric.  By our inductive
hypothesis, $v_k:= u-u_k\in\rho^{\frac{k\epsilon}{2}-2+\delta} \CI_b(M)$ satisfies
\begin{equation}
\frac{(\omega_{\sc,k}+ i\pa\db v_k)^n}{\omega_{\sc,k}^n}= e^{f_k}, \quad \mbox{with} \; f_k= f+
\log\left(\frac{\omega_{\sc}^n}{\omega^n_{\sc,k}} \right).
\label{acocy.8}\end{equation}
Since $f\in \cA^F_{\phg}(\tM)$ and $u_k\in \cA^{G^k}_{\phg}(\tM)$, we see that $f_k\in \cA^{\tilde{G}^k}_{\phg}(\tM_q)$ for some
$\tilde{G}^k > 0$.  Moreover, by \eqref{acocy.7},
$$
f_k= f+ \log\left(\frac{\omega_{\sc}^n}{\omega_{\sc,k}} \right)=  \log\left( \frac{ (\omega_{\sc} +i\pa \db u)^n}{\omega_{\sc}^n}  \right)    -\log\left(\frac{\omega_{\sc,k}^n}{\omega_{\sc}} \right)
$$
is in $\rho^{\frac{k\epsilon}{2}+\delta}\CI_b(M)$.  Now rewrite \eqref{acocy.8} as
\begin{equation}
   1 + \Delta_{\omega_{\sc,k}} v_k + \sum_{j=2}^n N^{\omega_{\sc,k}}_j(v_k)= e^{f_k},
\label{acocy.9}\end{equation}
where $\Delta_{\omega_{\sc,k}}$ is the $\db$-Laplacian associated to the Kähler form $\omega_{\sc,k}$,
and where
$$
N^{\omega_{\sc,k}}_j(h)= \frac{n!}{(n-j)! j!} \left( \frac{ \omega^{n-j}_{\sc,k}\wedge (i\pa\db h)^j}{\omega_{\sc,k}} \right),
\quad h\in \rho^{-2}\CI_b(M).
$$
By \eqref{acocy.7}, we deduce that $N^{\omega_{\sc,k}}_j(v_k)\in \rho^{\frac{jk\epsilon}{2}+ j\delta}\CI_b(M)$, so that
$$
        \Delta_{\omega_{\sc,k}} v_k  = w_k + (e^{f_k}-1),
$$
where $w_k\in \rho^{k\epsilon+ 2\delta}\CI_b(M)$ and $(e^{f_k}-1)\in \cA^{H^k}_{\phg}(\tM)\cap \rho^{\frac{k\epsilon}{2}+\delta}\CI_b(M)$
for some $H^k \geq 0$.  By Corollary~\ref{screg.3}, we can find a polyhomogeneous function
$h_{k+1}\in \rho^{\frac{k\epsilon}{2}-2+\delta}\CI_b(M)$ such that
$$
    v_k- h_{k+1} \in \rho^{k\epsilon-2+ \delta}\CI_b(M).
$$
Now take $u_{k+1}= u_k+ h_{k+1} \in \cA^{G^{k+1}-2}_{\phg}(\tM_q)$ for some positive index set $G^{k+1}$.  Since
$u-u_{k+1}\in  \rho^{k\epsilon-2+ \delta}\CI_b(M)$, we have in particular that
$$
    u_{k+1}= (u_{k+1}-u) + u \in \rho^{\epsilon-2}\CI_b(M),
$$
which completes the inductive step and the proof.
\end{proof}

This result can be used to prove the polyhomogeneity at infinity of the asymptotically conical Calabi-Yau metrics
which come from the construction of Tian-Yau \cite[Corollary~1.1]{Tian-Yau1991} and its refinement and generalization
 \cite[Theorem~A]{CH2014}. Let $\bM$ be a compact Kähler orbifold of complex dimension $n>1$ without $\bbC$-codimension $1$
singularities.  Let $D$ be a suborbifold divisor of $\bM$ containing all the singularities of $\bM$ such that $-K_{\bM}=q[D]$
with $q \in\bbN$ and $q>1$.   Suppose that $D$ admits a Kähler-Einstein metric with positive scalar curvature.  By the
orbifold Calabi ansatz \cite{Calabi}, \cite[Proposition~3.1]{LeBrun}, there exists a Calabi-Yau cone structure $h$ on $K_D\setminus 0$.
Using the $(q-1)$-covering map $\eta: N_D\setminus 0 \to K_D\setminus 0$ induced by the adjunction formula
$N_D^{q-1} \cong K_D^{-1}$ and a choice of meromorphic volume form $\Omega$ on $\bM$ with pole of order $q$ along $D$, the pullback of $h$ is a Calabi-Yau cone metric $g_0$ on $N_D\setminus 0$ with apex at 
infinity.  Consider the real blow-up  $\tM=[\bM;D]$ of $\bM$; this is a smooth manifold with boundary. We can then write
$$
               g_0 = \frac{dx^2}{x^4} + \frac{h}{x^2},
$$
where $x= \rho^{\frac{q-1}n}$ for some boundary defining function $\rho\in\CI(\tM)$. Thus $g_0$ is a scattering metric in terms
of this new defining function. What is actually happening here is that we are replacing the original smooth manifold with boundary
$\tM$ by a new one, $\tM_{\frac{q-1}n}$, where the (equivalent) $\CI$ structure is the one obtained by adjoining this new defining function, or
equivalently, by pulling back the original $\CI$ structure under the obvious homeomorphism. In this new structure, smooth
functions on $\tM$ have Taylor expansions in nonnegative integral powers of $x$ rather than $\rho$, etc.
Notice, however, that a function which is polyhomogeneous in the new structure is polyhomogeneous in the original
structure, and vice versa, and the notions of positivity and nonnegativity of index sets remain the same, even though the
index sets themselves transform.
\begin{corollary}
If $\mathfrak{t}$ is a Kähler class on $M=\bM\setminus D$ and $c>0$, then there exists a unique Calabi-Yau polyhomogeneous
scattering metric $g_{\CY}$ on $\tM_{\frac{q-1}{n}}$ in the Kähler class $\mathfrak{t}$ with
$$
g_{\CY}-\exp_{*}(cg_0) \in x^{\delta}\CI(\tM_{\frac{q-1}n}; \mathrm{Sym}^2({}^{\sc}T^*\tM_{\frac{q-1}n}))
$$
for some $\delta>0$, where $\exp: N_D\to \bM$ is the exponential map of any background Hermitian metric on $\bM$.
\label{acocy.11}\end{corollary}
\begin{proof}
By assumption, there exists a meromorphic volume form $\Omega$ on $\bM$ with a pole of order $q$ at $D$, so in particular,
$\Omega$ defines a polyhomogeneous scattering volume form on $\tM_{q-1}$. It is shown in \cite[Proposition~2.1]{CH2014} that
\begin{equation}
\Omega- (-1)^{n}(q-1)^{-1}\exp_{*}(\eta^{*}\Omega_0)\in x^{\frac{n}{q-1}}\CI_{b}(\tM_{\frac{q-1}n}; \Lambda^n({}^{\sc}T^*\tM_{\frac{q-1}n})),
\label{acocy.11a}\end{equation}
where $\Omega_{0}$ is the tautological holomorphic volume form on $K_{D}$. In addition, it is proven that the Kähler class $\mathfrak{t}$ (and indeed, any Kähler class on $M$,) can be represented by a \textbf{smooth} real $(1,1)$-form $\xi$ on $\tM$.  This shows in particular that $\mathfrak{t}$ is $(-2)$-almost compactly supported in the sense
of \cite[Definition~2.3]{CH2013}.  From \cite[Proof of Theorem~2.4]{CH2013}, one can then construct an asymptotically conical Kähler
metric $g_{\sc}$ in the Kähler class $\mathfrak{t}$ with Kähler form $\omega_{\sc}$ such that
$$
     \omega_{\sc}= \xi + c(i/2)\partial\db(\exp_{*}r^{2})
$$
in a neighborhood of the boundary of $\tM$, where $r$ is the radial function of the metric $g_0$.  Since $\xi$ is smooth on $\bM$, it is
in particular polyhomogeneous on $\tM_{\frac{q-1}n}$, so that $g_{\sc}$ is in fact a polyhomogeneous exact scattering metric with
\begin{equation}
g_{\sc} - \exp_{*}(cg_0) \in x^\delta \CI_b (\tM_{\frac{q-1}n} ; \mathrm{Sym}^2( {}^{\sc} T^* \tM_{\frac{q-1}n} ) )
\label{decay.1}
\end{equation}
for some $\delta>0$.

The existence and uniqueness of $g_{\CY}$ with Kähler form $\omega_{\CY}= \omega_{\sc}+ i\pa\db u$ is obtained in \cite{CH2013}
by showing that the complex Monge-Ampère equation
\begin{equation}
           \frac{ (\omega_{\sc} +i\pa \db u)^n}{\omega_{\sc}^n} = e^f,  \quad \mbox{with}  \quad f= \log\left( \frac{i^{n^2}\Omega\wedge \overline{\Omega}}{\omega_{\sc}^n} \right),
\label{acocy.12}\end{equation}
has a unique solution $u\in \rho^{\delta-2}\CI_b(M)$ for some $\delta>0$.  By \eqref{acocy.11a} and \eqref{decay.1},
$f\in \rho^{\delta}\CI_b(M; \Lambda^n({}^{\sc}T^*\tM))$.  Since $\Omega$ and $\omega_{\sc}$ are polyhomogeneous, $f$
is also polyhomogeneous.  Thus, the polyhomogeneity of $u$ and $\omega_{\CY}$ follows from Theorem~\ref{acocy.6}.
\end{proof}

Theorem \ref{acocy.6} can also be used to show that the asymptotically conical Calabi-Yau metrics of Goto \cite{Goto} and van Coevering \cite{vanC} on a crepant resolution of an irregular Calabi-Yau cone are polyhomogeneous. Indeed, let $C=L\times\mathbb{R}_{+}$ 
be an irregular Calabi-Yau cone of dimension $n$ with Calabi-Yau cone metric $g_{0}$, associated K\"ahler form $\omega_{0}$, and holomorphic volume form $\Omega_{0}$ normalized so that $\omega_{0}^{n}=i^{n^{2}}\Omega_{0}\wedge\overline{\Omega}_{0}$. Furthermore, let $p:C\to L$ denote the radial projection and let $\pi:M\to C$ be any crepant resolution, so that the holomorphic volume form $\pi^{*}\Omega_{0}$ extends to a holomorphic volume form $\Omega$ on $M$. For a given $c>0$, an asymptotically conical K\"ahler metric $g_{\sc}$ can be constructed in each K\"ahler class $\mathfrak{t}$ of $M$ whose K\"ahler form $\omega_{\sc}$ can be written as $$\omega_{\sc}=\pi^*p^*\alpha+c\pi^*\omega_{0}$$ outside some compact subset of $M$, for some closed primitive
basic $(1, 1)$-form $\alpha$ on $L$; see \cite[Lemma 5.7]{Goto} and also \cite[Section 4.2]{CH2013}. We compactify $M$ as a manifold with boundary $\tM$ using the boundary defining function $x=(\pi^{*}r)^{-1}$, where $r$ is the radial coordinate of the cone metric $g_{0}$. Then, since both $M$ and $C$ are biholomorphic away from the exceptional set of the resolution, and since the form $\pi^*p^*\alpha$ clearly extends to $\partial\tM$, we see that the metric $g_{\sc}$ is a polyhomogeneous exact scattering metric with
\begin{equation*}
g_{\sc} - \pi_{*}(cg_0) \in x^\delta \CI_b (\tM ; \mathrm{Sym}^2( {}^{\sc} T^* \tM ) )
\end{equation*}
for some $\delta>0$. As before, the existence of an asymptotically conical Calabi-Yau metric $g_{\CY}$ with K\"ahler form $\omega_{\CY}=\omega_{\sc}+i\partial\bar{\partial}u$ is obtained by showing that the complex Monge-Ampère equation
\begin{equation*}
           \frac{ (\omega_{\sc} +i\pa \db u)^n}{\omega_{\sc}^n} = e^f,  \quad \mbox{with}  \quad f= \log\left(\frac{(\pi^{*}\omega_{0})^{n}}{\omega_{\sc}^n} \right),
\end{equation*}
has a unique solution $u\in x^{\delta-2}\CI_b(M)$ for some $\delta>0$. Polyhomogeneity of $u$, and hence $\omega_{\CY}$, then follows from Theorem \ref{acocy.6} using the fact that $f$ is polyhomogeneous because $\omega_{\sc}$ is.

As for the asymptotically conical Calabi-Yau metrics of \cite[Theorem C]{CH2014} with irregular tangent cone at infinity, one can show that they too are polyhomogeneous.  In this example, the irregular cone is $C= K_D\setminus 0$ with $D=\bbC\bbP^2_{p_1,p_2}$ the blow-up of $\bbC\bbP^2$ at two points.  The asymptotically conical Calabi-Yau metric is constructed on $M=\bM\setminus D$ with $\bM=\bbC\bbP^3_p$ the blow-up of $\bbC\bbP^3$ at one point, where $D\in|-\frac12 K_{\bM}|$ is seen as the strict transform of a smooth quadric passing through $p$.  By \cite{FOW2009}, we know that $C=K_D\setminus 0$ admits an irregular Calabi-Yau cone metric $g_0$ with apex at the zero section.  The Calabi-Yau metric of  \cite[Theorem C]{CH2014} is then constructed using a very careful choice of exponential type map $\exp: N_D\to \bM$.  Notice however that this map does not provide  the right gauge to establish the polyhomogeneity of the metric, since it introduces non-polyhomogeneous terms in the complex Monge-Ampère equation used to construct the metric.  In fact, in \cite{CH2014}, a better diffeomorphism $\Phi: M\setminus K_1\to C\setminus K_2$ for some compact sets $K_1\subset M$ and $K_2\subset C$ is obtained using a gauge fixing argument as in \cite{Cheeger-Tian}.  With this identification, we get a compactification $\tM$ of $M$ such that $g:=\Phi^*g_0$ is a scattering metric and such that $\rho=\frac{1}{\Phi^*r}$, with $r$ the radial function of $(C,g_0)$, is a boundary defining function near $\pa\tM$.  With respect to the metric $g$, the  Calabi-Yau metric $g_{\CY}= g+h$ of  \cite[Theorem C]{CH2014} satisfies the elliptic quasi-linear equation
\begin{equation}
  \Ric(g_0+h)_{ij}+ (\nabla_i\mathfrak{B}_{g}(h)_j+ \nabla_j\mathfrak{B}_{g}(h)_i)=0 \quad \mbox{with} \; h\in \rho^{\delta}\CI_b(\tM; {}^{\sc}T^*\tM\otimes {}^{\sc}T^*\tM ) \; \mbox{for} \; \delta=0.0128>0,
\label{ic.1}\end{equation} 
where $\mathfrak{B}_g= \diver_g(h -\frac12 \tr_g(h)g)$ is the operator appearing in the Bianchi gauge condition and $\nabla$ is the Levi-Civita connection of $g$.  Using \cite[Lemma~1.6]{CH2013}, one can put this equation in the form
\begin{equation}
  Ph= F_0(h)\cdot h^2 + F_1(h)\cdot \left( \frac{h \nabla h}{\rho}\right) + F_2(h) \cdot \left( \frac{\nabla h }{\rho}\right)^2,
\label{ic.2}\end{equation}
where $P$ is an elliptic $b$-operator, $F_i:  \mathrm{Sym}^2({}^{\sc}T^*\tM)\to ({}^{\sc}T\tM)^i\otimes  \mathrm{Sym}^2({}^{\sc}T\tM)$, for $i=0,1,2$ are smooth maps mapping sections to sections, but not linearly, and ``$\cdot$" denotes some contraction of indices.  In particular, the right hand side of \eqref{ic.2} is in $\rho^{2\delta}\CI_b(\tM; {}^{\sc}T^*\tM\otimes {}^{\sc}T^*\tM )$.  Using Corollary~\ref{breg.9}, we can then apply a bootstrapping argument as in the proof of Theorem~\ref{accy.1} to conclude that the metric $g_{\CY}$ of \cite[Theorem~C]{CH2014} is in fact a polyhomogeneous exact scattering metric for the boundary compactification $\tM$.

\section{Deformations of compactifiable asymptotically cylindrical Calabi-Yau manifolds}  \label{tt.0}
We henceforth fix a compactifiable, asymptotically cylindrical Calabi-Yau manifold $(M, g_b)$ with compactification $(\bM,\bD)$.
To describe the complex  deformations of $M$, we appeal to the deformation theory of compactifiable complex manifolds
developed by Kawamata \cite{Kawamata1978}. There is a Kuranishi type theorem in this context.  In our setting, however,
the existence of a Calabi-Yau metric makes it possible to obtain a sharper result,
namely that the deformation theory is unobstructed.

First, recall from  \cite{Kawamata1978} that the infinitesimal complex deformations of $M$, as a compactifiable complex manifold,
are given by $H^1(\bM; T_{\bM}(\log \bD))$, where $T_{\bM}(\log \bD)$ is the logarithmic tangent sheaf.  Given a  Dolbeault representative
$\phi_1\in \Omega^{0,1}(\bM;T_{\bM}(\log\bD))$ of a class $[\phi_1]\in H^1(\bM; T_{\bM}(\log\bD))\cong H^{0,1}_{\db}
(\bM; T_{\bM}(\log \bD))$, the first step in `integrating' $\phi_1$ to an actual deformation is to solve the problem formally.
In other words, we wish to construct a possibly non-convergent series
\begin{equation}
   \phi(t) \sim \sum_{i=1}^{\infty}  \phi_i t^i,  \quad t\in\bbC,
\label{tt.2}\end{equation}
term by term so that the new formal $\db$-operator $\db+\phi(t)$ satisfies the Maurer-Cartan equation $\db \phi(t)+
\frac12 [\phi(t),\phi(t)]=0$ in the sense of Taylor series.  This equation is the one which indicates whether this $\db$
operator is integrable, i.e., corresponds to a new complex structure.  In terms of the coefficients of the power series \eqref{tt.2},
the Maurer-Cartan equation implies the sequence of equations
\begin{equation}
\db \phi_k= -\frac12 \sum_{i< k} [\phi_i, \phi_{k-i}].
\label{tt.3}\end{equation}
When $k=1$, this states simply that $\db \phi_1=0$, which is automatic by definition of the Dolbeault cohomology group
$H^{0,1}_{\db}(\bM; T_{\bM}(\log \bD))$.  When $k=2$, this gives the equation
\begin{equation}
  \db \phi_2= -\frac12 [\phi_1 ,\phi_1 ].
\label{tt.4}\end{equation}
There is an obvious cohomological obstruction to solving this equation.  Indeed, $[\phi_1,\phi_1]$ represents a class in
$H^{0,2}_{\db}(\bM; T_{\bM}(\log\bD))$ and \eqref{tt.4} has a solution if and only if this class is trivial.  But in our case, as
we now explain, this obstruction always vanishes -- as do the obstructions inherent to solving \eqref{tt.3} for all higher
values of $k$. The proof takes advantage of the asymptotically cylindrical Calabi-Yau metric
$g_b$ and uses the same strategy of Tian and Todorov \cite{Tian1987, Todorov1989}; we refer also to \cite{Huybrechts} for a nice
introduction to the subject.

By using the meromorphic form $\Omega\in H^0(\bM; \Omega^n_{\bM}(\log\bD))$, we first define a sheaf isomorphism
\begin{equation}
   \eta:  \Lambda^p T_{\bM}(\log\bD)  \to \Omega^{n-p}_{\bM}(\log \bD),  \quad
\eta(v_1\wedge\ldots \wedge v_p)= \iota_{v_1}\ldots \iota_{v_p} \Omega.
\label{tt.5}\end{equation}
This induces an isomorphism
\begin{equation}
     \eta: \Omega_b^{0,q}(\tM, \Lambda^p ({}^bT^{1,0}\tM)) \to \Omega^{n-p,q}_b(\tM),
\label{tt.6}\end{equation}
which in turn can be used to define the $b$-operator
\begin{equation}
  T: \Omega^{0,q}_b(\tM;\Lambda^p({}^bT^{1,0}\tM))\to \Omega^{0,q}_b(\tM; \Lambda^{p-1} ({}^bT^{1,0}\tM)), \qquad
T= \eta^{-1}\circ \pa \circ \eta.
\label{tt.7}\end{equation}
It is not hard to check, see \cite{Huybrechts}, that $T$ anti-commutes with $\db$, namely
\begin{equation}
   T\circ \db= -\db \circ T.
\label{tt.8}\end{equation}
In addition, $T$ satisfies the following fundamental property.
\begin{lemma}[Tian,Todorov]
For $\alpha\in \Omega^{0,p}_b(\tM; T_{\bM}(\log\bD))$ and $\beta\in \Omega^{0,q}_b(\tM;T_{\bM}(\log \bD))$, we have that
$$
        (-1)^p [\alpha,\beta]= T(\alpha\wedge \beta) - T(\alpha)\wedge \beta- (-1)^{p+1} \alpha\wedge T(\beta).
$$
In particular, if $\alpha$ and $\beta$ are $T$-closed, then $[\alpha,\beta]$ is $T$-exact.
\label{tt.9}\end{lemma}
\begin{proof}
The proof is a local computation; see \cite{Huybrechts} for details.
\end{proof}

We are now ready to solve \eqref{tt.3} by induction on $k$.

\begin{proposition}
Let $(M,g_b)$ be a compactifiable asymptotically cylindrical Calabi-Yau manifold with compactification $\bM$.  Suppose that
$\phi_1\in L^2_b\cH^{0,1}_{-}(M; {}^bT^{1,0}\tM)$ represents an infinitesimal deformation.  Then there exists a formal power
series $\sum_{k=0}^{\infty} \phi_k t^k$  with
\begin{equation}
                   \db \phi_k = -\frac12 \sum_{i=1}^{k-1} [\phi_i, \phi_{k-i}],
\label{tt.13a}\end{equation}
where each $\phi_k$  is a bounded polyhomogeneous $(0,1)$-form with values in ${}^bT^{1,0}\tM$ such that
$\eta(\phi_k)= \pa \beta_k$ for some polyhomogeneous form $\beta_k$.
\label{tt.13}\end{proposition}
\begin{proof}
We first claim that if $\phi_1$ is harmonic, then $\eta(\phi_1)\in \rho^{-\epsilon} H^{\infty}_b\Omega^{n-1,1}(M)$ is harmonic as well.
Indeed, since $\Omega$ is holomorphic, $\db\circ \eta = \eta\circ \db$,  so that $\db \eta(\phi_1) = 0$.
Next, since $g_b$ is Calabi-Yau, $\eta$ is compatible (up to a constant scalar factor) with the Hermitian metrics on
${}^bT^{1,0}\tM$ and $\Lambda^{n,0}({}^bT^*\tM)$, so that $\db^* \circ \eta = \eta \circ \db^*$, and hence
$\db^* \eta(\phi_1) = 0$ as well.

Since $\phi_1$ is bounded polyhomogeneous, so is $\eta(\phi_1)$, so applying Lemma~\ref{bi.5} to $\eta(\phi_1)$ and its complex
conjugate shows that it is both $\db$-closed and $\pa$-closed. Hence, by Lemma~\ref{tt.9}, $\eta([\phi_1,\phi_1])$ is
$\pa$-exact and $\db$-closed, i.e., $\eta([\phi_1,\phi_1])= \pa \beta$ with
$\beta= -\eta(\phi_1\wedge \phi_1)$ bounded polyhomogeneous. 
By  Lemma~\ref{dd.1} (the $\pa\db$-lemma), we can find a polyhomogeneous $(n-1,1)$-form $\mu$ with $\pa \mu$ bounded such that
$$
        \eta[\phi_1,\phi_1]= \db\pa \mu.
$$
Thus, taking $\phi_2= -\frac12 \eta^{-1}\pa \mu$, we have that
$$
          \db\phi_2 + \frac12 [\phi_1,\phi_1]=0.
$$
Furthermore, $\eta(\phi_2)=-\frac12 \pa \mu$ is $\pa$-exact.

Suppose now that we have found $\phi_2,\ldots,\phi_{k-1}$ with the desired properties.  Then by the Tian-Todorov Lemma,
$$
        \eta[\phi_i, \phi_{k-i}]= -\pa \eta( \phi_i\wedge \phi_{k-i}),
$$
i.e., $\eta([\phi_i, \phi_{k-i}])$ is $\pa$-exact for $i<k$. Thus, $\sum_{i=1}^{k-1}[\phi_i,\phi_{k-i}]$ is $\pa$-exact.
It also $\db$-closed, since
\begin{equation}
\begin{aligned}
\db\left(\sum_{i=1}^{k-1}[\phi_i,\phi_{k-i}]  \right) &=  \sum_{i=1}^{k-1}\left( [\db\phi_i,\phi_{k-i}]- [\phi_i,\db\phi_{k-i}] \right)  \\
 &=  -\frac12 \sum_{i=1}^{k-1}\left(  \sum_{j=1}^{i-1}  [[\phi_j,\phi_{i-j}],\phi_{k-i}] - \sum_{\ell=1}^{k-i-1}   [\phi_i,[\phi_{\ell},\phi_{k-i-\ell}]  ] \right) \\
 &= -\frac12 \sum_{i=1}^{k-1} \sum_{j=1}^{i-1} \left(  [[\phi_j,\phi_{i-j}],\phi_{k-i}] -  [\phi_{k-i}, [\phi_{j}, \phi_{i-j}]]  \right),  \\
 &=   \sum_{i=1}^{k-1} \sum_{j=1}^{i-1} [\phi_{k-i},[\phi_j, \phi_{i-j}]].
 \end{aligned}
\label{tt.13b}\end{equation}
But this is precisely equal to the coefficient of $t^k$ in
$$
   [  \sum_{i=1}^{k-1} \phi_i t^i, [ \sum_{i=1}^{k-1} \phi_i t^i,  \sum_{i=1}^{k-1} \phi_i t^i]],
$$
and therefore vanishes by the Jacobi identity.  By Lemma~\ref{dd.1}, we can thus find a polyhomogeneous $(n-1,1)$-form
$\mu_k$ with $\pa \mu_k$ bounded such that $\db\pa\mu_k = -\frac12 \eta(\sum_{i=1}^{k-1} [\phi_i,\phi_{k-i}])$.
Now take $\phi_k=\eta^{-1} \pa \mu_k$ to complete the inductive step.
\end{proof}

To find actual deformation families, we wish to show that this formal series converges in
a suitable topology, and for this we must study the mapping properties of a generalized inverse of $\Delta_{\db}$.
Fix $\epsilon$ as in \eqref{defe} and consider the generalized inverse $G_{-\epsilon}$ of
\begin{equation}
        \Delta_{\db}: \rho^{-\epsilon}H^{k+2}_b\Omega^{p,q}(M) \to \rho^{-\epsilon}H^k_b\Omega^{p,q}(M)
\label{tt.13d}\end{equation}
in the sense of \cite[Proposition~5.64]{MelroseAPS} and \cite[Theorem~6.1]{MazzeoEdge}. Namely,
$$
G_{-\epsilon}: \rho^{-\epsilon}H^{k}_b\Omega^{p,q}(M) \to \rho^{-\epsilon}H^{k+2}_b\Omega^{p,q}(M)
$$
is the unique $b$-pseudodifferential operator of order $-2$ which satisfies
$$
              G_{-\epsilon}\Delta_{\db}= \Id- \Pi_1,  \quad \Delta_{\db} G_{-\epsilon} = \Id - \Pi_0,
$$
where $\Pi_1$ is the $\rho^{-\epsilon}L^2_b$-orthogonal projection onto $\ker^{p,q}_-\Delta_{\db}$ and
$$\Pi_0: \rho^{-\epsilon}L^2_b\Omega^{p,q}(M)\to L^2_b\cH^{p,q}(M) \hookrightarrow \rho^{-\epsilon}L^2_b\Omega^{p,q}(M)$$ is the
projection defined by
$$
      \Pi_0(u) = \sum_{i=1}^{\ell} \langle u, v_i\rangle_{L^2_b} v_i,
$$
where $v_1,\ldots, v_{\ell}$ is an orthonormal basis of $L^2\cH^{p,q}(M)$.

\begin{proposition}
For any $\delta  \in [0, \inf\cI_0)$,
$$
\pa \db^* G_{-\epsilon}: \rho^{\delta}\cC^{k,\alpha}_{g_b}(M; \Lambda^{p,q}(M))\to \rho^{\delta}\cC^{k,\alpha}_{g_b}(M; \Lambda^{p+1,q-1}(M))
$$
is a bounded operator.
\label{gf.1}\end{proposition}
\begin{proof}
We need to invoke some of the more technical aspects of the structure of the Schwartz kernel of this operator, for which we refer
to \cite{MelroseAPS}, \cite{MazzeoEdge} and also to \cite{Jeffres-Mazzeo-Rubinstein}, where a very similar but more complicated
result is proven.  First note that $\pa\db^* G_{-\epsilon}$ is a $b$-pseudodifferential operator of order zero.  A priori, its
Schwartz kernel could have a leading logarithmic term at order $\rho^0$ in its polyhomogeneous expansion at the left
boundary face $\lb(M^2_b)$ of the $b$-double space.  However, we rule this out by observing that this operator maps
polyhomogeneous forms with positive index sets to polyhomogeneous forms with positive index sets. To check this last
fact, note that if $\beta$ is polyhomogeneous with positive index set, then $\psi=G_{-\epsilon}\beta$ is polyhomogeneous and
$\Delta_{\db}\psi=\beta+\gamma$ with $\gamma \in L^2_b \cH^{p,q}(M)$. Thus, Lemma~\ref{bi.5c} implies that
$\db^*\psi$ is bounded polyhomogeneous.  Applying this lemma once more, this time to the complex conjugate of
$\db^*\psi$, we conclude that $\pa\db^*\psi$ is polyhomogeneous with positive index set.

This property implies that the index set $E_{\lb}$ of the polyhomogeneous expansion of the Schwartz kernel of
${\pa\db^*G_{-\epsilon}}$ is strictly positive; in fact, $\inf E_{\lb}\ge \inf\cI_0$. Now \cite[Proposition~3.27]{MazzeoEdge}
shows that 
\begin{equation}
     \pa\db^* G_{-\epsilon}:  \rho^{\delta}\cC^{k,\alpha}_{g_b}(M; \Lambda^{p,q}(M))\to \rho^{\delta}\cC^{k,\alpha}_{g_b}(M; \Lambda^{p+1,q-1}(M))
\label{tt.13e}\end{equation}
is bounded for $0\le \delta< \inf\cI_0$.
\end{proof}

We now define a function space slightly smaller than $\cC^{k,\alpha}_{g_b}(M; \Lambda^{p,q}(M))$ in which restriction to $\pa \tM$
makes sense.  Let $\chi\in \CI(M)$ equal $1$ near $\pa \tM$ and be supported in a collar neighborhood $c: \pa \tM\times
[0,1)\to \tM$ of $\pa \tM$, and let $\pi: \pa \tM\times [0,1)\to \pa \tM$ be the projection onto $\pa\tM$.
For $0<\delta<\inf \cI_0$, define
\begin{equation}
  \cC^{k,\alpha}_{0,\delta}(M; \Lambda^{p,q}(M)):= \chi c_* \pi^* \cC^{k,\alpha}(\pa\tM; \left.  \Lambda^{p,q}({}^bT^*\tM)\right|_{\pa\tM})+ \rho^{\delta}\cC^{k,\alpha}_{g_b}(M; \Lambda^{p,q}(M)).
\label{gf.2}\end{equation}
This is a subspace of $\cC^{k,\alpha}_{g_b}(M;\Lambda^{p,q}(M))$ and is naturally isomorphic to the direct sum
$$
          \cC^{k,\alpha}(\pa\tM ; \left. \Lambda^{p,q}({}^bT^*_{\bbC}\tM)\right|_{\pa\tM})\oplus \rho^{\delta}\cC^{k,\alpha}_{g_b}(M; \Lambda^{p,q}(M)).
$$
The norm on this latter space induces a norm on $\cC^{k,\alpha}_{0,\delta}(M; \Lambda^{p,q}(M))$.
\begin{proposition}
For $\delta>0$ sufficiently small,
$$
    \pa\db^* G_{-\epsilon}: \cC^{k,\alpha}_{0,\delta}(M; \Lambda^{p,q}(M))\to \cC^{k,\alpha}_{0,\delta}(M; \Lambda^{p+1,q-1}(M))
$$
is bounded.
\label{gf.3}\end{proposition}
\begin{proof}
By definition, the indicial operator $I(P)$ of $P:=\pa\db^*G_{-\epsilon}$ is the restriction of the Schwartz kernel of $P$ to the front
face of the $b$-double space.  Recall from \cite{MelroseAPS} that $I(P)$ is an $\bbR^+$-invariant operator on the cylinder
$\pa\tM\times (0,+\infty)_{\rho}$.  Moreover, the corresponding indicial family $I(P,\lambda)$, which is the Mellin transform of $I(P)$,
satisfies
\begin{equation}
     I(P)\varpi^* u= \varpi^* I(P,0)u \quad \mbox{for}\quad u\in \cC^{k,\alpha}(\pa\tM;  \left. \Lambda^{p,q}({}^bT^*_{\bbC}\tM)\right|_{\pa\tM}),
\label{gf.4}
\end{equation}
where $\varpi: \pa\tM\times (0,+\infty)_{\rho}\to \pa\tM$ is the projection onto the left factor.  Therefore,
\begin{equation}
  \chi I(P) \chi\pi^* u=  \chi\pi^* I(P,0)u -  \chi I(P) (1-\chi)\varpi^*u.
\label{gf.5}\end{equation}
Clearly,
 $$
 I(P,0): \cC^{k,\alpha}(\pa\tM;  \left. \Lambda^{p,q}({}^bT^*_{\bbC}\tM)\right|_{\pa\tM})\to \cC^{k,\alpha}(\pa\tM;  \left. \Lambda^{p+1,q-1}({}^bT^*_{\bbC}\tM)\right|_{\pa\tM})
 $$
is bounded. On the other hand, applying \cite[Proposition~3.27]{MazzeoEdge} as in the proof Proposition~\ref{gf.1}, we
see that
 $$
   \chi I(P)(1-\chi):  \varpi^* \cC^{k,\alpha}(\pa\tM;  \left. \Lambda^{p,q}({}^bT^*_{\bbC}\tM)\right|_{\pa\tM})\to \rho^{\delta}\cC^{k,\alpha}_{g_b}(M; \Lambda^{p+1,q-1}(M))
 $$
is also bounded. One similarly checks that $\chi I(P) \chi$ is bounded on $\rho^{\delta}\cC^{k,\alpha}_{g_b}$-forms.  Altogether,
\begin{equation}
\chi I(P) \chi:  \cC^{k,\alpha}_{0,\delta}(M; \Lambda^{p,q}(M))\to \cC^{k,\alpha}_{0,\delta}(M; \Lambda^{p+1,q-1}(M))
\label{gf.6}\end{equation}
is bounded.
Now, by construction, the Schwartz kernel of $P- \chi I(P) \chi$ has positive index sets at all front faces. Thus,
by \cite[Proposition~3.27]{MazzeoEdge},
\begin{equation}
      P-\chi I(P)\chi : \cC^{k,\alpha}(M;\Lambda^{p,q}(M))\to \rho^{\delta}\cC^{k,\alpha}(M;\Lambda^{p+1,q-1}(M))
 \label{gf.7}\end{equation}
is bounded for $\delta$ sufficiently small. Combining \eqref{gf.6} and \eqref{gf.7} yields the result.
 \end{proof}

\begin{theorem}
Let $(M,g_b)$ be a compactifiable asymptotically cylindrical Calabi-Yau manifold with compactification $\bM$.  Then the logarithmic deformations of $M$ (in the sense of Kawamata \cite{Kawamata1978})  are unobstructed.
\label{tt.13c}\end{theorem}
\begin{proof}
We construct the formal power series of Proposition~\ref{tt.13} more systematically.  Given an infinitesimal deformation
$\phi_1\in L^2_b\cH^{0,1}_-(M;T_{\bM}(\log\bD))$, choose the coefficients of the power series of Proposition~\ref{tt.13} by
\begin{equation}
    \phi_k= \frac12\eta^{-1} \left(\pa \db^* G_{-\epsilon} \left(  \sum_{j=1}^{k-1} \eta(\phi_j\wedge \phi_{k-j}) \right)  \right),
\label{tt.13d}\end{equation}
since in the proof of the $\pa\db$-lemma (Lemma~\ref{dd.1}), we can take $\psi= G_{-\epsilon}\beta$.
By \eqref{tt.13d} and Proposition~\ref{gf.3}, if $\delta>0$ is sufficiently small, there is a positive constant $K_{k,\alpha}$ such that
\begin{equation}
\| \phi_{\ell} \|_{\cC^{k,\alpha}_{\delta,0}}\le K_{k,\alpha} \sum_{i=1}^{\ell-1}\left(  \|\phi_i\|_{\cC^{k,\alpha}_{\delta,0}} \cdot \|\phi_{\ell-i}\|_{\cC^{k,\alpha}_{\delta,0}} \right).
\label{tt.13ee}\end{equation}
Now apply the argument of \cite{Kodaira} to conclude that for $m\in\bbN$,  there is $\delta_{m}>0$ such that
\begin{equation}
    \phi(t)= \sum_{k=1}^{\infty} \phi_k t^k
\label{tt.13f}\end{equation}
converges in $\cC^{m,\alpha}_{0,\delta}$ for $|t|<\delta_m$.  This does not immediately imply that $\phi(t)$ is smooth. To prove this,
note that from its construction, $\phi$ is a solution of the non-linear equation
\begin{equation}
     \Delta_{\db}\phi= \db^*\db \phi= -\frac12 \db^*[\phi,\phi].
\label{gf.8}\end{equation}
This equation can be put in the form
\begin{equation}
       \Delta_{\db}\phi + \phi\cdot P\phi= \nabla\phi\cdot \nabla \phi,
\label{gf.8a}\end{equation}
where $P$ is some second order differential operator and $\cdot$ denotes some contraction of indices.  When $\phi$ is sufficiently
small in $\cC^0$-norm, this is a quasi-linear elliptic equation, so by taking $\delta_m$ smaller if necessary, we see that $\phi$
is smooth for $|t|<\delta_m$.  Similarly, restricting this equation to the boundary, we see that $\left.\phi(t)\right|_{\pa\tM}$
is smooth. Thus, $\phi_\pa(t):= \chi c^*\pi^*(\left.\phi(t)\right|_{\pa\tM})$ is smooth and $v= \frac{\phi-\phi_{\pa}}{\rho^{\delta}}$
satisfies the equation
\begin{equation}
  (\rho^{-\delta}\Delta_{\db} \rho^{\delta})v= \rho^{-\delta}\left( -\frac12\db^*[\phi,\phi]- \Delta_{\db} \phi_{\pa} \right).
\label{gf.9}\end{equation}
By definition of $\phi_{\pa}$, the right hand side of \eqref{gf.9} is in $\cC^{k,\alpha}_{g_b}(M; \Lambda^{0,1}(T^*M)\otimes T_{\bM}(\log\bD))$.
Since $(M,g_b)$ has bounded geometry, interior Schauder estimates and a bootstrapping argument imply that 
$$v\in \CI_{g_b}(M;
\Lambda^{0,1}(TM)\otimes T_{\bM}(\log\bD)).$$  Consequently,  $\phi= \phi_{\pa}+ \rho^{\delta}v \in \CI_{0,\delta}(M;
\Lambda^{0,1}(TM)\otimes T_{\bM}(\log\bD))$.  Using  \eqref{gf.8} and Corollary~\ref{breg.9}, we can apply a bootstrapping argument
as in the proof of Theorem~\ref{accy.1} to show that $\phi$ is in fact polyhomogeneous.   Proceeding as in \cite{Kodaira}, we
also check that $\phi$ is smooth in $t$.

Finally, notice that by construction,  $\phi_{\pa}(t)\in \CI(\pa\tM; \left. \Lambda^{0,1}({}^bT^*\tM)\otimes T_{\bM}(\log\bD)\right|_{\pa \tM})$ corresponds to a deformation of the complex structure of  $N_{\bD}\setminus 0$, \ie a $\iota$-invariant deformation of  $D\times \bbC^*$.  From \eqref{cycy.1}, we see that the Calabi-Yau metric on $M$ induces on $D\times \bbC^*$ a Calabi-Yau cylindrical metric 
$$
      g_{\pa}= g_D + \lambda\frac{dw\odot d\overline{w}}{|w|^2}.
$$  
The Calabi-Yau manifold $(D\times \bbC^*,g_{\pa})$ is naturally compactified by $D\times \bbC\bbP^1$.  For this compactification, we have a natural identification
\begin{equation}
\begin{aligned}
  L^2\cH_-^{0,1}(D\times \bbC^*; T^{1,0}(D\times \bbC^*)) &\cong H^1(D\times\bbC\bbP^1; T^{1,0}D\oplus \mathcal{O}_{D\times \bbC\bbP^1}) \\
                &\cong  H^1(D;T^{1,0}D) \oplus  H^1(D; \mathcal{O}_{D})\\
                &\cong \cH^{0,1}(D;T^{1,0}D)\oplus \cH^{0,1}(D),
\end{aligned}  
\label{cycy.2}\end{equation}
where in the last line the spaces of harmonic forms are defined with respect to the Calabi-Yau metric $g_D$.  
The identification $\Upsilon:\cH^{0,1}(D;T^{1,0}D)\oplus \cH^{0,1}(D) \to L^2\cH_-^{0,1}(D\times \bbC^*; T^{1,0}(D\times \bbC^*))$ is given by
\begin{equation}
 \Upsilon(\omega_1,\omega_2)= \pr_1^*(\omega_2) +  \pr_1^*(\omega_2)\otimes w\frac{\pa}{\pa w},
\label{cycy.3}\end{equation} 
where $\pr_1:D\times \bbC^*\to D$ is the projection on the first factor.  Notice in particular that elements of $L^2\cH_-^{0,1}(D\times \bbC^*; T^{1,0}(D\times \bbC^*))$ are $\bbC^*$-invariant.  Now, the restriction $\phi_{\pa}(t)$ of $\phi(t)$ can be recovered  by applying the construction \eqref{tt.13d} to $D\times \bbC^*$ starting with the restriction $\phi_{1,\pa}\in L^2\cH_-^{0,1}(D\times \bbC^*; T^{1,0}(D\times \bbC^*))$ of the infinitesimal deformation $\phi_1$.  Since the Laplacian on $D\times \bbC^*$ is $\bbC^*$-invariant, so is 
the generalized inverse $G_{-\epsilon}$. This means that the construction \eqref{tt.13d} is carried out in a $\bbC^*$-invariant way, 
hence $\phi_{\pa}(t)$ is $\bbC^*$-invariant.  We also deduce from \eqref{cycy.3} and \eqref{tt.13d} that $\phi_{\pa}(t)$ is of the form
\begin{equation}
  \phi_{\pa}(t)= \pr_1^*\mu_1+ \pr_1^*(\mu_2)\otimes w\frac{\pa}{\pa w}
\label{cycy.4}\end{equation}
with $\mu_1\in \Omega^{0,1}(D)$ and $\mu_2\in \Omega^{0,1}(D;T^{1,0}D)$.  In this decomposition, the first term corresponds to a deformation of the complex structure on $D$, while the second term corresponds to a deformation of the holomorphic structure of the trivial $\bbC^*$-bundle over $D$.  This shows that $\phi_{\pa}(t)$ naturally extends to a deformation of $D\times \bbC$ (and $D\times \bbC\bbP^1$).  The whole construction is $\iota$-invariant, so it descends to a deformation of $N_{\bD}$ as a holomorphic orbifold line bundle.  

The new line bundle obtained from such a deformation is not necessarily holomorphically trivial, but nevertheless, the proof of \cite[Theorem~{3.1}]{HHN2012} still works, so that there is a diffeomorphism $\psi_t$ on $M$ such that if $J_t$ is the new complex structure defined by $\phi(t)$, then $\psi_t^*J_t$ extends to a smooth complex structure $\overline{J}_t$ on $\bM$, making $(M,J_t)$ a compactifiable complex manifold as in Definition~\ref{tt.1}.
\end{proof}

Combining this result with the result of Kovalev \cite{Kovalev2006}, we obtain the following.
\begin{corollary}
Let $(M,g_b)$ be a compactifiable asymptotically cylindrical Calabi-Yau manifold with compactification $\bM$.  Then any Ricci-flat asymptotically cylindrical metric on $M$ sufficiently close to $g_b$ is Kähler with respect to some logarithmic deformation of the complex structure on $M$.
\label{Koiso.1}\end{corollary}

We are also interested in studying relative logarithmic deformations, i.e., deformations which fix the complex structure on $N_{\bD}$.
Infinitesimal relative logarithmic deformations correspond to
$$
    \Im\left(  H^1(\bM ; T_{\bM}(\log\bD)(-\bD)) \to H^1(\bM;T_{\bM}(\log\bD)) \right),
$$
and by Theorem~\ref{ht.5}, this space is the same as $L^2_b\cH^{0,1}(M; T_{\bM}(\log\bD))$.

\begin{theorem}
Let $(M,g_b)$ be a compactifiable asymptotically cylindrical Calabi-Yau manifold with compactification $\bM$.  Then the relative logarithmic deformations of $M$  are unobstructed.
\label{tt.15}\end{theorem}
\begin{proof}
If $\phi_1\in L^2_b\cH^{0,1}(M;T_{\bM}(\log\bD))$ represents an infinitesimal deformation, then Theorem~\ref{tt.13c} gives a
deformation \eqref{tt.13f}.  We must check that this solution $\phi(t)$ decays at infinity so that it is a relative logarithmic deformation.

Choosing $\delta<\inf\cI_0$, we see from Proposition~\ref{gf.1} that instead of \eqref{tt.13ee}, there is a positive constant
$K_{k,\alpha}$ such that
\begin{equation}
          \| \phi_{\ell} \|_{\rho^{\delta}\cC^{k,\alpha}_{g_b}}\le K_{k,\alpha} \sum_{i=1}^{\ell-1}\left(  \|\phi_i\|_{\rho^{\delta}\cC^{k,\alpha}_{g_b}} \cdot \|\phi_{\ell-i}\|_{\rho^{\delta}\cC^{k,\alpha}_{g_b}} \right).
\label{tt.17}\end{equation}
We  thus conclude that for $m\in \bbN$, there is $\delta_m>0$ such that  $\phi(t)\in \rho^{\delta}\cC_{g_b}^{m,\alpha}(M;
\Lambda^{0,1}({}^bT^*M)\otimes T_{\bM}(\log\bD))$ for $|t|< \delta_m$.  Since $\phi_{\pa}=0$,  we deduce from
Theorem~\ref{tt.13c} that $\phi(t)= \rho^{\delta}v\in \rho^{\delta}\cC_{g_b}^{\infty}(M; \Lambda^{0,1}({}^bT^*M)\otimes T_{\bM}(\log\bD))$,
which gives the desired decay.
\end{proof}

\section{A families index for Dirac-type $b$-operators with fixed indicial family} \label{lfi.0}
We consider a smooth fibre bundle
\begin{equation}
\xymatrix{
       \tM \ar[r] & \tN \ar[d]^{\phi}\\  & B
}
\label{lfi.1}\end{equation}
with base $B$ a smooth connected manifold and typical fibre $\tM$ an even dimensional oriented manifold with boundary.  We will suppose that the restriction of $\phi$ to the boundary of $\tN$ induces the trivial fibre bundle
\begin{equation}
\xymatrix{
       \pa \tM \ar[r] & \pa \tN= \pa \tM\times B \ar[d]^{\phi|_{\pa}}\\  & B,
}
\label{lfi.2}\end{equation}
where $\phi|_{\pa} = \pi_R: \pa \tM \times B\to B$ is the projection onto the right factor.  Let $\rho\in\CI(\tN)$ be a choice of boundary
defining function and let $g_b$ be a family of fibrewise polyhomogeneous exact $b$-metrics on the fibres of \eqref{lfi.1} with
restriction to $\pa \tN$ given by
$$
      \left.  g_b \right|_{\pa \tN}= \pi_{L}^* h,
$$
where $\pi_L: \pa \tM\times B \to \pa \tM$ is the projection on the left factor and $h$ is the restriction to $\pa \tM$ of an exact
polyhomogeneous $b$-metric on $M=\tM\setminus \pa \tM$.  In other words, the restriction of the family $g_b$ to $\pa \tN$ is  constant in $b\in B$.  Let $\Cl(\tN/B)$ be the family of Clifford bundles associated to ${}^bT(\tN/B)$ and $g_b$.   Finally, let $\cE\to \tN$ be a smooth family of Clifford modules with Clifford connections $\nabla^{\cE}$.  Assume moreover that the restriction of $(\cE, \nabla^{\cE})$ to $\pa \tN$ is the pull-back under $\pi_{L}$ of the restriction of a Clifford module with Clifford connection on $\tM$ associated to the Clifford bundle $\left. \Cl({}^b T(\tN/B))\right|_{\phi^{-1}(b)}$ for some $b\in B$.  In other words,  $\left.  (\cE, \nabla^{\cE}) \right|_{\pa \tN}$ is `constant' in $b\in B$.  Let $\eth\in \Diff^1_b(\tN/B;\cE)$ be the corresponding family of Dirac-type operators.  With our assumptions, the indicial operator $I(\eth)$ is the same for each element of the family.  We will further assume that
\begin{equation}
\dim \ker_{L^2}\eth^+(b) \; \mbox{and} \; \dim\ker_{L^2} \eth^{-}(b) \; \mbox{are independent of} \; b\in B.
\label{assumption.1}\end{equation}

A simple example of a family satisfying all of these hypotheses is the family of signature operators associated to $g_b$.  More importantly
for us is the family of Dolbeault operators associated to a family of asymptotically cylindrical Calabi-Yau metrics;
that these operators satisfy all of the conditions above is proved in \S~\ref{wp.0}.

If the operator $I(\eth,0)$ is invertible, then Theorem~\ref{lfi.16}, the main result of this section, is an immediate consequence of the
families index theorem of Melrose and Piazza \cite{MP1997}.   Thus, we shall concentrate on the case where $I(\eth,0)$ is not invertible.
The families index theorem of Melrose and Piazza \cite{MP1997} does not apply then since this is no longer a Fredholm family.
Nevertheless, assumption \eqref{assumption.1} together with the constancy of the family of indicial operators makes it possible
to derive  a local formula for the Chern character of the $L^2$-index bundle of the family $\eth$.

Before going into the precise statement of the result and details of the proof, let us point out that the results from \cite{MelroseAPS,MP1997}
used here are only stated for exact $b$-metrics with index set $F=\bbN_0\times\{0\}$. Nevertheless, these results
do admit straightforward generalizations when the $b$-metrics are \textbf{polyhomogeneous}, \cf \cite{MazzeoEdge}.  Indeed, 
the presence of a polyhomogeneous exact $b$-metric necessitates only mild changes to the index sets appearing in these results
and their proofs.  Moreover, only the parts of the index sets close to zero are relevant, so if we replace the boundary defining function
$\rho$ by $x=\rho^{\frac{1}k}$ for some large $k\in\bbN$, then
$$
                 g_b-g_p\in x^N\CI_b(M;{}^{b}T\tM_{\frac1k}\otimes {}^{b}T\tM_{\frac1k}),
$$
where $g_p$ is a product $b$-metric and $\tM_{\frac1k}$ is the $k^{\mathrm{th}}$ root of $\tM$ as defined in \cite{EMM} (i.e., the manifold
$\tM$ with the new $\CI$ structure obtained by adjoining $x = \rho^{\frac1k}$).  Thus, for the purposes of applying the results of
\cite{MelroseAPS,MP1997}, this change effectively presents the metric $g_b$ as a product $b$-metric.  From now on, we will therefore
apply the results of \cite{MelroseAPS,MP1997} to polyhomogeneous $b$-metrics without further comments.

To define the local families $L^2$-index, choose a connection for the fibre bundle \eqref{lfi.1}, i.e., a splitting
$$
{}^b T\tN= T_{H}\tN \oplus {}^b T(\tN/B),  \quad \phi^*TB\cong T_{H} \tN.
$$
We assume that this agrees on $\pa \tN$ with the canonical splitting induced by the identification $\pa \tN= \pa \tM\times B$.
We can then associate to $\eth$ a Bismut superconnection
$$
        \bbA= \eth + \bbA_{[1]}+ \bbA_{[2]},
$$
see \cite[(9.23)]{MP1997} for a definition. The rescaled Bismut superconnection is then given by
$$
   \bbA_t= t^{\frac12} \delta_t \circ \bbA \circ \delta_t^{-1}= t^{\frac12}\eth + \bbA_{[1]} + t^{-\frac12}\bbA_{[2]},
$$
where $\delta_t$ is the automorphism which multiplies elements of $\CI(\tN; \phi^*\Lambda^j(T^*B)\otimes \cE)$ by $t^{-\frac{i}2}$,
\cf \cite[p.281]{BGV}.   Let us also denote by $\Pi_0$ the orthogonal projection onto the $L^2$-kernel bundle of $\eth$.  The operator
$$
    \nabla^{L^2}= \Pi_0 \bbA_{[1]}\Pi_0
$$
defines a smooth $\bbZ_2$-graded connection on the $L^2$-kernel bundle of $\eth$.

\begin{proposition}
The $b$-Chern character of the Bismut superconnection is such that
\begin{gather}
\lim_{t\to\infty} {}^b\!\Ch(\bbA_t)_{[0]}= \widetilde{\ind}(\eth(b)),  \quad \forall b\in B, \label{lfi.3a}  \\
\lim_{t\to\infty} {}^b\!\Ch(\bbA_t)_{[2n]}= \Ch(\ker_{L^2} \eth, \nabla^{L^2})_{[2n]}, \quad n>0, \label{lfi.3b}
\end{gather}
where $\widetilde{\ind}(\eth(b))$ is the extended index of Melrose \cite[(In.30)]{MelroseAPS}.
\label{lfi.3}\end{proposition}

The proof of this proposition is carried out in a set of lemmas, following the strategy of \cite{BGV} and \cite{MP1997}.  However,
since our family of operators is not Fredholm, many important modifications are necessary.  Let us first introduce some notation.
As in \cite[(15.8) and (A.9)]{MP1997}, set
$$
   \cN^{\epsilon} = \Omega^{*}(B)\otimes_{\CI(B)}\Psi^{-\infty,\epsilon}_{\phi}(\tN;\cE)=  \Omega^{*}(B)\otimes_{\CI(B)}\cA^{\epsilon-}(\tN\times_{\phi}\tN;\cE\otimes {}^{b}\Omega^{\frac12}_{\fib});
$$
this is the space of smooth families of operators of order $-\infty$ with Schwartz kernel conormal and vanishing at order
$\epsilon-\nu$ for all $\nu>0$ at the boundaries of $\tN\times_{\phi}\tN$.  We also set, \cf \cite[(15.8) and (A.11)]{MP1997},
$$
  \cM^{\epsilon}= \Omega^*(B)\otimes_{\CI(B)}(\rho^{\epsilon}\Psi^{0}_{b,\phi}(\tN;\cE) + \rho^{\epsilon}\Psi^{-\infty,\epsilon}_{b,\phi}(\tN;\cE)  + \Psi^{-\infty,\epsilon}_{\phi}(\tN;\cE) )  ,
$$
where we refer to \cite[(A.11)]{MP1997} for the definition of the space of family pseudodifferential operators $\Psi^{0,\delta}_{b,\phi}(\tN;\cE)$.
Notice that contrary to the definition in \cite{MP1997}, the space of operators $\cM^{\epsilon}$ is residual in the sense that the
Schwartz kernels of its elements decay like $\rho^\epsilon$ at the front face of the $b$-double space.  There are filtrations
\begin{equation}
\begin{aligned}
 \cN^{\epsilon}_i &= \sum_{k\ge i}\Omega^{k}(B)\otimes_{\CI(B)}\Psi^{-\infty,\epsilon}_{\phi}(\tN;\cE), \\
\cM^{\epsilon}_i &= \sum_{k\ge i}\Omega^k(B)\otimes_{\CI(B)}(\rho^{\epsilon}\Psi^{0}_{b,\phi}(\tN;\cE) + \rho^{\epsilon}\Psi^{-\infty,\epsilon}_{b,\phi}(\tN;\cE)  + \Psi^{-\infty,\epsilon}_{\phi}(\tN;\cE)   ).
\end{aligned}
\label{suco.2}\end{equation}
As in \eqref{defe}, we choose $\epsilon \in (0,\inf\cI)$, where $\cI$ is the index set of the expansions of elements of the
$L^2$-kernel of $\eth$.  Since the fibration $\phi: \tN\to B$ and $\cE$ are trivial over $\pa\tM$, we see from \cite[(9.25)]{MP1997}
that we can choose $\epsilon$ small enough to ensure that
$$
     \bbA^2= \eth^2 \mod \cM^{\epsilon}_{1}.
$$

Now, using the projection $\Pi_0$ onto the $L^2$-kernel, we can decompose the Bismut superconnection as follows,
$$
     \bbA= \widetilde{\bbA}+ \omega \quad \mbox{with} \quad \omega=\Pi_0\bbA (\Id-\Pi_0) + (\Id-\Pi_0)\bbA\Pi_0\in \cN^{\epsilon}_1.
$$
In terms of the decomposition $L^2_b(\tN/B;\cE)= \ran(\Pi_0) \oplus \ran (\Id-\Pi_0)$, we have
$$
   \omega= \left(  \begin{array}{cc} 0 & \mu \\ \nu & 0  \end{array}\right) \quad \mbox{with} \quad \mu= \Pi_0\bbA (\Id-\Pi_0), \; \nu= (\Id-\Pi_0)\bbA\Pi_0.
$$
Hence the curvature is
\begin{equation}
\left( \begin{array}{cc} X & Y \\ Z & T  \end{array} \right) := \cF= (\widetilde{\bbA}+\omega)^2= \widetilde{\bbA}^2+ [\widetilde{\bbA},\omega] + \omega\wedge \omega =\left( \begin{array}{cc}R+\mu\nu & \Pi_0[\widetilde{\bbA},\mu](\Id-\Pi_0) \\ (\Id-\Pi_0)[\widetilde{\bbA},\nu]\Pi_0 & S+ \nu \mu \end{array} \right),
\label{lfi.8}\end{equation}
where $R=\Pi_0 \widetilde{\bbA}^2\Pi_0$ and $S=(\Id -\Pi_0) \widetilde{\bbA}^2(\Id-\Pi_0)$.  We conclude that
 $$
   \left( \begin{array}{cc} X & Y \\ Z & T \end{array} \right) = \left( \begin{array}{cc} R_{[2]}+ \mu_{[1]}\nu_{[1]} & \mu_{[1]}\eth \\ \eth\nu_{[1]} & \eth^2 \end{array} \right) \mod \left( \begin{array}{cc} \cN^{\epsilon}_3 & \cN^{\epsilon}_2 \\
    \cN^{\epsilon}_2 & \cM^{\epsilon}_1 \end{array} \right). $$

Now consider the family of Fredholm operators
\begin{equation}
    \eth^2 : \rho^{-\epsilon}H^{k+2}_b(\tN/B;\cE)\to \rho^{-\epsilon}H_b^k(\tN/B;\cE).
\label{lfi.9}\end{equation}
This has cokernel canonically identified with the $L^2$-kernel of $\eth$.  Since the index of \eqref{lfi.9} is independent of
$b\in B$, the kernels of the operators of this family form a vector bundle over $B$.  Thus by \cite[Proposition~{5.64}]{MelroseAPS}
and \cite[Theorem~6.1]{MazzeoEdge}, there exists a smooth family of generalized inverses $b\mapsto G(b)$ of \eqref{lfi.9} such that
$$
    \eth^2G= \Id-\Pi_0, \quad G\eth^2= \Id-\Pi_1,
$$
where $\Pi_1$ is a smooth family of projections onto the kernel of \eqref{lfi.9}. The projection $\Pi_0$ acts on
$\rho^{-\epsilon}L^2(\phi^{-1}(b);\cE)$ by
$$
    \Pi_0(b): \rho^{-\epsilon}L^2(\phi^{-1}(b);\cE)\to \ker_{L^2}\eth(b),  \quad \Pi_0(b)(u)= \sum_{i=1}^{k} \langle u, v_i\rangle_{L^2} v_i,
$$
where $v_1,\ldots, v_k$ is a choice of orthonormal basis of $\ker_{L^2} \eth (b)$.  The formal adjoint $G^*$ of $G$ is a smooth
family of generalized inverses for
\begin{equation}
  \eth^2: \rho^{\epsilon}H^{k+2}_b(N/B;\cE)\to \rho^{\epsilon}H^k_b(N/B;\cE)
\label{lfi.10}\end{equation}
and satisfies
$$
     G^* \eth^2= \Id- \Pi_0, \quad \eth^2G^*=\Id-\Pi_1^*.
$$
Acting on $\rho^{\epsilon}L^2(N/B;\cE)$, we have
\begin{equation}
\eth^2G^*(\Id-\Pi_0)= \eth^2G^*(\eth^2G)= \eth^2(\Id-\Pi_0)G= \eth^2G= \Id-\Pi_0.
\label{lfi.11}\end{equation}
Taking the adjoint yields that
\begin{equation}
  (\Id-\Pi_0)G\eth^2= \Id-\Pi_0
\label{lfi.12}\end{equation}
on $\rho^{-\epsilon}L^2(N/B;\cE)$.  Similarly, we have that
\begin{equation}
   \eth G^*\eth= \eth G^*\eth (\Id-\Pi_0)= \eth G^*\eth (\eth^2 G)= \eth(\Id-\Pi_0)\eth G= \eth^2 G= \Id-\Pi_0,
\label{lfi.13}\end{equation}
along with the adjoint equation
\begin{equation}
      \eth G \eth= \Id-\Pi_0.
\label{lfi.4}\end{equation}
In particular, these identities imply that
\begin{equation}
\begin{gathered}
  X_{[2]}- Y_{[1]}GZ_{[1]}= R_{[2]}+ \mu_{[1]}\nu_{[1]} -(\mu_{[1]}\eth)G(\eth \nu_{[1]})= R_{[2]}, \\
 X_{[2]}- Y_{[1]}G^* Z_{[1]}= R_{[2]}+ \mu_{[1]}\nu_{[1]} -(\mu_{[1]}\eth)G^*(\eth \nu_{[1]})= R_{[2]}.
 \end{gathered}
\label{lfi.5}\end{equation}
\begin{lemma}
There exists a family of  operators $A$ with $A-\Id\in \cN^{\epsilon}_1$ such that
$$
         A\cF A^{-1}= A \left( \begin{array}{cc} X & Y \\ Z & T \end{array} \right)A^{-1} = \left( \begin{array}{cc} U& 0 \\ 0 & V \end{array} \right)
         $$
with $U= X-YGZ \!\! \mod \cN^{\epsilon}_3$ and $V=T \!\!\mod \cN^{\epsilon}_1$.
\label{lfi.6}\end{lemma}
\begin{proof}
We proceed by induction on $i=1,\ldots, \dim B$ and  assume  that we have found $A_i$ with $A_i-\Id\in \cN_1^{\epsilon}$ such that
$$
      A_i\cF A_i^{-1}= \left( \begin{array}{cc} X_i & Y_i \\ Z_i & T_i \end{array} \right)\in \left( \begin{array}{cc} \cN_2^{\epsilon} & \cN_i^{\epsilon} \\ \cN_i^{\epsilon} & T+\cN_1^{\epsilon} \end{array} \right),
 $$
where $T_i=\eth^2 \mod \cM_1^{\epsilon}$ (we take $A_1=\Id$).  Notice that if we write $A_i= \Id+ K$, then
$$
A_i^{-1}= \sum_{j=0}^{\dim B} (-1)^j K^j \, .
$$
Now, set
$$
\left( \begin{array}{cc} \widetilde{X}_i & \widetilde{Y}_i \\ \widetilde{Z}_i & \widetilde{T}_i \end{array} \right):= \left( \begin{array}{cc} \Id & -Y_iG \\ G^*Z_i & \Id \end{array} \right)
\left( \begin{array}{cc} X_i & Y_i \\ Z_i & T_i \end{array} \right)\left( \begin{array}{cc} \Id & -Y_iG \\ G^*Z_i & \Id \end{array} \right)^{-1}.
$$
Since $\displaystyle \left( \begin{array}{cc} 0 & -Y_iG\\ G^*Z_i & 0 \end{array} \right)\in \cN_i^{\epsilon}$, we see that
$$
\left( \begin{array}{cc} \Id & -Y_iG \\ G^*Z_i & \Id \end{array} \right)^{-1} -  \left( \begin{array}{cc} \Id & Y_iG \\ -G^*Z_i & \Id \end{array} \right) \in \cN_{2i}^{\epsilon},
$$
and hence
$$
\left( \begin{array}{cc} \widetilde{X}_i & \widetilde{Y}_i \\ \widetilde{Z}_i & \widetilde{T}_i \end{array} \right)= \left( \begin{array}{cc} \Id & -Y_iG \\ G^*Z_i & \Id \end{array} \right)
\left( \begin{array}{cc} X_i & Y_i \\ Z_i & T_i \end{array} \right)\left( \begin{array}{cc} \Id & Y_iG \\ -G^*Z_i & \Id \end{array} \right) \; \mod \cN_{2i}^{\epsilon}.
$$
This gives
\begin{equation}
\begin{gathered}
\widetilde{X_i}=X_i-Y_iGZ_i-Y_iG^*Z_i+ Y_i GT_i G^* Z_i = X_i\mod \cN_{2i}^{\epsilon}, \\
\widetilde{Y}_i= Y_i(\Id-G T_i)+ (X_i- Y_i G Z_i) Y_i G \mod \cN_{2i}^{\epsilon}, \\
\widetilde{Z}_i= (\Id-T_i G^*) Z_i + (G^* Z_i) (X_i- Y_i G^* Z_i) \mod \cN_{2i}^{\epsilon}, \\
\widetilde{T}_i= T_i+ G^*Z_i Y_i+ Z_iY_iG+ G^* Z_i X_iY_iG= T_i \mod \cN_1^{\epsilon},
\end{gathered}
\end{equation}
so using \eqref{lfi.11} and \eqref{lfi.12}, we compute that
\begin{equation}
\begin{gathered}
\widetilde{Y}_i= Y_i(\Id- G\eth^2)= Y_i(\Id-\Pi_0)(\Id- G\eth^2)=0 \mod \cN_{i+1}^{\epsilon}, \\
\widetilde{Z}_i= (\Id-\eth^2 G^*) Z_i= (\Id-\eth^2G^*)(\Id-\Pi_0)Z_i=0 \mod \cN_{i+1}^{\epsilon}.
\end{gathered}
\end{equation}
This shows that we can continue the induction to construct the element $A$ as desired.  Now, if $\displaystyle A=
\left( \begin{array}{cc} \Id+K & M \\ N & \Id+L \end{array} \right)$ with $K,M,N,L\in \cN_1^{\epsilon}$, then we have that
$$
     \left( \begin{array}{cc} \Id+K & M \\ N & \Id+L \end{array} \right)  \left( \begin{array}{cc} X & Y \\ Z & T \end{array} \right)=  \left( \begin{array}{cc} U & 0 \\ 0 & V \end{array} \right)\left( \begin{array}{cc} \Id+K & M \\ N & \Id+L \end{array} \right),
 $$
so that
\begin{gather}
V= (T+LT+NY)(\Id+L)^{-1}=T \mod \cN_1^{\epsilon}, \notag\\
\label{lfi.14}U=(X+KX+MZ)(\Id+K)^{-1}= X+MZ \mod \cN_3^{\epsilon}, \\
Y+MT= UM-KY\in \cN_2^{\epsilon}. \notag
\end{gather}
Multiplying the last equation by $G$ gives that $M=-YG \mod \cN_2^{\epsilon}$.  Substituting this into \eqref{lfi.14}, we obtain finally
that $U=X-YGZ \mod \cN_3^{\epsilon}$, as claimed.
 \end{proof}

We now show that the contribution of $V$ to the Chern character vanishes in positive degree when $t$ tends to infinity.
\begin{lemma}
For $k>0$, the form $(e^{t\delta_t(V)})_{[k]}$ lies in $\cN^{\epsilon}_k$ and decreases rapidly along with all its derivatives as $t$ tends to
infinity.  In particular, it decreases rapidly with all its derivatives as a differential form valued in trace class operators.
\label{lfi.15}\end{lemma}
\begin{proof}
Writing $V=\eth^2+A$ with $A\in \cM_1^{\epsilon}$, we have that $\displaystyle e^{-t\delta_t(V)}= \sum_{k=0}^{\dim B} (-t)^k I_k(t)$, where
$$
I_k(t)= \int_{\Delta_k} e^{-\sigma_0 t\eth^2}\delta_t(A) e^{-\sigma_1 t\eth^2} \delta_t(A)\ldots e^{-\sigma_{k-1}t\eth^2}\delta_t(A) e^{-\sigma_k t\eth^2} d\sigma_1\ldots d\sigma_k.
$$
Here $\Delta_k$ is the simplex
$$
\Delta_k= \{ (\sigma_0,\ldots, \sigma_k) \in \bbR^{k+1} \quad | \quad \sum_{i=0}^k \sigma_i=1, \; \sigma_i\ge 0\}.
$$
The family of operators $e^{-t\eth^2}\in \Psi_b^{-\infty}(\tN/B;\cE)$ is bounded on  $L^2$ uniformly in $t\in [0,\infty)$.  Since
$A\in \cM_1^{\epsilon}$, we have that $(\Id-\Pi_0)e^{-t\eth^2}(\Id-\Pi_0)A$ and $A(\Id-\Pi_0)e^{-t\eth^2}(\Id-\Pi_0)$ are in $\cN^{\epsilon}_1$.
Clearly, it then suffices to show that $(\Id-\Pi_0)e^{-t\eth^2}(\Id-\Pi_0)A$ and $A(\Id-\Pi_0)e^{-t\eth^2}(\Id-\Pi_0)$ are rapidly decreasing with
all their derivatives as $t$ tends to infinity to obtain the result.

Suppose that we establish that
$$
  \begin{array}{llcl} p_{1,t}: & \cN_1^{\epsilon} &\to & \cN^{\epsilon}_1 \\
                                           & U  & \mapsto & (\Id-\Pi_0)e^{-t\eth^2}(\Id-\Pi_0)U
 \end{array}, \quad
\begin{array}{llcl} p_{2,t}: & \cN_1^{\epsilon} &\to & \cN^{\epsilon}_1 \\
                                           & U  & \mapsto & U(\Id-\Pi_0)e^{-t\eth^2}(\Id-\Pi_0)
 \end{array}
$$
have the property that $t^{\frac12}p_{i,t}$ is uniformly bounded with all its derivatives in $B$ as $t$ tends to infinity.
Noticing that for $k\in\bbN$,
$$
       t^k p_{i,t}= (2k)^{k}\left( \frac{t^{\frac12}}{\sqrt{2k}}p_{i, \frac{t}{2k}}\right)^{2k},
$$
we see that $t^k p_{i,t}$ is also uniformly bounded as $t$ tends to infinity and that the same is true for all horizontal derivatives
using Duhamel's formula.  Thus $(\Id-\Pi_0)e^{-t\eth^2}(\Id-\Pi_0)A$ and $A(\Id-\Pi_0)e^{-t\eth^2}(\Id-\Pi_0)$ are rapidly
decreasing with all derivatives as $t$ tends to infinity.

It remains to show that $t^{\frac12}p_{i,t}$ is uniformly bounded with all horizontal derivatives as $t$ tends to infinity.  For this purpose,
we proceed as in the proof of \cite[Proposition~{7.37}]{MelroseAPS} and write the heat kernel of $\eth^2$ in terms of the resolvent,
\begin{equation}
   e^{-t\eth^2} = \frac{1}{2\pi i} \int_{\gamma_A} e^{-t\lambda} (\eth^2-\lambda)^{-1}d\lambda,
\label{res.1}\end{equation}
where $\gamma_A$  is a contour in $\bbC$ that can be taken to be inward along a line segment of argument $-\delta$,
$\frac12>\delta>0$, with end point $(-A-1,0)$, and outward along a line segment of argument $\delta$ from this point.
Choosing a cut-off function $\chi\in \CI(\bbR)$ such that $\chi(r)=1$ for $r<\frac{C}{2}$ and $\chi(r)=0$ for $r>\frac34 C$,
where $0 < C \ll 1$ will be specified later, \eqref{res.1} decomposes as a sum of the two terms
\begin{equation}
\begin{gathered}
   H_1(t)= \frac{1}{2\pi i} \int_{\gamma_A} \chi(\Re \lambda) e^{-t\lambda}(\eth^2-\lambda)^{-1} d\lambda, \\
   H_2(t)= \frac{1}{2\pi i} \int_{\gamma_A} (1-\chi(\Re \lambda)) e^{-t\lambda}(\eth^2-\lambda)^{-1} d\lambda.\end{gathered}
\label{res.2}\end{equation}
In the expression for $H_2$, the integrand is supported in $\Re z\ge \frac{B}{2}$.  Since $(\eth^2-\lambda)^{-1}$ is uniformly
bounded in the calculus with bounds for the part of $\gamma_A$ in that region, we see that $H_2(t): \cN_1^{\epsilon}\to
\cN^{\epsilon}_1$ decays exponentially quickly with all its derivatives as $t$ tends to infinity.

We thus focus attention on $H_1(t)$.  First replace $\gamma_A$ by the simpler contour integral $\Im z= \delta>0$ where
$\lambda=z^2$ and $\Im z>0$ is the physical region.  Using the Cauchy formula, write $H_1= H_1'+ H_1''$, where
\begin{equation}
\begin{gathered}
 H_1'(t)= \frac{1}{2\pi i} \int_{\Im z= \delta}  \chi (\Re \lambda) e^{-t\lambda}(\eth^2-\lambda)^{-1}d\lambda, \\
 H_1''(t)= \frac{1}{2\pi i} \int_{S(A,\delta)} \db \chi(\Re \lambda) e^{-t\lambda} (\eth^2-\lambda)^{-1} d\lambda\wedge d \overline{\lambda}.
\end{gathered}
\end{equation}
Since $\db\chi(\Re z)$ is supported in $\Re z\ge \frac{B}2$, the second term $H_1''(t)$ decays exponentially quickly
with all derivatives as $t\to \infty$, so we reduce further and focus solely on the first term.  Introducing $z$ as a
variable of integration, we have that
\begin{equation}
  H_1'(t)= \frac{1}{\pi i} \int_{\Im z=\delta} \chi(\Re (z^2))e^{-t z^2} (\eth^2-z^2)^{-1} z dz.
\label{res.3}\end{equation}
We know from \cite{MelroseAPS} that $(\eth^2-z^2)^{-1}$ extends meromorphically to $\bbC$ with values in the calculus
with bounds. It has a double pole at $z=0$ with coefficient of $1/z^2$ equal to the projection onto the $L^2$-kernel of $\eth^2$
and with residue, i.e., the coefficient of $1/z$ equal to
\begin{equation}
    \sum_{\ell} U_{\ell} \overline{U}_{\ell} dg_b.
\label{res.4}\end{equation}
Here, $U_{\ell}\in \CI(\phi^{-1}(b);E_b)+ \rho^{\delta} H^{\infty}_b (\phi^{-1}(b);E)$ for some $\delta>0$ is a basis of those solutions
of $\eth^2 U=0$  orthogonal to the subspace of $L^2$-solutions which have boundary values orthonormal in $L^2(\pa\phi^{-1}(b);E_b)$.

Only $(\Id-\Pi_0)H_1'(t)(\Id-\Pi_0)$ is really used in the definition of $p_{i,t}$, so that it is only necessary to deal with
$(\Id-\Pi_0)(\eth^2-z^2)^{-1}(\Id-\Pi_0)$. This has a simple pole at $z=0$ with residue given by  \eqref{res.4}.
In particular, provided that the constant $C$ used in the definition of $\chi$ above is sufficiently small so that
$(\eth^2-z^2)$ has only a pole at $z=0$ on the support of $\chi(\Re (z^2))$, we see that
$$
    P(z)=z \chi(\Re z^2) (\Id-\Pi_0)(\eth^2-z^2)^{-1} (\Id-\Pi_0)
$$
is a family of operators which is smooth down to $\Im z\searrow 0$ and with values in the calculus with bounds
$\Psi^{m,0,0}_{b,os,\infty}(\phi^{-1}(b);E_b)$ (see \cite[(5.107)]{MelroseAPS}). This induces a family of operators
$$
    P(z): \cN^{\epsilon}_1\to \cN^{\epsilon}_1
$$
which is uniformly bounded as $\Im z\searrow 0$.  Taking the limit $\delta\to 0$ and making the change of variable
$Z=z/s$, $s=1/t^{\frac12}$, we see that
$$
      (\Id-\Pi_0)H_1'(t)(\Id-\Pi_0)= \frac{s}{\pi i} \int_{-\infty}^{+\infty} e^{-Z^2} P(sZ) dZ.
$$
After removing the factor $s$ on the right, this integral is a differentiable family in the argument $s^2$
with values in the space of bounded operators on $\cN_1^{\epsilon}$  (acting by composition on the left or on the right).
By the discussion above, this shows that $t^{\frac12}p_{i,t}$ is uniformly bounded as $t$ tends to infinity.  Moreover, using
Duhamel's formula, the same argument can be applied to show the uniform boundedness of the horizontal derivatives
of $t^{\frac12}p_{i,t}$. This completes the proof.
\end{proof}

\begin{proof}[Proof of Proposition~\ref{lfi.3}]
The formula \eqref{lfi.3a} follows from \cite[Proposition~7.37]{MelroseAPS}.  For \eqref{lfi.3b}, Lemma~\ref{lfi.6} gives that
$$
e^{-t\delta_t(\cF)}= \delta_t(A)^{-1} \left( \begin{array}{cc} e^{-t\delta_t(U)} & 0 \\  0 & e^{-t\delta_t(V)} \end{array} \right) \delta_t(A).
$$
Using Lemma~\ref{lfi.6} and Lemma~\ref{lfi.15}, and since $A-\Id\in \cN_1^{\epsilon}$, we see that if $k>0$,
 then $(e^{-t\delta_t(\cF}))_{[k]}$ is a differential form with values in the space of trace-class operators such that
$$
\left(e^{-t\delta_t(\cF)}\right)_{[k]}= \left( \begin{array}{cc} e^{-t\delta_t(U)} & 0 \\  0 & 0 \end{array} \right) +
\mathcal{O}(t^{-\frac12}) =  \left( \begin{array}{cc} e^{-R_{[2]}} & 0 \\  0 & 0 \end{array} \right) + \mathcal{O}(t^{-\frac12}).
$$
Consequently, for $k>0$,
\begin{equation}
\begin{aligned}
\lim_{t\to\infty} {}^b\!\Ch(\bbA_t)_{[2k]}&= \lim_{t\to \infty} {}^{b}\Str((e^{-t\delta_t(\cF)})_{[2k]})= \lim_{t \to \infty} \Str((e^{-t\delta_t(\cF)})_{[2k]}) \\
  &=  \Str(e^{-R_{[2]}})_{[2k]}= \Ch(\ker_{L^2} \eth, \nabla^{L^2})_{[2k]}.
\end{aligned}
\end{equation}
\end{proof}

Combining this result with \cite[Proposition~11 and Proposition~16]{MP1997}, we obtain a formula for the Chern character of the
$L^2$-kernel bundle with respect to the connection $\nabla^{L^2}$.
\begin{theorem}
The Chern character of the $L^2$-kernel bundle of $\eth$ with respect to the connection $\nabla^{L^2}$ is given by
$$
  \Ch(\ker_{L^2} \eth, \nabla^{L^2})_{[2n]}=\left[ \frac{1}{(2\pi i)^{\frac{m}2}} \int_{N/B} \hA(N/B;g_b)\Ch'(\cE)  - d_{B}\gamma \right]_{[2n]}, \quad n>0.
$$
Here, $m=\dim \tM$ and
$$
 \gamma= \int_0^{\infty} {}^{b}\!\STr\left( \frac{d\bbA_t}{dt}e^{-\bbA_t^2}  \right)dt.
$$
\label{lfi.16}\end{theorem}
\begin{proof}
By \cite[Proposition~11]{MP1997},
$$
         \frac{d}{dt} {}^b\!\Ch(\bbA_t)= -d_B \gamma(t) -\hat{\eta}(t), \quad \mbox{where} \;  \gamma(t)= {}^{b}\!\STr
\left( \frac{d\bbA_t}{dt}e^{-\bbA_t^2}  \right), \quad \hat{\eta}(t)\frac{1}{\sqrt{\pi}} \Str_{\Cl(1)}\left( \frac{d\bbB_t}{dt} e^{-\bbB_t^2} \right).
$$
Here, $\bbB_t$ is the rescaled odd superconnection associated to the family of Dirac-type operators on $\pa \tN$,
see \cite[p.38]{MP1997}.  By assumption, this family is trivial, so the only non-zero contribution is in degree $0$.
Thus in fact,
\begin{equation}
\frac{d}{dt} {}^b\!\Ch(\bbA_t)_{[2k]}= -(d_B \gamma(t))_{[2k]} \quad \mbox{for} \; k>0.
\label{var.1}\end{equation}

Next, by \cite[(15.17)]{MP1997},
$$
      \gamma(t)= \mathcal{O}(t^{-\frac12}) \quad \mbox{as} \; t\to 0^+.
$$
On the other hand, using Proposition~\ref{lfi.3} and its proof, we can argue as in the proof of \cite[Theorem~{9.23}]{BGV} to
conclude that for $k>0$,
$$
      \gamma(t)_{[k]}= \mathcal{O}(t^{-\frac32}) \quad \mbox{as} \; t\to \infty.
$$
Finally, \cite[(15.15)]{MP1997} tells us that
$$
   \lim_{t\to 0} {}^b\!\Ch(\bbA_t)_{[2k]}= \left[ \frac{1}{(2\pi i)^{\frac{m}2}} \int_{N/B} \hA(N/B;g_b)\Ch'(\cE) \right]_{[2k]} \quad \mbox{for} \; k>0,
$$
so the result follows by integrating \eqref{var.1} with respect to $t$.
\end{proof}

\section{The curvature of the Quillen connection} \label{qc.0}
In this section, we suppose that $\tN=[\bN; \bD]$, where $\bN$ is a complex orbifold and $\bD$ is an effective orbifold divisor  with $(\bN,\bD)$ satisfying hypotheses (i) and (ii) of section~\ref{ht.0}.  We also assume that $B$ is a complex manifold
and that $\phi: N\to B$ is induced by a  holomorphic fibration $\overline{\phi}: \bN\to B$.  In this case, $E$ is a Hermitian vector
bundle over $\bN$ such that the family of Dirac-type operators is the family of Dolbeault operators
\begin{equation}
       \eth= \sqrt{2} ( \db + \db^*)
\label{qc.0a}\end{equation}
associated to the Clifford bundle $\cE=\Omega^{0,*}(N/B)\otimes E$.  We will also assume that $g_b$ is a family of Kähler metrics
inducing a structure of Kähler fibration on $\phi:N\to B$ in the sense of \cite{BGSII}, with associated connection $T_H N$ for
$\phi:\tN\to B$. Finally, we assume not only that the family of nullspaces $\ker_{L^2}\eth$ is a bundle over $B$, but also
that the $L^2$-kernels of $\eth$ acting on $\Omega^{0,q}(N/B)\otimes E$ determine a bundle over $B$ in each degree $q$.

With these extra assumptions, we now consider the $L^2$-determinant bundle associated to the family \eqref{qc.0a}. This
is the complex line bundle over $B$ given by
\begin{equation}
  \det(\eth^+)= (\Lambda^{\max}\ker_{L^2}(\eth^+))^{-1}\otimes \Lambda^{\max}\ker_{L^2}(\eth^-).
\label{qc.1}\end{equation}
The $L^2$-connection $\nabla^{L^2}$ induces a connection $\nabla^{\det(\eth^+)}$ on $\det(\eth^+)$.  By Theorem~\ref{lfi.16}, the
curvature of this connection is given by
\begin{equation}
     (\nabla^{\det(\eth^+)})^2 = \left[ \frac{1}{(2\pi i)^{\frac{m}2}} \int_{N/B} \hA(N/B;g_b)\Ch'(\cE)  - d_{B}\gamma \right]_{[2]}.
\label{qc.1b}\end{equation}

A more natural choice of connection on $\nabla^{\det(\eth^+)}$ is the Quillen connection.  To describe it, we introduce the $\db$-torsion
of the family of operators $\eth$, following the approach in \cite{MelroseAPS}.  We first define the determinant of the restriction
$\Delta_q$ of the Laplacian $\Delta=\eth^2= 2(\db^*\db+ \db\db^*)$ to elements of type $(0,q)$.

First consider the function
$$
   \zeta_0(\Delta_q,s)= \frac{1}{\Gamma(s)} \int_0^1 t^{s-1} \; {}^{b}\!\Tr(e^{-t\Delta_q}) dt \quad \mbox{for} \; \Re s \gg 0.
$$
Since ${}^{b}\!\Tr(e^{-t\Delta_q})$ admits a short-time asymptotic expansion,
$$
{}^{b}\!\Tr(e^{-t\Delta_q}) \sim \sum_{k=-m}^{\infty} a_k t^{\frac{k}2} \quad \mbox{as} \; t\searrow 0,
$$
the $\zeta$-function $\zeta_0(\Delta_q,s)$ extends to a meromorphic function on $\bbC$ with only simple poles.
Because of the $\Gamma(s)$ factor, $\zeta_0(\Delta_q,s)$ is holomorphic near $s=0$. For $t \nearrow \infty$, we
also consider the  $\zeta$-function
\[
    \zeta_{\infty}(\Delta_q,s)= \frac{1}{\Gamma(s)} \int_1^{\infty} t^{s-1} \;{}^{b}\!\Tr(e^{-t\Delta_q}) \, dt
\quad \mbox{for} \, \Re s \ll 0.
\]
There is an expansion
$$
   {}^{b}\!\Tr(e^{-t\Delta_q}) \sim \sum_{k\ge 0}^{\infty} a_k t^{-k/2} \quad t\to \infty,
$$
so $\zeta_{\infty}(\Delta_q, s)$ extends meromorphically to $\bbC$ with at most simple poles.  As before, this extension is holomorphic
near $s=0$, so altogether,
$$
    \zeta(\Delta_q,s)= \zeta_0(\Delta_q,s) + \zeta_{\infty}(\Delta_q,s)
$$
is holomorphic near $s=0$. We then define the regularized determinant of $\Delta_q$ by the usual formula
$$
               \log \det(\Delta_q):= - \zeta'(\Delta_q,0).
$$

The $\db$-torsion of the family $\eth$ is now defined by
$$
    \log T(\db)= \sum_{q}  (-1)^q q \log\det (\Delta_q).
$$
Alternatively, we can define the $\db$-torsion in terms of one $\zeta$-function using the regularized supertrace and the number
operator $\Q$, which acts on $\Omega^{0,q}(N/B)\otimes E$ as multiplication by $q$, namely
$$
      \log T(\db) = \zeta'(\db,0),
$$
where $\zeta(\db,s)= \zeta_0(\db,s)+ \zeta_{\infty}(\db,s)$ with
\begin{equation}
\begin{gathered}
\zeta_0(\db,s)= \sum_q (-1)^q q \zeta_0(\Delta_q,s) = \frac{1}{\Gamma(s)} \int_0^1 t^{s-1} \; {}^{b}\!\STr(\Q e^{-t\Delta}) dt \quad \mbox{for} \; \Re s \gg 0.  \\
\zeta_{\infty}(\db,s)= \sum_q (-1)^q q \zeta_{\infty}(\Delta_q,s)= \frac{1}{\Gamma(s)} \int_1^{\infty} t^{s-1} \;{}^{b}\!\STr(\Q e^{-t\Delta})dt
\quad \mbox{for} \; \Re s \ll 0.
\end{gathered}
\end{equation}

We now define the \textbf{Quillen metric} of $\det(\eth^+)$ by
$$
    \| \cdot \|_{Q}= T(\db)^{\frac12}\|\cdot \|_{L^2},
$$
where $\|\cdot \|_{L^2}$ is the metric induced by the $L^2$-norm on $\ker_{L^2}\eth$.

To introduce the corresponding Quillen connection, we need a bit more preparation.  Following \cite{BGV}, first consider
the Fréchet bundle $\phi_*\cE\to B$ with fibre
$$
         \left.  \phi_*\cE\right|_b= \dot{\cC}^{\infty}(\phi^{-1}(b); \cE\times {}^{b}\Omega^{\frac12}(\phi^{-1}(b)));
$$
${}^{b}\Omega^{\frac12}(\phi^{-1}(b))$ is the half $b$-density bundle and $\dot{\cC}^{\infty}(\phi^{-1}(b); \cE\times
{}^{b}\Omega^{\frac12}(\phi^{-1}(b)))$ is the space of smooth sections with rapid decay at infinity.  The choices of connection
for $\phi:N\to B$ and $\cE$ induces a connection $\nabla^{\phi_*\cE}$ for $\phi_*\cE$ (\cf \cite[Proposition~9.13]{BGV}).
In fact, $\bbA_{[1]}=\nabla^{\phi_*\cE}$ (see \cite[Proposition~{10.16}]{BGV}). It is convenient to consider a truncated version of the
Bismut superconnection involving only the terms of degree $0$ and $1$,
$$
      \widetilde{\bbA}= \bbA_{[0]} + \bbA_{[1]}= \eth+ \nabla^{\phi_*\cE},
$$
which has the rescaling
$$
    \widetilde{\bbA}_t= t^{\frac12}\eth+ \nabla^{\phi_*\cE}.
$$

For $t\in \bbR^+$, define the following differential forms on $B$:
$$
\alpha^{+}(t) = \frac{1}{2t^{\frac12}} {}^{b}\!\STr \left( \sqrt{2}\db^* e^{-\widetilde{\bbA}_t^2} \right)
\quad \mbox{and}\quad \alpha^{-}(t) = \frac{1}{2t^{\frac12}} {}^{b}\!\STr \left( \sqrt{2}\db e^{-\widetilde{\bbA}_t^2} \right).
$$
In degree $1$,
\begin{equation}
(e^{-\widetilde{\bbA}^2_t})_{[1]} = (e^{-\bbA^2_t})_{[1]}= -t\int_0^1 e^{-(1-\sigma)t\eth^2} t^{-\frac12} [\nabla^{\phi_*\cE},\eth] e^{-\sigma t\eth^2} d\sigma = -t^{\frac12}\int_0^1 e^{-(1-\sigma)t\eth^2} [\nabla^{\phi_*\cE},\eth] e^{-\sigma t\eth^2} d\sigma.
\label{qc.2}\end{equation}
Due to our assumptions on the restriction of the family $\eth$ to $\pa N$, $[\nabla^{\phi_*\cE},\eth]$ has vanishing indicial family,
so the integrand in \eqref{qc.2} is trace class, and we can thus define $\alpha^{\pm}(t)_{[1]}$ without using the $b$-supertrace:
$$
         \alpha^{+}(t)_{[1]}= \frac{1}{2t^{\frac12}} \STr (\sqrt{2}\db^* (e^{-\widetilde{\bbA}^2_t})_{[1]} ) \quad \mbox{and} \quad \alpha^{-}(t)_{[1]}= \frac{1}{2t^{\frac12}} \STr (\sqrt{2}\db (e^{-\widetilde{\bbA}^2_t})_{[1]} ).
$$

\begin{lemma}
The $1$-form components of the differential forms $\alpha^{\pm}(t)$ satisfy
$$
            \overline{\alpha^+(t)_{[1]}} = -\alpha^{-}(t)_{[1]}.
$$
These are rapidly decreasing (with all derivatives in $B$) as $t$ tends to infinity and have short-time asymptotics
$$
          \alpha^{\pm}(t)_{[1]} \sim \sum_{k=-N}^{\infty} t^{\frac{k}2} a^{\pm}_k \quad \mbox{as} \; t\searrow 0.
$$
\label{qc.3}\end{lemma}
\begin{proof}
Since $(e^{-\widetilde{\bbA}^2_t})_{[1]}$ is trace class,  the first assertion follows as in \cite[Lemma~9.42]{BGV}, using the fact that $\db^*$ is
the formal adjoint of $\db$ and that $\nabla^{\phi_*\cE}$ respects the metric on $\phi_*\cE$.  To see that $\alpha^{+}(t)_{[1]}$ is
rapidly decreasing as $t$ tends to infinity, notice that
$$
    \alpha^{+}(t)_{[1]}= -\STr(\sqrt{2} \db^{*}[\nabla^{\phi_*\cE},\eth] e^{-t\eth^2})=
- \STr(\sqrt{2} \db^*[\nabla^{\phi_*\cE},\eth] (\Id-\Pi_0)e^{-t\eth^2}(\Id-\Pi_0)),
$$
so we can use the decay properties of $p_{i,t}:\cN^{\epsilon}_1\to \cN^{\epsilon}_1$ established in the proof of Lemma~\ref{lfi.15} to
conclude that $\alpha^{\pm}(t)_{[1]}$ is rapidly decreasing as $t$ tends to infinity.  The short-time asymptotics
follow from the corresponding asymptotics of the heat kernel.   Taking the complex conjugate of $\alpha^+(t)_{[1]}$, we get
the corresponding statement for $\alpha^-(t)_{[1]}$.
\end{proof}

This lemma shows that the $1$-forms
$$
     \beta^{\pm}(s)= 2\int_0^{\infty} t^{s} \alpha^{\pm}(t)_{[1]}dt
$$
are well-defined and holomorphic in $s$ for $\Re s \gg 0$.  The short-time asymptotics of $\alpha^{\pm}(t)$ also allow us
to extend $\beta^{\pm}(s)$ to a meromorphic function in $s$ for $s\in\bbC$ with at most simple poles.  In particular,
we can define the finite part at $s=0$ by
$$
      \beta^{\pm}= \left. \frac{d}{ds}\frac{1}{\Gamma(s)} \beta^{\pm}(s) \right|_{s=0}.
$$

\begin{lemma}
As a function on $B$, the differential of $\log T(\db)= \zeta'(\db,0)$ equals
$$
            d\zeta'(\db,0)= \beta^+ - \beta^-.
$$
\label{qc.4}\end{lemma}

\begin{proof}
Using Duhamel's formula,
\begin{equation}
      d\zeta'_{\infty}(\db,s)= \frac{d}{ds} \left(    \frac{1}{\Gamma(s)} \int_1^{\infty}  \left( -t^{s-1} {}^{b}\!\STr\left(\Q \int_0^t e^{-(t-\sigma)\eth^2} [\nabla^{\phi_*\cE},\eth^2] e^{-\sigma \eth^2}d\sigma \right)       \right) dt     \right)
\label{qc.5}\end{equation}
for $\Re s \ll 0$. Since $[\nabla^{\phi_*\cE},\eth^2]$ has vanishing indicial family, the term inside the $b$-supertrace is trace class, so
we can replace the $b$-supertrace with the usual supertrace.  The decay properties of $p_{i,t}$ show that the integrand
is rapidly decreasing (with all derivatives in $B$) as $t$ tends to infinity, so the differential $d\zeta'_{\infty}(\db,s)$ extends to
an entire function of $s$.  In particular,
$$
 d\zeta_{\infty}'(\db,0)= - \left. \frac{d}{ds} \left(  \frac{1}{\Gamma(s)} \int_1^{\infty} t^{s}\STr (\Q [\nabla^{\phi_*\cE},\eth^2] e^{-t\eth^2}) dt \right)  \right|_{s=0}.
$$
Similarly, for $\Re s \gg 0$, we have that
$$
d\zeta_{0}'(\db,s)= -  \frac{d}{ds} \left(  \frac{1}{\Gamma(s)} \int_0^{1} t^{s}\STr(\Q [\nabla^{\phi_*\cE},\eth^2] e^{-t\eth^2}) dt \right).
$$
By the short-time asymptotics of the heat kernel, we extend this differential meromorphically in $s\in\bbC$ with only simple poles;
the result is regular at $s=0$.
This shows that
\begin{equation}
  d\zeta'(\db,0)= d\zeta_0'(\db,0) + d\zeta_{\infty}'(\db,0)= - \left. \frac{d}{ds} \left(  \frac{1}{\Gamma(s)} \int_0^{\infty} t^{s}\STr(\Q [\nabla^{\phi_*\cE},\eth^2] e^{-t\eth^2}) dt \right)  \right|_{s=0}.
\label{qc.5a}\end{equation}
Now, using the relation
$$
[\nabla^{\phi_*\cE},\eth^2]= [\nabla^{\phi_*\cE},\eth]\eth + \eth[\nabla^{\phi_*\cE},\eth],
$$
we see that
\begin{equation}
\begin{aligned}
\STr( \Q [\nabla^{\phi_*\cE},\eth^2] e^{-t\eth^2}) &=  \STr( \Q [\nabla^{\phi_*\cE},\eth]\eth e^{-t\eth^2}) + \STr( \Q \eth [\nabla^{\phi_*\cE},\eth] e^{-t\eth^2}) \\
   &= \STr( \Q [\nabla^{\phi_*\cE},\eth] e^{-t\eth^2} \eth) + \STr( \Q \eth [\nabla^{\phi_*\cE},\eth] e^{-t\eth^2}) \\
    & = \STr( [\Q,\eth] [\nabla^{\phi_*\cE},\eth] e^{-t\eth^2}),
 \end{aligned}
\label{qc.5b}\end{equation}
where in the last step, we use that
$$
    0= \STr( \left[(\Q[\nabla^{\phi_*\cE},\eth] e^{-t\eth^2}), \eth\right]) = \STr(\Q[\nabla^{\phi_*\cE},\eth] e^{-t\eth^2} \eth) + \STr(\eth\Q[\nabla^{\phi_*\cE},\eth] e^{-t\eth^2}).
$$
Therefore, since $[\Q,\eth]= \sqrt{2}(\db -\db^*)$, \eqref{qc.5a} and \eqref{qc.5b} show that
 \begin{equation}
\begin{aligned}
d\zeta'(\db,0) &= - \left. \frac{d}{ds} \left(  \frac{1}{\Gamma(s)} \int_0^{\infty} t^{s}\STr( \sqrt{2}(\db -\db^*)[\nabla^{\phi_*\cE},\eth] e^{-t\eth^2}) dt \right)  \right|_{s=0} \\
      &=    \beta^+ - \beta^-,
\end{aligned}
\label{qc.6}\end{equation}
which gives the claimed result.
\end{proof}

The \textbf{Quillen connection} on $\det(\eth^+)$ is given by
$$
       \nabla^{Q}= \nabla^{\det(\eth^+)}+ \beta^+.
$$
As the terminology suggests, $\nabla^Q$ is compatible with the Quillen metric.
\begin{proposition}
The Quillen connection is the Chern connection of $\det(\eth^+)$ with respect to the Quillen metric.
\label{qc.7}\end{proposition}
\begin{proof}
We know that $\nabla^{\det(\eth^+)}$ is the Chern connection of $\det(\eth^+)$  with respect to  the $L^2$-metric.
Since $\|\cdot\|_Q= e^{\frac{\zeta'(\db,0)}2}\|\cdot\|_{L^2}$, we see that $\nabla^Q$ is compatible with the Quillen metric provided
$$
    \beta^+= \frac{d\zeta'(\db,0)}2 +\omega
$$
with $\omega$ an imaginary $1$-form.  But by Lemma~\ref{qc.4},  $\omega= \frac{\beta^+ + \beta^-}2$, and by Lemma~\ref{qc.3},
this is imaginary.  To see that $\nabla^Q$ is the Chern connection of the Quillen metric, we note that $\beta^+$ is a $(1,0)$
form on $B$, which follows from the fact (see \cite[Theorem~1.14]{BGSII}) that $[\nabla^{\phi_*\cE},\db]$ is an operator valued
$(1,0)$-form.
\end{proof}

\begin{theorem}
The curvature of the Quillen connection equals
$$
    (\nabla^Q)^2= \left[ \frac{1}{(2\pi i)^{\frac{m}{2}}} \int_{N/B} \Td(T^{1,0}(N/B),g_b)\Ch(E)  \right]_{[2]},  \quad \mbox{where}  \; m=\dim_{\bbR}M.
$$
\label{qc.8}\end{theorem}
\begin{proof}
The curvature of $\nabla^{\det(\eth^+)}$ is given by \eqref{qc.1b}.  On the other hand, since $\beta^+ - \beta^-=  d\zeta'(\db,0)$ is a
closed form, we have
\begin{equation}
    (\nabla^Q)^2= (\nabla^{\det(\eth^+)})^2 + d\beta^+ = (\nabla^{\det(\eth^+)})^2 + \frac{d(\beta^+ +\beta^-)}2.
\label{qc.9}\end{equation}
Since $\widetilde{\bbA}$ equals $\bbA$ up to terms of order $2$, we see that
\begin{equation}
\begin{aligned}
\frac12(\beta^+ +\beta^-)&=   \left. \frac{d}{ds} \left( \frac{1}{\Gamma(s)} \int_0^{\infty} t^{s}(\alpha^+(t)_{[1]} +\alpha^-(t)_{[1]}) dt \right)  \right|_{s=0} \\
  &=  \left. \frac{d}{ds} \left( \frac{1}{\Gamma(s)} \int_0^{\infty} t^{s} \STr\left( \frac{d\bbA_t}{dt} e^{-\bbA^2_t}\right)_{[1]} dt \right)  \right|_{s=0} \\
  &=  \int_0^{\infty}  \STr\left( \frac{d\bbA_t}{dt} e^{-\bbA^2_t}\right)_{[1]} dt,
\end{aligned}
\label{qc.10}\end{equation}
where in the last step we have used that $\STr\left( \frac{d\bbA_t}{dt} e^{-\bbA^2_t}\right)_{[1]}$ is integrable in $t$.  Thus, by
combining \eqref{qc.1b}, \eqref{qc.9} and \eqref{qc.10}, we see that
\begin{equation}
    (\nabla^Q)^2 = \left[ \frac{1}{(2\pi i)^{\frac{m}2}} \int_{N/B} \hA(N/B;g_b)\Ch'(\cE)   \right]_{[2]}.
\label{qc.11}\end{equation}
Taking advantage of the fact that $\phi: N\to B$ is a Kähler fibration, this integral can be rewritten in terms of the Todd form of
$T^{1,0}(N/B)$ and the Chern character form of the Hermitian bundle $E$.
\end{proof}

\begin{remark}
Instead of using the $\db$-torsion as in \cite{BGSIII} to define the Quillen metric, we could as well have used the
approach of Bismut-Freed \cite{BFI} and defined the metric as
$$
      \| \cdot \|_{BF} := \det(\eth^-\eth^+)^{-\frac12}\|\cdot \|= e^{-\frac12 \zeta'(\eth^-\eth^+,0)}\|\cdot \|_{L^2}
$$
with a compatible connection $\nabla^{BF}$ whose curvature is given by \eqref{qc.11}.  The advantage of the Bismut-Freed
approach is that it works in non-holomorphic settings, but its disadvantage is that in holomorphic settings, the
connection $\nabla^{BF}$ is typically not the Chern connection of the metric $\|\cdot \|_{BF}$.
\end{remark}

\section{The Weil-Petersson metric on the moduli space of asymptotically cylindrical
Calabi-Yau manifolds}  \label{wp.0}
Let $(Z ,g_b)$ be a compactifiable asymptotically cylindrical Calabi-Yau manifold with compactification $\bZ$ such that
$Z=\bZ\setminus \bD$ for some divisor $\bD$ in $\bZ$ as in Definition~\ref{tt.1}.  Assume that $H^{1}(\bZ;\bbR)=0$.  This is the case for instance if $H^1(Z;\bbR)=0$.   Let $L\to \bZ$
be an ample line bundle over $\bZ$ with induced Kähler class $\bell \in H^{1,1}(\bZ)\subset H^2(\bZ;\bbR)$ and denote by
$\ell\in H^2(Z;\bbR)$  the restriction of $\bell$ to $Z$.  By Theorem~\ref{tt.15}, there is a well-defined relative moduli space
in a neighborhood of $Z$,
\begin{equation}
\xymatrix{
              Z \ar[r]^i & \mathfrak{X} \ar[d]^{\phi} \\
                    & \cM_{\rel},
}
\label{wp.1}\end{equation}
with $\phi^{-1}(m_0)= Z$ and
$$
T_{m_0}\cM_{\rel}\cong L^2\cH^{0,1}(Z;{}^{b}\!T^{1,0}Z)\cong \Im\left( H^1(\bZ;T_{\bZ}(\log(\bD))(-\bD)) \to H^1(\bZ;T_{\bZ}(\log(\bD))     \right).
$$
As the complex structure is deformed, the line bundle $K_{\bZ}(\bD)= K_{\bZ}\otimes L_{\bD}$ automatically remains topologically trivial.
Thus, since $H^{0,1}(\bZ) = 0$, we conclude that $K_{\bZ}(\bD)$ also remains \textbf{holomorphically} trivial.  This means that a
global non-zero meromorphic volume form with a simple pole at $D$ continues to exist as we deform the complex structure.
Consequently, we can apply the construction of Haskins-Hein-Nordström and obtain the existence of Calabi-Yau metrics on
the deformed spaces. The following result allows us to make a canonical choice.
\begin{lemma}
The class $\bell$ on $Z$ remains a Kähler class as the complex structure is deformed in $\cM_{\rel}$.
\label{wp.2}\end{lemma}
\begin{proof}
Any class $\xi\in \Im\left( H^1(\bZ;T_{\bZ}(\log(\bD)(-\bD))) \to H^1(\bZ;T_{\bZ}(\log\bD))\right)$ can be paired with
$\bell$ to obtain an element in
$$
    \Im (H^2(\bZ; T^*_{\bZ}\otimes T_{\bZ}(\log\bD)(-\bD))\to H^2(\bZ; T^*_{\bZ}\otimes T_{\bZ}(\log\bD))).
$$
Since $T_{\bZ}(\log\bD)$ is naturally a subsheaf of $T_{\bZ}$, there is a canonical map
$H^2(\bZ; T^*_{\bZ}\otimes T_{\bZ}(\log\bD))\to H^2(\bZ; T^*_{\bZ}\otimes T_{\bZ})$.  Composing with the map $H^2(\bZ; T^*_{\bZ}\otimes
T_{\bZ})\to H^{0,2}(Z)$ induced by the trace $T^*_{\bZ}\otimes T_{\bZ}\to\mathcal{O}_{\bZ}$, we obtain from $\xi$ and $\bell$ a class
$$
    \iota_{\xi} \bell \in H^{0,2}(\bZ).
$$
This encodes how $\bell$ changes when the complex structure is varied in the direction $\xi$.  Restricting to $Z$, again using
the sheaf map $T_{\bZ}(\log\bD)\to T_{\bZ}$, we see that the restriction of $\iota_{\xi}\bell$ to $Z$ comes from an element of
$$
       \Im(H^2(\bZ; \mathcal{O}_{\bZ}(-\bD))\to H^2(\bZ; \mathcal{O}_{\bZ}))\cong  L^2\cH^{0,2}(Z).
$$
But by Corollary~\ref{ht.5a}, this latter space is trivial so $\left. \iota_{\xi}\bell \right|_{Z}=0$.

On the other hand, consider the long exact sequence associated to the pair $(\bZ, \bZ\setminus \cN_{\bD})$, where $\cN_{\bD}$ is a tubular
neighborhood of $\bD$ in $\bZ$, which is of course diffeomorphic to the normal bundle of $\bD$. Depending on whether or not $H^1(Z;\bbR)$ is trivial, we deduce  that either
\begin{equation}
\ker( H^2(\bZ;\bbC)\to H^2(Z;\bbC))=0 \quad \mbox{or that} \quad \ker( H^2(\bZ;\bbC)\to H^2(Z;\bbC))\cong H^2_c(\cN_{\bD};\bbC)\cong H^0(D;\bbC)\cong \bbC.
\label{wp.3}\end{equation}
In the first case, we have automatically that $\iota_{\xi}\bell=0$. In the second case,  $H^2_c(\cN_{\bD};\bbC)$ is generated by a $(1,1)$-current supported on $\bD$. Since $\iota_{\xi}\bell$ is of type $(0,2)$ and
lies in this space, we conclude that $\iota_{\xi}\bell=0$.  Since the class $\xi$ was arbitrary, $\bell$ must remain
unchanged as the complex structure is deformed in $\cM_{\rel}$.
\end{proof}

By this lemma, we can use the class $\bell$ and Theorem~\ref{un.3} with $\lambda=1$ to define for each $m\in\cM_{\rel}$ a
unique asymptotically cylindrical Calabi-Yau metric $g_m$ on $Z_m=\phi^{-1}(m)$ with Kähler form $\omega_m$ belonging to the class
$\ell=\left. \bell\right|_{Z}\in H^2(Z;\bbR)$. This family of Calabi-Yau metrics gives $\phi$ the structure of a Kähler fibration in the
sense of \cite[Definition~1.4]{BGSII}.  Indeed, let $h_m$ be the Hermitian metric on the polarization $L$ over the fibre $\phi^{-1}(m)$
with first Chern form on $\phi^{-1}(m)$ equal to the Kähler form $\omega_m$. Let $\omega$ be the corresponding first Chern form
of the line bundle $L\to \mathfrak{X}$ with Hermitian metric $h_m$.  Clearly, $\omega$ is a closed $(1,1)$-form which restricts
to $\omega_m$ on $\phi^{-1}(m)$ for each $m$.  Define $T_{H}\mathfrak{X}$ as the orthogonal complement of the vertical tangent
bundle $T(\mathfrak{X}/\cM_{\rel})$ with respect to the symmetric $2$-form $g= \omega(J\cdot,\cdot)$, where $J$ is the complex
structure on $\mathfrak{X}$.  Notice that $g$ is not necessarily positive-definite on $\mathfrak{X}$, but it is when restricted to the
fibres of $\phi$, so $T_{H}\mathfrak{X}$ is a well-defined vector bundle inducing the decomposition $T\mathfrak{X}=
T_H\mathfrak{X}\oplus T(\mathfrak{X}/\cM_{\rel})$.  Since the action of $J$ preserves $g$ and  $T\mathfrak{X}/\cM_{\rel}$, it
preserves $T_H\mathfrak{X}$.  This means that the triple $(\phi, g_m, T_H\mathfrak{X})$ is a Kähler fibration.

More importantly, the family of Calabi-Yau metrics $g_m$ induces a metric on the moduli space $\cM_{\rel}$.
\begin{definition}
The \textbf{Weil-Petersson metric} on $\cM_{\rel}$ is defined by
$$
        g_{\WP}(u,v) = \int_{Z_m}\langle u, v\rangle_{g_m} d\mu(g_m) \quad \mbox{for} \; \; u,v\in L^2\cH^{0,1}(Z_m,g_m; {}^{b}\!T^{1,0}\widetilde{Z}_m)\cong T_m\cM_{\rel}.
$$
\label{wp.4}\end{definition}

For compact Calabi-Yau manifolds, the volume is a natural quantity which appears in many computations related to this
Weil-Petersson metric.  Although the volume of an asymptotically cylindrical Calabi-Yau manifold is infinite,
there is a notion of renormalized volume which is an adequate replacement. This is defined by
\begin{equation}
      {}^{R}\!\Vol(Z_m, g_m,\rho) := \FP_{s=0} \int_{Z_m} \rho^{s} d\mu(g_m),
\label{wp.5}\end{equation}
where $\rho\in \CI(Z)$ is a choice of boundary defining function for $Z$ and $\FP_{s=0} f(s)$ denotes the finite part at $s=0$ of the meromorphic
function $f$.  This uses the fact that $s\mapsto \int_{Z_m}\rho^s d\mu(g_m)$ is meromorphic in $s$ with at most a simple pole at $s=0$.
This definition depends on the choice of $\rho$.  However, replacing $\rho$ by $c\rho$ for some positive constant $c$, a simple
computation shows that
$$
       {}^{R}\!\Vol(Z_m, g_m,c\rho)=  {}^{R}\!\Vol(Z_m, g_m,\rho) + \Vol(\pa \widetilde{Z}_m; g_m)\log c.
$$
Thus, choosing $\rho$ appropriately, we can assume that
\begin{equation}
   {}^{R}\!\Vol(Z_{m_0}, g_{m_0},\rho)=1, \quad \mbox{where} \; m_0\in \cM_{\rel} \; \mbox{is such that} \;   Z=\phi^{-1}(m_0).
\label{wp.6}\end{equation}
The good news is that by doing so, nearby Calabi-Yau manifolds in this family also have renormalized volume $1$ for this
same choice of $\rho$.
\begin{lemma}
Suppose that $\rho\in \CI(\widetilde{Z})$ is a boundary defining function such that \eqref{wp.6} holds.  Then
$$
          {}^{R}\!\Vol(Z_m,g_m,\rho)=1 \quad \forall m\in \cM_{\rel}.
$$
\label{wp.7}\end{lemma}
\begin{proof}
Let $\omega_m$ be the Kähler form of $g_m$.  Then, by Lemma~\ref{wp.2} and Corollary~\ref{accy.5}, we know that
$$
\omega_m = \omega_{m_0}+ du
$$
for some polyhomogeneous $L^2$ $1$-form $u$.  Using Stokes' theorem, we thus obtain that
$$
  {}^{R}\!\Vol(Z_m,g_m,\rho)-{}^{R}\!\Vol(Z_{m_0},g_{m_0},\rho) = \int_{Z_m} \frac{\omega_m^n-\omega_{m_0}^n}{n!}=\frac{1}{n!} \int_{Z_m} d \left( \sum_{j=1}^n \frac{n!}{j!(n-j)!} u\wedge (du)^{j-1}\wedge \omega_{m_0}^{n-j} \right)=0.
$$
\end{proof}

\begin{proposition}
The Weil-Petersson metric is Kähler.  Furthermore, the first Chern form of the bundle ${}^{b}\!T^{1,0} (\mathfrak{X}/\cM_{\rel})$ with
Hermitian structure induced by the family of Calabi-Yau metrics associated to the class $\bell$ is given by
$$
      c_1({}^{b}\!T^{1,0} (\mathfrak{X}/\cM_{\rel}))= -c_1(\Lambda^n(({}^{b}\!T^{1,0} (\mathfrak{X}/\cM_{\rel}))^*))= \frac{\phi^*\omega_{WP}}{\pi},
$$
where $\omega_{WP}$ is the Kähler form of $g_{\WP}$.
\label{wp.8}\end{proposition}
\begin{proof}
We know that $\Lambda^n( ({}^{b}\!T^{1,0} (\mathfrak{X}/\cM_{\rel}))^*)$ has a global holomorphic section which is flat with respect to the Levi-Civita connection of $g_m$, hence
$$
    c_1({}^{b}\!T^{1,0} (\mathfrak{X}/\cM_{\rel}))=-\phi^* c_1( H^0( \overline{\mathfrak{X}}/\cM_{\rel}; \Omega^n_{\overline{\mathfrak{X}}/\cM_{\rel}}(\log \bD))).
$$
Let $\Omega$ be a local non vanishing holomorphic section of $ H^0( \overline{\mathfrak{X}}/\cM_{\rel}; \Omega^n_{\overline{\mathfrak{X}}/\cM_{\rel}}(\log \bD))$ on $\cM_{\rel}$.  Then on $Z_m$,
$$
    \Omega(m)\wedge \overline{\Omega(m)} = c_m (-i)^{n^2} d\mu(g_m)
$$
for some constant $c_m>0$, or equivalently,
$$
        |\Omega(m)|^2_{g_m}= c_m,
$$
so that
\begin{equation}
   c_1(H^0( \overline{\mathfrak{X}}/\cM_{\rel}; \Omega^n_{\overline{\mathfrak{X}}/\cM_{\rel}}(\log \bD)))= \frac{i}{2\pi} \db_{\cM_{\rel}}\pa_{\cM_{\rel}} \log |\Omega(m)|^2_{g_m}= \frac{i}{2\pi} \db_{\cM_{\rel}}\pa_{\cM_{\rel}} \log c_m.
\label{wp.9}\end{equation}
Now, the constant $c_m$ can be conveniently rewritten using the renormalized volume,
$$
     c_m= (-i)^{-n^2} \frac{  \FP_{s=0} \int_{Z_m} \rho^s \Omega(m)\wedge\overline{\Omega(m)}  }{{}^{R}\!\Vol(Z_m,g_m,\rho)}= (-i)^{-n^2} \FP_{s=0} \int_{Z_m} \rho^s \Omega(m)\wedge\overline{\Omega(m)},
$$
where, by Lemma~\ref{wp.7}, we can assume that $\rho$ is chosen so that ${}^{R}\!\Vol(Z_m,g_m,\rho)=1$.  Substituting this expression
in \eqref{wp.9}, we compute that
$$
  c_1(H^0( \overline{\mathfrak{X}}/\cM_{\rel}; \Omega^n_{\overline{\mathfrak{X}}/\cM_{\rel}}(\log \bD)))(\xi,\overline{\eta})= -\frac{i}{2\pi} \frac{1}{(-i)^{n^2}c_m}  \int_{Z_m} \iota_{\xi}\Omega(m)\wedge \overline{\iota_{\eta}\Omega(m)}.
$$

Just as for the moduli space of compact Calabi-Yau manifolds, see for instance \cite[p.640-641]{Tian1987}, we also have that
$$
     \int_{Z_m}  \iota_{\xi}\Omega(m)\wedge \overline{\iota_{\eta}\Omega(m)}= c_m (-i)^{n^2} \frac{2}{i} \omega_{\WP}(\xi,\overline{\eta}),
$$
and the result follows from this.
\end{proof}

Consider the holomorphic vector bundle over $\mathfrak{X}$ with coefficients
\begin{equation}
       E= \bigoplus_{p=1}^n (-1)^p p \Omega^{p}(\mathfrak{X/\cM_{\rel}}).
\label{wp.10}\end{equation}
This has a Hermitian structure induced by the family of Calabi-Yau metrics $g_m$ associated to the polarization $\bell$.  Let
$\eth= \sqrt{2}(\db+ \db^*)$ be the corresponding family of Dolbeault operators.  To apply Theorem~\ref{qc.8} to this family,
we must check that the family $\eth$ satisfies the hypotheses of \S~\ref{lfi.0} and \S~\ref{qc.0}.  Since we are on the moduli space
of relative deformations, it is clear from Theorem~\ref{tt.15} and Corollary~\ref{accy.5} that the indicial family $I(\eth,\lambda)$
is unchanged as we move on $\cM_{\rel}$.  On the other hand, we have shown that the family of Calabi-Yau metrics $g_m$
gives $\phi$ the structure of a Kähler fibration $(\phi,g_m, T_H\mathfrak{X})$.  It follows from the next result that
the $L^2$-kernels of $\eth$ acting on $\Omega^{0,q}(\mathfrak{X}/\cM_{\rel})\otimes E$ form a bundle over $\cM_{\rel}$.
\begin{lemma}
The Hodge numbers $h^{p,q}(m):= \dim L^2\cH^{p,q}(\phi^{-1}(m),g_m)$ are independent of $m\in \cM_{\rel}$.
\label{ct.1}\end{lemma}
\begin{proof}
Since $\cM_{\rel}$ is assumed to be connected, it suffices to show that $h^{p,q}$ is locally constant.  For $\epsilon>0$ sufficiently
small,  $L^2\cH^{p,q}(\phi^{-1}(m),g_m)$ corresponds to the kernel of the Fredholm operator
$$
\Delta_{\db}: \rho^{\epsilon}H^2_b(\phi^{-1}(m), \Omega^{p,q}(\phi^{-1}(M)))\to \rho^{\epsilon} L^2_b(\phi^{-1}(m), \Omega^{p,q}(\phi^{-1}(M))).
$$
Hence, there is a small neighborhood $\cU$ of any $m_0\in \cM_{\rel}$ such that $h^{p,q}(m)\le h^{p,q}(m_0)$ for all $m\in \cU$.
On the other hand,
$$
     \dim L^2\cH^k(\phi^{-1}(m),g_m)= \sum_{p+q=k} h^{p,q}.
$$
But by the result of \cite{MelroseAPS}, see also \cite{HHM2004}, the number $\dim L^2\cH^{k}(\phi^{-1}(M),g_m)$ depends only on the topology of $\widetilde{Z}=[\bZ; \bD]$, the blow-up of $\bZ$ at $\bD$ in the sense of \cite{MelroseAPS}, so it is independent of $m$.  This
means that none of the individual Hodge numbers $h^{p,q}$ can decrease, so $h^{p,q}(m)= h^{p,q}(m_0)$ for all $m\in \cU$.
\end{proof}

The family $\eth$ thus has a well-defined determinant line bundle.  We can use Theorem~\ref{qc.8} to compute its curvature.
\begin{theorem}
The curvature of the determinant line bundle of the family of Dolbeault operators $\eth=\sqrt{2}(\db+\db^*)$ associated to the bundle \eqref{wp.10} is
$$
        \frac{i}{2\pi} (\nabla^Q)^2= \frac{\chi(Z)}{12 \pi} \omega_{WP}.
$$
\label{wp.11}\end{theorem}
\begin{proof}
By Theorem~\ref{qc.8},
\begin{equation}
(\nabla^Q)^2= \left[ \frac{1}{(2\pi i)^{n}} \int_{N/B} \Td(T^{1,0}(\mathfrak{X}/\cM_{\rel}),g_b)\Ch(E)  \right]_{[2]}, \quad \mbox{where} \; n=\dim_{\bbC} Z.
\label{wp.12}\end{equation}

On the other hand, by \cite[p.374]{BCOV1994},
\begin{equation}
 \Td(T^{1,0}(\mathfrak{X}/\cM_{\rel})\Ch(E))= -(2\pi i)^{n-1}c_{n-1}+ (2\pi i)^{n}\frac{n}2 c_n -\frac{(2\pi i)^{n+1}}{12}c_1 c_n + \mbox{(higher order terms)},
\label{wp.13}\end{equation}
where $c_i= c_i(T^{1,0}(\mathfrak{X}/\cM_{\rel}))$ are the Chern forms of the Hermitian bundle $T^{1,0}(\mathfrak{X}/\cM_{\rel})$ and the $2\pi i$ factors appear because we follow the convention of \cite{BGV} for the definitions of the Todd form and the Chern character form.
On the other hand, by Proposition~\ref{wp.8},
\begin{equation}
  c_1(T^{1,0}(\mathfrak{X}/\cM_{\rel}))= \frac{\phi^*\omega_{\WP}}{\pi}.
\label{wp.14}\end{equation}
Hence, combining \eqref{wp.12}, \eqref{wp.13} and \eqref{wp.14}, the result follows from the Chern-Gauss-Bonnet theorem for manifolds with cylindrical ends (see \cite{MelroseAPS}),
$$
    \int_{Z_m} c_n= \chi(Z_m)= \chi(Z).
$$

\end{proof}

\bibliography{PolyhomCY}
\bibliographystyle{amsplain}

\end{document}